\documentclass{article}
\usepackage{fullpage,diagbox}
\usepackage{mathtools, appendix}

\mathtoolsset{showonlyrefs}

\newcommand{\E}{\mbox{$\mathbb{E}$}}
\newcommand{\clr}{\color{black}}
\newcommand{\hulc}{\textsc{HulC}~}
\newcommand{\hulcnospace}{\textsc{HulC}}
\newcommand{\ahulc}{\textsc{Adaptive HulC}~}
\newcommand{\ahulcnospace}{\textsc{Adaptive HulC}}
\newcommand{\uhulc}{\textsc{Unimodal HulC}~}
\newcommand{\uhulcnospace}{\textsc{Unimodal HulC}}

\usepackage[ruled,linesnumbered, vlined, noend]{algorithm2e}

\usepackage{amsthm, amsfonts, amssymb, amsmath,enumitem, bbm, mathabx}
\usepackage{natbib}
\newtheorem{thm}{Theorem}
\newtheorem{lem}{Lemma}

\newtheorem{prop}{Proposition}

\newcounter{myalgctr}
\newenvironment{rem}{%      define a custom environment
   \vskip1mm\indent%         create a vertical offset to previous material
   \refstepcounter{myalgctr}% increment the environment's counter
   \textbf{Remark \themyalgctr}% or \textbf, \textit, ...
   }{\hfill$\diamond$\par}  %          create a vertical offset to following material
\RequirePackage[colorlinks=true, pdfstartview=FitV,
linkcolor=blue, citecolor=blue, urlcolor=blue,hypertexnames=false]{hyperref}
\numberwithin{myalgctr}{section}
\DeclareMathOperator*{\argmin}{arg\,min}

\DeclareRobustCommand{\rchi}{{\mathpalette\irchi\relax}}
\newcommand{\irchi}[2]{\raisebox{\depth}{$#1\chi$}}
\usepackage{graphicx} % support the \includegraphics command and options
\usepackage{color, float,authblk}
\title{\bf The \hulcnospace: Confidence Regions from Convex Hulls}
\author{Arun Kumar Kuchibhotla}
\author{Sivaraman Balakrishnan}
\author{Larry Wasserman}
\affil{\texttt{\{arunku,\,siva,\,larry\}@stat.cmu.edu}}
\affil{Department of Statistics \& Data Science, Carnegie Mellon University,\\ 5000 Forbes Ave, Pittsburgh, PA 15213, USA.}
\date{}
\begin{document}

\maketitle

\begin{abstract}
We develop and analyze the \hulcnospace, an intuitive and general method for constructing confidence sets using the convex hull 
of estimates constructed from subsets of the data. 
{\color{black}We present this method in the context of independent data.}
Unlike classical methods which are based on estimating the (limiting) distribution of an estimator, the \hulc is often 
simpler to use and effectively bypasses this step.
In comparison to the bootstrap, the \hulc requires fewer regularity conditions and succeeds in many examples where the bootstrap 
provably fails. Unlike subsampling, the \hulc does not require knowledge of the rate of convergence of the estimators on which it is based. 
The validity of the \hulc requires knowledge of the (asymptotic) median-bias of the estimators. We further analyze 
a variant of our basic method, 
called the \ahulcnospace, which is fully data-driven and estimates the median-bias using subsampling. We show that the \ahulc retains the aforementioned
strengths of the \hulcnospace. In certain cases where the underlying estimators are pathologically asymmetric the \hulc and \ahulc can fail to provide useful confidence sets. 
We propose a final variant, the \uhulcnospace, which can salvage the situation in cases where the distribution of the underlying estimator is (asymptotically) unimodal. 
We discuss these methods in the context of several challenging inferential problems which arise
in parametric, semi-parametric, and non-parametric inference. Although our focus is on validity under weak regularity conditions, we
also provide some general results on the width of the \hulc confidence sets, showing that in many
cases the \hulc confidence sets have near-optimal~width.
\end{abstract}

\section{Introduction}\label{sec:intro}
Estimation and uncertainty quantification are two of the most fundamental aspects of statistical analysis. The theory of point estimation is very well-studied starting from the principle of maximum likelihood estimation~\citep{stigler2007epic,pfanzagl2011parametric,lehmann2006theory}. Relatively more recent frameworks of parametric efficiency~\citep[Chapters 4--8]{van2000asymptotic} and semiparametric influence functions~\citep{bickel1993efficient} provide general methods of constructing good estimators. Uncertainty quantification, for instance{\color{black}, when} testing a statistical hypothesis or constructing a confidence set, most often follows from studying the asymptotic distribution of the estimator. In many cases this approach requires \emph{estimating} the asymptotic distribution of the estimator {\color{black}(properly normalized)}. Even in favorable cases, when this asymptotic distribution is mean zero Gaussian, one needs to further estimate the asymptotic variance of the estimator
in order to construct a valid confidence set. As a consequence, in practice, methods which yield uncertainty quantification while using only a method for point estimation are often favored. 

Generic techniques to obtain uncertainty quantification that do not require any more than the estimation method are the bootstrap and subsampling~\citep{efron1979bootstrap,politis1994large,shao2012jackknife,hall2013bootstrap}. 
The bootstrap however requires that the estimator 
be Hadamard differentiable;
see~\cite{dumbgen1993nondifferentiable} and~\citet[Section 3.2]{shao2012jackknife}. 
Subsampling is more general,
but requires knowing the rate of convergence of the
estimator. \cite{bertail1999subsampling} provides a scheme to estimate
the unknown rate of convergence, but this method 
% is hard to implement and 
cannot estimate the
slowly varying components of the rate (such as $\log n$ factors);
see~\citet[page 3]{sherman1997omnibus} for details.

In this paper, we propose a new method,
the \hulc (Hull based Confidence)
that does not
require variance estimation and is applicable in many examples where
the bootstrap and subsampling are not. 
The \hulc does not
require knowing the rate of convergence of the estimator. 
In many cases, the \hulc does not involve any tuning
parameters. 
Besides being asymptotically valid,
the \hulc
is {\em eventually finite sample} 
% (EFS)
meaning that the coverage is exact
for all samples of size $n\geq n_0$
for some finite $n_0$. {\color{black}Throughout the paper, we restrict ourselves to independent data.
Although applicable for dependent data, \hulc for dependent data is beyond the scope of the current paper and will be dealt with elsewhere.}

The basis for the \hulc is an assumption that
the estimators on which it is based are not
pathologically asymmetric: their distributions do not place all their
mass to one side of the target parameter.
We measure the asymmetry in terms of the median bias of
the estimator. This makes the method widely
applicable and easy to use. 
Our method has some similarity to the typical values
approach of~\cite{hartigan1969using,hartigan1970exact}. See, in
particular, point 5 in Section 7 of~\cite{hartigan1970exact}.
{\color{black}The work of~\citet{Ibragimov2010} also uses estimators computed on a fixed number of splits of the data as we do and combines them via a $t$-statistic to obtain an asymptotically valid confidence interval; also, see~\cite{lam2022cheap}. We note that, unlike ours, this approach relies heavily on the asymptotic normality of a properly normalized estimate.}

Mean unbiasedness is a popular criterion for good estimators and mean bias reduction is well-studied in the statistics literature~\citep{firth1993bias,kosmidis2009bias,kim2016higher}. However, as noted in~\cite{pfanzagl2017optimality} the fact that an estimator is mean unbiased does not naturally aid in uncertainty quantification.
In contrast, median unbiasedness implies that the estimator is equally likely to underestimate and overestimate the target of interest. As will be shown in this article, this property can lead to a simple method for constructing confidence intervals. Median unbiasedness and median bias reduction are not as widely known as the mean unbiasedness and mean bias reduction, but we will develop their implications for inference. We refer the reader to~\cite{pfanzagl2011parametric} for details regarding median unbiased estimation and to~\cite{kenne2017median,kosmidis2020mean} for median bias reduction methods {\color{black}in parametric models}. 

Inspired by the practical success of resampling methods like the bootstrap and subsampling, the \hulc directly exploits our relatively strong understanding of point estimation to address challenging inferential problems.
As with these methods, the width of the intervals we construct is naturally related to the accuracy of the underlying estimators, i.e. the \hulc based on a very accurate estimator will lead to small confidence sets. On the other hand, in contrast to 
these methods, the \hulc uses sample-splitting to avoid strong regularity conditions, and its validity relies instead on a relatively mild assumption. This follows a line of recent work by the authors
(for instance, \cite{wasserman2020universal,chakravarti2019gaussian,rinaldo2019bootstrapping}), and more classical work by~\cite{bickel1982adaptive}, where sample-splitting eases the challenges of statistical inference, often at a surprisingly small price.

The remainder of this article is organized as follows. In Section~\ref{sec:convex-hull}, we describe our assumptions and the \hulc method for constructing confidence regions for univariate and multivariate parameters. We compare the proposed confidence interval to Wald confidence intervals based on asymptotic Normality in terms of their widths. We also compare to the bootstrap and subsampling in terms of applicability. In Section~\ref{sec:applications-parametric-models}, we discuss the applicability of the \hulc to some standard examples where limiting distributions are well-understood but constructing valid confidence sets can still be challenging; the examples we consider include mean and median estimation, Binomial proportion estimation, and parameter estimation in exponential families. Our method involves an assumption on the median bias of the estimators under consideration. In Section~\ref{sec:adaptive-convex-hull}, we describe the \ahulc which estimates the median bias using subsampling. Interestingly, in contrast to directly using subsampling for constructing a confidence set, the \ahulc does not require 
knowledge of the rate of convergence. In Section~\ref{sec:applications-nonparametric-models}, we provide some applications of the \ahulc to nonparametric models including shape constrained regression. In Section~\ref{sec:unimodality-lanke}, we provide an extension, called the \uhulcnospace, based on the assumption of unimodality. Between our {\color{black}median bias} assumption and unimodality assumption, we believe that many challenging confidence set construction problems based on independent observations are solved. {\color{black}In Section~\ref{subsec:multivariate}, we briefly discuss the application of \hulc for multivariate parameters/functionals.} Finally, in Section~\ref{sec:conclusions}, we summarize the article and discuss some future directions. 
Throughout the article, we focus on the pointwise validity (as in~\cite{politis1994large}) of our confidence region{\color{black}, where we treat the distribution of the data as fixed, as the sample size increases.
%Uniform validity will be discussed elsewhere. 
Some preliminary results on uniform validity of \hulc are presented in~\cite{kuchibhotlamedian}. 
}

The proofs of all the main results are provided in the supplementary material. {\color{black}Section~\ref{subsec:union-bound-Wald} provides a discussion on the application of Bonferroni inequality with Wald interval, supplementing HulC method for multivariate parameters in Section~\ref{subsec:multivariate}.} The sections and equations of the supplementary file are prefixed with ``S.'' and ``E.'', respectively, for convenience. We provide the code to reproduce the figures in the paper, including an implementation of our methods in R together with Jupyter notebooks illustrating their application at \url{https://github.com/Arun-Kuchibhotla/HulC}.
\section{The \hulcnospace: Hull based Confidence Regions}\label{sec:convex-hull}
In this section, we describe the \hulc and compare it to classical asymptotic normality based confidence intervals. 
We present several results for the \hulcnospace, and in order to aid readability we provide a brief roadmap here:
\begin{enumerate}
\item Focusing first on univariate parameters, in Theorem~\ref{thm:coverage-algorithm-known-Delta}, we show that when the median bias of the estimators is known to be at most $\Delta$ the \hulc (as described in Algorithm~\ref{alg:confidence-known-Delta}) has guaranteed coverage 
of at least $1 - \alpha$. We also show that, under some mild additional conditions, if the underlying estimators have median bias exactly $\Delta$ then the \hulc has coverage exactly $1 - \alpha$. 
\item In Proposition~\ref{prop:approximate-to-exact-validity}, we investigate properties of a (slightly) conservative variant of the \hulcnospace, showing that the \hulc when provided with the \emph{asymptotic} median bias still ensures finite-sample $1-\alpha$ coverage, for sufficiently large sample sizes.
This setting is practically useful because in many cases we know the limiting distribution of our estimates is normal (say) and in these cases 
the asymptotic median bias is known to be 0.

\item In Theorem~\ref{thm:infinite-order-coverage} and Remark~\ref{rem:miscoverage-Delta}, we show that the guarantees of the (non-conservative) \hulc erode gracefully, i.e. if we run the \hulc with a parameter $\Delta$ but the true median bias is at most $\widetilde{\Delta}$
then the \hulc has coverage which degrades from the nominal level (multiplicatively) as a function of $|\Delta - \widetilde{\Delta}|$.

\item In~\eqref{eq:width-convex-hull} and~\eqref{eq:width-simple-calculation}, we provide two simple analyses of the width of the \hulc intervals. In~\eqref{eq:width-convex-hull} we show that in the classical setting where the estimates have an asymptotic normal distribution, 
the width of the \hulc interval is the same as that of the corresponding Wald interval up to a factor of $\sqrt{\log_2(\log_2(2/\alpha))}$. In~\eqref{eq:width-simple-calculation}, we show that under much more generality the \hulc
based on $B^*$ splits yields a variance-sensitive confidence interval whose expected width is upper bounded by $2 \sigma B^*/\sqrt{n}$, where $\sigma$ is the standard deviation of the estimators on which the \hulc is based.

% \item Turning our attention to multivariate parameters, in Lemma~\ref{lem:fixed-B-coverage-result-multivariate}, we analyze the \hulc based on the convex hull of the underlying estimates 
% and on the rectangular hull of the underlying estimates, providing coverage guarantees as a function of the median bias. We further compare the multivariate \hulc intervals to those based on multivariate CLTs, highlighting
% several advantages of the former.
\end{enumerate}

\subsection{\hulc for univariate parameters}
Suppose $\theta_0\in\mathbb{R}$ is a parameter or functional of interest. Let $X_1, \ldots, X_n$ be independent random variables from some measurable space $\mathcal{X}$. For $B\ge1$, let $\widehat{\theta}_1, \ldots, \widehat{\theta}_B$ be independent estimators of $\theta_0$. These can be obtained by splitting the data $X_1, \ldots, X_n$ into $B$ batches and computing an estimate from each batch.
{\color{black}Formally, let $S_1, \ldots, S_B$ be a (random) partition of $\{1, 2, \ldots, n\}$ into $B$ subsets. For $1\le j\le B$, let $\widehat{\theta}_j$ be the estimator computed on $j$-th batch of observations $\{X_k:\, k\in S_j\}$.}
Define the median bias of the estimator $\widehat{\theta}_j$ for $\theta_0$ as
\begin{equation}\label{definition:median-bias}
\mbox{Med-Bias}_{\theta_0}(\widehat{\theta}_j) ~:=~ \left(\frac{1}{2} - \min\left\{\mathbb{P}(\widehat{\theta}_j - \theta_0 \ge 0), \mathbb{P}(\widehat{\theta}_j - \theta_0 \le 0)\right\}\right)_+,
\end{equation}
where $(x)_+ = \max\{x, 0\}$ for any $x\in\mathbb{R}$. {\color{black}From the definition, it is clear that $\mbox{Med-Bias}_{\theta_0}(\widehat{\theta}_j)$ lies in $[0, 1/2]$.}
Using the independence of the estimators $\widehat{\theta}_j, 1\le j\le B$, we obtain the following result (proved in Section~\ref{appsec:fixed-B-coverage-result} of the supplementary material). 
\begin{lem}\label{lem:fixed-B-coverage-result}
If $\widehat{\theta}_j, 1\le j\le B$ are independent random variables and
\begin{equation}\label{eq:delta-definition-main-lemma}
\Delta ~:=~ \max_{1\le j\le B}\,\mathrm{Med\mbox{-}Bias}_{\theta_0}(\widehat{\theta}_j){\color{black}~\in~[0, 1/2]},
\end{equation}
then
\[
\mathbb{P}\left(\theta_0 \notin \left[\min_{1\le j\le B}\widehat{\theta}_j,\,\max_{1\le j\le B}\widehat{\theta}_j\right]\right) \le \left(\frac{1}{2} - \Delta\right)^B + \left(\frac{1}{2} + \Delta\right)^B.
\]
\end{lem}
{\color{black}Observe that $\Delta$ in general depends on $\theta_0$, the distribution of the data $X_1, \ldots, X_n$ as well as the batch sizes $|S_1|$, \ldots, $|S_B|$. For simplicity, we do not index $\Delta$ with these quantities. Lemma~\ref{lem:fixed-B-coverage-result} provides a two-sided confidence interval with a bound on the miscoverage probability. It is easy to also obtain one-sided confidence intervals with explicit bounds on miscoverage. For instance, if there exists a $\delta\in[-1/2, 1/2]$ such that $\mathbb{P}(\widehat{\theta}_j \ge \theta_0) \ge 1/2 + \delta$ for all $1\le j\le B$, then $\mathbb{P}(\theta_0 \le \max_{1\le j\le B}\widehat{\theta}_j) \ge 1 - (1/2 - \delta)^B$.}

An estimator $\widehat{\theta}$ is said to {\bf median unbiased} for $\theta_0$ if
$\mbox{Med-Bias}_{\theta_0}(\widehat{\theta}) = 0$
\citep{pfanzagl2011parametric}.
It is worth
noting that median unbiasedness does not imply that the estimator is
symmetric. 
The non-strict inequality
in the definition~\eqref{definition:median-bias} is important:
it allows for
$\mathbb{P}(\widehat{\theta}_j - \theta_0 \ge 0)$ and
$\mathbb{P}(\widehat{\theta}_j - \theta_0 \le 0)$ to be equal to $1$
or be larger than $1/2$. This is useful in cases where 
$\theta_0$ is on the boundary or $\widehat{\theta}_j$ has a discrete
distribution and puts non-zero mass at $\theta_0$. 
An estimator $\widehat\theta_n$ based on $n$ observations 
is {\bf asymptotically median unbiased} if
$\lim_{n\to\infty}\mbox{Med-Bias}_{\theta_0}(\widehat\theta_n)\to 0$.
{\color{black}One of the prominent examples of asymptotically median unbiased estimators is the class of asymptotically normal estimators; some of these are described in Section~\ref{sec:applications-parametric-models}. A simple example with non-zero asymptotic median bias is the estimation of $\theta_0 = \mu_0^2$, where $\mu_0$ is the mean of i.i.d. random variables $X_1, \ldots, X_n$. The estimator $\widehat{\theta} = \binom{n}{2}^{-1}\sum_{i\neq j}X_iX_j$ has an asymptotic median bias of $|\mathbb{P}(\rchi^2_1 \le 1) - 1/2|$ when $\mu_0 = 0$ as shown in Section~\ref{subsec:mean-square}. (Here $\rchi^2_1$ denotes the chi-square random variable with degrees of freedom 1.)}

For any $B\ge1$ and $\Delta \ge 0$, set the {\color{black}upper bound on the} miscoverage probability from Lemma~\ref{lem:fixed-B-coverage-result} as
\begin{equation}\label{eq:miscoverage-function-B-Delta}
P(B; \Delta) ~:=~ \left(\frac{1}{2} - \Delta\right)^B + \left(\frac{1}{2} + \Delta\right)^B.
\end{equation} 
If $\Delta \ge 0$ is known, then choosing $B := B_{\alpha,\Delta}\ge1$ such that $P(B; \Delta) \le \alpha$, we conclude that
\[
\mathbb{P}\left(\theta_0 \notin \left[\min_{1\le j\le B}\widehat{\theta}_j,\,\max_{1\le j\le B}\widehat{\theta}_j\right]\right) \le \alpha.
\]
In words, the smallest rectangle (interval) containing $B_{\alpha,\Delta}$ independent estimators of $\theta_0$ has a coverage of at least $1 - \alpha$\footnote{{\color{black}Throughout this paper we use the phrase ``asymptotically valid'' (or ``valid'') to indicate that the coverage is asymptotically (or finite-sample) \emph{at least} $1 - \alpha$. When the coverage is exactly $1 - \alpha$ we indicate this by the phrase ``asymptotically exact'' (or simply ``exact'' in the finite-sample setting).}}. 
{\color{black}The smallest interval containing $B_{\alpha,\Delta}$ estimators is their (convex) hull and hence we call this interval \hulc (Hull based Confidence) interval.}
Because $B$ is an
integer, $P(B; \Delta)$ decreases in steps as $B$ changes over
positive integers and this can lead to conservative coverage {\color{black}i.e., miscoverage probability strictly less than $\alpha$ as there may not exist an integer $B$ such that $P(B; \Delta) = \alpha$}.  
This issue can be resolved easily by randomizing the choice of $B$. 
{\color{black}
Formally, we generate a random variable $U$ from the uniform distribution on $[0,1]$ and set
\begin{equation}\label{eq:definition-tau}
\tau_{\alpha,\Delta} := \frac{\alpha - P(B_{\alpha,\Delta}; \Delta)}{P(B_{\alpha,\Delta} - 1; \Delta) - P(B_{\alpha,\Delta};\Delta)}\quad\mbox{and}\quad B^* := \begin{cases} B_{\alpha,\Delta} - 1, &\mbox{if }U \le \tau_{\alpha,\Delta},\\
B_{\alpha,\Delta}, &\mbox{if }U > \tau_{\alpha,\Delta}.
\end{cases} 
\end{equation}
If $\tau_{\alpha,\Delta} = 0$, then $B^* = B_{\alpha,\Delta}$.
}
Most often $\Delta$ is unknown.  
This issue will be resolved 
in Section~\ref{sec:adaptive-convex-hull}
where we show how to estimate $\Delta$.

Algorithm~\ref{alg:confidence-known-Delta} gives the steps to find a randomized confidence interval with $1 - \alpha$ coverage when the median bias $\Delta$ is known.
\begin{algorithm}[h]
    \caption{Confidence Interval with Known $\Delta$ (\hulcnospace)}
    \label{alg:confidence-known-Delta}
    \SetAlgoLined
    \SetEndCharOfAlgoLine{}
    \KwIn{data $X_1, \ldots, X_n$, coverage probability $1 - \alpha$, a value $\Delta$, and an estimation procedure $\mathcal{A}(\cdot)$ that takes as input observations and returns an estimator with a median bias of at most $\Delta$.}
    \KwOut{A confidence interval $\widehat{\mathrm{CI}}_{\alpha,\Delta}$ such that $\mathbb{P}(\theta_0\in\widehat{\mathrm{CI}}_{\alpha,\Delta}) \ge 1 - \alpha$.}
    Find the smallest integer $B = B_{\alpha,\Delta}\ge1$ such that $P(B;\Delta) \le \alpha$. Recall $P(B; \Delta)$ from~\eqref{eq:miscoverage-function-B-Delta}.\;
    % \hspace{0.1in} 
    Fix $B^*$ as defined in~\eqref{eq:definition-tau}.\;
    % \hspace{0.1in} 
	Randomly split the data $X_1, \ldots, X_n$ into $B^*$ disjoint sets $\{\{X_i:i\in S_j\}: 1\le j\le B^*\}.$ These need not be equal sized sets, but having approximately equal sizes yields good width properties.\;
	Compute estimators $\widehat{\theta}_j := \mathcal{A}(\{X_i:\,i\in S_j\})$, for $1\le j\le B^*$\;
    \Return the confidence interval 
    \begin{equation}\label{eq:confidence-interval-known-Delta}
    \widehat{\mathrm{CI}}_{\alpha, \Delta} ~:=~ \left[\min_{1\le j\le B^*}\widehat{\theta}_j, \max_{1\le j\le B^*}\widehat{\theta}_j\right].
    \end{equation}
\end{algorithm}

There are no restrictions on the input $\mathcal{A}(\cdot)$ in Algorithm~\ref{alg:confidence-known-Delta} except that it produces an estimate with median bias bounded by $\Delta$. Its rate of convergence and variance play a role only in the width properties of the resulting confidence interval, not in the validity guarantee. A better estimator will lead to a smaller confidence interval. Here are two examples of the estimation procedure $\mathcal{A}(\cdot)$:
\begin{itemize}[leftmargin=0.2in]
	\item If $X_1, \ldots, X_n$ are identically distributed and $\theta_0 = \mathbb{E}[X_1]$, then one can take
	% \[
	$\widehat{\theta}_j = \mathcal{A}(\{X_i:\,i\in S_j\}) = {|S_j|^{-1}}\sum_{i\in S_j}X_i.$
	% \]
	In general, the median bias of the sample mean is unknown, but typically tends to zero as $|S_j|\to\infty$.
	If the observations are symmetrically distributed around $\theta_0$, then $\mbox{Med-Bias}_{\theta_0}(\widehat{\theta}_j) = 0.$
	\item If $X_1, \ldots, X_n$ are random variables generated from a parametric model $p_{\theta_0}$ that belongs to the parametric family $\{p_{\theta}:\,\theta\in\Theta\}$, then one can take
	% \[
	$\widehat{\theta}_j = \mathcal{A}(\{X_i:\,i\in S_j\})$
	 % := \argmax_{\theta\in\Theta} \prod_{i\in S_j}p_{\theta}(X_i),
	% \]
	as the maximum likelihood estimator (MLE) of $\theta_0$ based on the observations $X_i, i\in S_j$. 
	Under standard regularity conditions, MLE has an asymptotic normal distribution and hence the median bias of $\widehat{\theta}_j$ converges to zero. 
\end{itemize} 
The following result (proved in Section~\ref{appsec:coverage-algorithm-known-Delta} of the supplementary material) establishes that the confidence interval from Algorithm~\ref{alg:confidence-known-Delta} has a coverage of at least $1 - \alpha$.
\begin{thm}\label{thm:coverage-algorithm-known-Delta}
If $X_1, \ldots, X_n$ are independent random variables and the estimation procedure $\mathcal{A}(\cdot)$ in Algorithm~\ref{alg:confidence-known-Delta} returns estimates that have a median bias of at most $\Delta$, then the confidence interval $\widehat{\mathrm{CI}}_{\alpha,\Delta}$ returned by Algorithm~\ref{alg:confidence-known-Delta} satisfies
\begin{equation}\label{eq:over-coverage-in-general}
\mathbb{P}\left(\theta_0 \in \widehat{\mathrm{CI}}_{\alpha,\Delta}\right) \ge 1 - \alpha.
\end{equation}
Further, if $\mathbb{P}(\widehat{\theta}_j \le \theta_0) = \mathbb{P}(\widehat{\theta}_1 \le \theta_0)$, $\mathbb{P}(\widehat{\theta}_j = \theta_0) = 0$ for all $j$, and the estimation procedure $\mathcal{A}(\cdot)$ in Algorithm~\ref{alg:confidence-known-Delta} returns estimates that have a median bias of exactly $\Delta$, then
\begin{equation}\label{eq:exact-coverage-same-median-bias}
\mathbb{P}\left(\theta_0 \in \widehat{\mathrm{CI}}_{\alpha,\Delta}\right) = 1 - \alpha.
\end{equation}
\end{thm}
In Algorithm~\ref{alg:confidence-known-Delta}, it is implicitly assumed that $B^*$ defined in step 2 is smaller than the sample size $n$ so that the estimation procedure can be applied on $B^*$ splits of the data. Recall $P(B; \Delta)$ from~\eqref{eq:miscoverage-function-B-Delta} and that $B_{\alpha,\Delta}$ is the smallest integer such that $P(B; \Delta) \le \alpha.$ It is easy to prove that $P(B; \Delta)$ is an increasing function of $\Delta\in[0, 1/2]$ and hence we obtain that 
$\max\{(1/2 + \Delta)^{B},2^{-B + 1} \}\le P(B; \Delta) \le 2(1/2 + \Delta)^{B}$.
Therefore, $B_{\alpha,\Delta}$ satisfies
\begin{equation}\label{eq:inequalities-B-alpha-Delta}
\max\left\{\left\lceil\frac{\log(1/\alpha)}{\log(2/(1 + 2\Delta))}\right\rceil, \left\lceil\frac{\log(2/\alpha)}{\log(2)}\right\rceil \right\}~\le~ B_{\alpha,\Delta} ~\le~ \left\lceil\frac{\log(2/\alpha)}{\log(2/(1 + 2\Delta))}\right\rceil.
\end{equation}
Here $\lceil x\rceil$, for any real $x$, denotes the smallest integer larger than $x$.
It is easy to verify that,
$B_{\alpha,\Delta}\to\infty$ as $\Delta\to0.5$. 
Figure~\ref{fig:B-alpha-delta} shows the plot of $B_{\alpha,\Delta}$ as $\Delta$ varies from $0$ to $0.4$ and $\alpha$ varies between $0.05, 0.1, 0.15$. {\color{black}The right panel of Figure~\ref{fig:B-alpha-delta} shows the values of $B_{\alpha,\Delta}$ for some choices of $\alpha$ and $\Delta$. In the favourable case of $\Delta = 0$, and for the usual choices of $\alpha = 0.05, 0.1$, the number of independent splits required is about 5, a feasible choice for any reasonable sample size.}
\begin{figure}[!h]
  \begin{minipage}{0.70\linewidth}
    \centering
\includegraphics[width=\textwidth,height=2.7in]{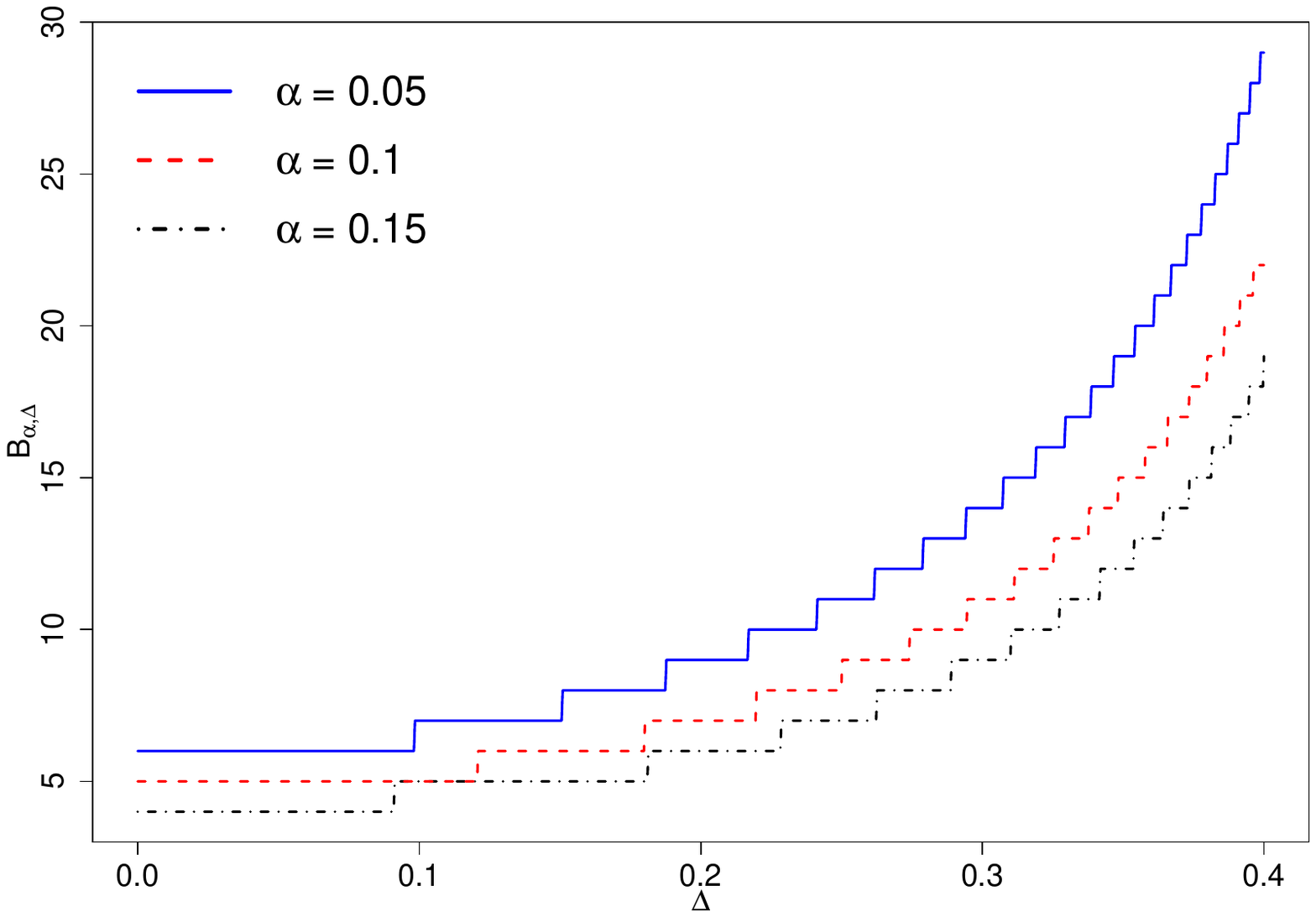}      \par\vspace{0pt}

  \end{minipage}%
  \begin{minipage}{0.30\linewidth}
    \centering
    % \begin{table}[ht]
% \centering
\begin{tabular}{r|rrr}
  \hline
\diagbox[height=2em]{$\Delta$}{$\alpha$} & 0.15 & 0.1 & 0.05 \\ 
  \hline\hline
0 & 4 & 5 & 6 \\ 
  0.05 & 4 & 5 & 6 \\ 
  0.1 & 5 & 5 & 7 \\ 
  0.15 & 5 & 6 & 7 \\ 
  0.2 & 6 & 7 & 9 \\ 
  0.25 & 7 & 9 & 11 \\ 
  0.3 & 9 & 11 & 14 \\ 
  0.35 & 12 & 15 & 19 \\ 
  0.4 & 19 & 22 & 29 \\ 
   \hline
\end{tabular}

    \par\vspace{0pt}
\end{minipage}
\caption{Some example values of $B_{\alpha,\Delta}$ for different values of $\alpha\in\{0.05, 0.1, 0.15\}$ and $\Delta\in[0, 0.4]$. Left panel: the plot as $\Delta$ changes continuously. Right panel: values of $B_{\alpha,\Delta}$ as $\Delta$ changes from $0.0$ to $0.4$ in increments of $0.05$. }%}
\label{fig:B-alpha-delta}
\end{figure}
\subsection[Confidence regions with known asymptotic median bias]{\hulc when asymptotic median bias is known}
Algorithm~\ref{alg:confidence-known-Delta} requires knowledge of the median bias of the estimators. In some settings, estimation procedures can be constructed so as to ensure median unbiasedness (i.e., $\Delta = 0$). When $\Delta = 0$, $B_{\alpha,0} = \lceil\log_2(2/\alpha)\rceil$. These examples are discussed in Section~\ref{sec:applications-parametric-models}. 

Because $B_{\alpha,\Delta}$ is a piecewise constant function in $\Delta$, we do not need to know $\Delta$ exactly. 
This observation implies that for estimators that are asymptotically symmetric around $\theta_0$, one can take $\Delta$ to be zero in Algorithm~\ref{alg:confidence-known-Delta} and still retain (asymptotic) validity. Formally, if $\mbox{Med-Bias}_{\theta_0}(\widehat{\theta}_j) \to 0$ as $|S_j|\to\infty$, then the convex hull of $B_{\alpha,0} = \lceil\log_2(2/\alpha)\rceil$ estimators has an asymptotic coverage of at least $1 - \alpha$. Furthermore, the convex hull is \emph{eventually finite sample} valid, meaning that there is a sample size $n_0$ such that the coverage is at least $1 - \alpha$ for all $n\ge n_0$. Now, we provide more details. 

Proposition~\ref{prop:approximate-to-exact-validity} proved in Section~\ref{appsec:approximate-to-exact-validity} of the supplementary material formally establishes that $B_{\alpha,\Delta}$ is a piecewise constant function of $\Delta$ (as illustrated in 
Figure~\ref{fig:B-alpha-delta}).
\begin{prop}\label{prop:approximate-to-exact-validity}
For $\widetilde{\Delta}, \Delta\in[0, 1/2)$ and $\alpha\in(0, 1)$, if
\begin{equation}\label{eq:consistency-subsampling}
2B_{\alpha,\Delta}|\Delta - \widetilde{\Delta}| ~\le~ B_{\alpha,\Delta}\left[\min\left\{\left(\frac{\alpha}{P(B_{\alpha,\Delta};\Delta)}\right)^{1/B_{\alpha,\Delta}},\,\left(\frac{P(B_{\alpha,\Delta} - 1; \Delta)}{\alpha}\right)^{1/B_{\alpha,\Delta}}\right\} - 1\right],
\end{equation}
then $B_{\alpha, \widetilde{\Delta}} = B_{\alpha, \Delta}.$ Moreover, if
\begin{equation}\label{eq:consistency-subsampling-asymptotic-median-unbiased}
2B_{\alpha,0}(B_{\alpha,0} - 1)\widetilde{\Delta}^2(1 + 2\widetilde{\Delta})^{(B_{\alpha,0} - 2)_+}
% 2B_{\alpha,0}(B_{\alpha,0} - 1)\widetilde{\Delta}^2e^{2(B_{\alpha,0} - 2)_+\widetilde{\Delta}} 
~\le~ \frac{\alpha}{P(B_{\alpha,0}; 0)} - 1.
\end{equation}
then $B_{\alpha,\widetilde{\Delta}} = B_{\alpha,0}$.
\end{prop}
\begin{rem}
Note that the right hand side of~\eqref{eq:consistency-subsampling} is non-zero if and only if $\tau_{\alpha,\Delta} \neq 0$ in~\eqref{eq:definition-tau}. In a typical application, one would take $\Delta$ as the hypothesized (or asymptotic) value of the median bias and $\widetilde{\Delta}$ is the true median bias. Hence, the right hand side of~\eqref{eq:consistency-subsampling} can be computed exactly for any user choice of $\alpha\in(0, 1)$. 
As a practical matter, the user can change $\alpha$ by a tiny amount to increase the right hand side of~\eqref{eq:consistency-subsampling}.
 % shows the behavior of the right hand side of~\eqref{eq:consistency-subsampling}. 
In the most common setting of asymptotic normality, $\Delta = 0$, and consequently the requirement becomes more relaxed as in~\eqref{eq:consistency-subsampling-asymptotic-median-unbiased}; this relaxation stems from the fact that $\Delta\mapsto P(B; \Delta)$ has zero first derivative at $\Delta = 0$. 
{\color{black}Using the definition of Lambert function, the fact that $(1 + 2\widetilde{\Delta})^{(B - 2)_+} \le e^{2(B - 2)_+\widetilde{\Delta}}$, and that $B_{\alpha,0} - 2 \le B_{\alpha,0} - 1 \le B_{\alpha,0}$, the requirement~\eqref{eq:consistency-subsampling-asymptotic-median-unbiased} can be shown to be implied by
\begin{equation}\label{eq:relaxed-requirement-median-unbiased}
B_{\alpha,0}\widetilde{\Delta} ~\le~ W_0\left(\frac{1}{\sqrt{2}}\sqrt{\frac{\alpha}{P(B_{\alpha,0}; 0)} - 1}\right),
\end{equation}
where $W_0(\cdot)$ represents the principal branch of the Lambert function.}
 % The right hand side of~\eqref{eq:consistency-subsampling-asymptotic-median-unbiased} is shown in 
{\color{black}Figures~\ref{fig:right-hand-side-of-consistency-subsampling-scaled} and~\ref{fig:right-hand-side-of-consistency-subsampling-median-unbiased} show the behavior of the right hand sides of~\eqref{eq:consistency-subsampling} and~\eqref{eq:relaxed-requirement-median-unbiased} for various values of $\alpha$ and $\Delta$.}
\end{rem}

Recall from the calculation surrounding~\eqref{eq:miscoverage-function-B-Delta} that the smallest interval containing $B_{\alpha,\widetilde{\Delta}}$ estimators has a coverage of at least $1 - \alpha,$ if the estimators have a median bias of at most $\widetilde{\Delta}$. Proposition~\ref{prop:approximate-to-exact-validity} implies that one need not know the median bias $\widetilde{\Delta}$ of the estimators exactly in order to find $B_{\alpha,\widetilde{\Delta}}$. Suppose the estimators $\widehat{\theta}_j, j\ge1$ have a known asymptotic median bias of $\Delta$. %Assume that $\Delta$ is known. 
Recall $\tau_{\alpha,\Delta}$ defined in~\eqref{eq:definition-tau}. Proposition~\ref{prop:approximate-to-exact-validity} implies that for every $\alpha\in(0, 1)$ satisfying $\tau_{\alpha,\Delta} \neq 0$ there exists $N_{\alpha}\ge1$ such that for all $n\ge N_{\alpha},$
\begin{equation}\label{eq:infinite-order-coverage}
\mathbb{P}\left(\theta_0 \in \left[\min_{1\le j\le B_{\alpha,\Delta}}\widehat{\theta}_j,\,\max_{1\le j\le B_{\alpha,\Delta}}\widehat{\theta}_j\right]\right) \ge 1 - \alpha.
\end{equation}
Inequality~\eqref{eq:infinite-order-coverage} is obvious from Lemma~\ref{lem:fixed-B-coverage-result} with $B_{\alpha,\widetilde{\Delta}}$ estimators. Proposition~\ref{prop:approximate-to-exact-validity} along with asymptotic median bias of $\Delta$ implies that $B_{\alpha,\widetilde{\Delta}} = B_{\alpha,\Delta}$ for $n\ge N_{\alpha}$. The threshold sample size $N_{\alpha}$ depends on how fast $|\widetilde{\Delta} - \Delta|$ converges to zero and how big the right hand side of~\eqref{eq:consistency-subsampling},~\eqref{eq:consistency-subsampling-asymptotic-median-unbiased} are. The coverage guarantee~\eqref{eq:infinite-order-coverage} can be compared to the coverage guarantee for Wald, bootstrap, and subsampling intervals. None of these intervals have a guarantee of at least $1 - \alpha$ coverage even for large sample sizes; the coverage only converges to $1 - \alpha$ with sample size. 

{\color{black}To understand how $N_{\alpha}$ depends on $\alpha$, we consider bounds on $|\widetilde{\Delta} - \Delta|$ which use traditional Berry--Esseen bounds.} In most cases including parametric and semiparametric models, Berry--Esseen type bounds are available that provide bounds of the form,
\begin{equation}\label{eq:BE-bound-general}
\sup_{t\in\mathbb{R}}\left|\mathbb{P}\left(\frac{|S_j|^{1/2}(\widehat{\theta}_j - \theta_0)}{\sigma} \le t\right) - \Phi(t)\right| ~\le~ \frac{\mathfrak{C}_X}{|S_j|^{1/2}},
\end{equation}
where $|S_j|$ is the number of observations in the $j$-th split of the sample based on which $\widehat{\theta}_j$ is computed. Here $\mathfrak{C}_X$ is a constant that depends on the true distribution of the data. (If $\widehat{\theta}_j$ is the sample mean, then $\mathfrak{C}_X$ can be bounded in terms of the skewness of the random variables $X_i, i\ge1$.) For results of this type, see~\cite{pfanzagl1971berry,pfanzagl1973accuracy},~\cite{bentkus1997berry,bentkus2005lyapunov}, and~\cite{pinelis2017optimal}. (In semi/non-parametric models as in~\cite{zhang2011berry} and~\cite{han2019berry}, the rate of convergence may be slower than $|S_j|^{-1/2}$.) In this case, assuming $|S_j| \asymp \sqrt{n/B_{\alpha,0}}$ (i.e., data is split approximately equally into $B^*$ many samples), we get that,
\begin{equation}\label{eq:BE-median-bias}
\widetilde{\Delta} ~\le~ \max_{1\le j\le B_{\alpha,0}}\left|\mathbb{P}(\widehat{\theta}_j - \theta_0 \le 0) - \frac{1}{2}\right| ~\le~ \mathfrak{C}_X\sqrt{\frac{B_{\alpha,0}}{n}}.
\end{equation}
Note that this conclusion requires a weaker bound than the one in~\eqref{eq:BE-bound-general} because we only care about $t = 0$ in~\eqref{eq:BE-bound-general}. For example, in case of the sample mean, if the observations are symmetric around the population mean, then $\widetilde{\Delta} = 0$ irrespective of any moment assumptions, but a general Berry--Esseen bound~\eqref{eq:BE-bound-general} need not hold true without additional moment assumptions. If we take $\mathcal{M}$ to be the set of all strictly increasing functions and $\mathcal{S}$ is the class of all continuous distributions $F$ with $F(0) = 1/2$, then $\widetilde{\Delta}$ can also be bounded as
\begin{equation}\label{eq:general-bound-median-bias}
\widetilde{\Delta} \le \max_{1\le j\le B_{\alpha,0}}\,\inf_{h\in\mathcal{M}}\,\inf_{F\in\mathcal{S}}\,\left|\mathbb{P}\left(h(\widehat{\theta}_j) - h(\theta_0) \le 0\right) - F(0)\right|.
\end{equation}
This follows from the fact that $\{\widehat{\theta}_j \le \theta_0\} = \{h(\widehat{\theta}_j) \le h(\theta_0)\}$ for all strictly increasing functions $h$. Allowing for arbitrary increasing transformations may result in better normal approximations in many cases. Classical examples include the Fisher's z-transformation for the correlation coefficient and Anscombe's arcsine transformation for Binomial random variable; see~\cite{borges1970approximation,gebhardt1969some,borges1971derivation,efron1982transformation} for some examples. Because symmetric distributions belong to $\mathcal{S}$ and the standard normal distribution belongs to it, the right hand side of~\eqref{eq:general-bound-median-bias} is always better (i.e., smaller) than the bound attained by~\eqref{eq:BE-bound-general}. Moreover, if $\widehat{\theta}_j$ has zero median, then the right hand side of~\eqref{eq:general-bound-median-bias} is zero but~\eqref{eq:BE-bound-general} can result in a constant order upper bound.

If inequality~\eqref{eq:BE-median-bias} holds true, then {\color{black}the requirement~\eqref{eq:relaxed-requirement-median-unbiased} holds if
\begin{equation}\label{eq:sample-size-requirement}
\mathfrak{C}_X\frac{B_{\alpha,0}^{3/2}}{n^{1/2}} ~\le~ W_0\left(\frac{1}{\sqrt{2}}\sqrt{\frac{\alpha}{P(B_{\alpha,0}; 0)} - 1}\right),
\end{equation}}
% where $\mathfrak{C}'_X$ is a slightly adjusted version of the constant $\mathfrak{C}_X^2$ {\color{black}and $W_0(\cdot)$ represents the principal branch of the Lambert function.}
This equivalence~\eqref{eq:sample-size-requirement} follows from inequality~\eqref{eq:inequalities-B-alpha-Delta} for $B_{\alpha,0}$. 
{\color{black}The right hand side of~\eqref{eq:sample-size-requirement}} can be as large as $0.4$ even for small values of $\alpha$ as shown in Figure~\ref{fig:right-hand-side-of-consistency-subsampling-median-unbiased}.

By making use of an Edgeworth expansion, estimators with smaller median bias can be constructed via median bias reduction~\citep{pfanzagl1973asymptotic,kenne2017median}. \citet[Section 6]{pfanzagl1973asymptotic} provides a general recipe for constructing estimators with a median bias of $o(n^{-(s-2)/2})$ for any $s \ge 3$. \citet[Eq. (3)]{kenne2017median} yields estimates $\widehat{\theta}_j$ that satisfy $|\mathbb{P}(\widehat{\theta}_j \le \theta_0) - 1/2| = O((B_{\alpha,0}/n)^{3/2})$. In this case, $\widetilde{\Delta} \le \mathfrak{D}_XB_{\alpha,0}^{3/2}/n^{3/2}$ for some constant $\mathfrak{D}_X$ and hence, requirement~\eqref{eq:sample-size-requirement} can be relaxed to
\[
\mathfrak{D}_X\frac{B_{\alpha,0}^{5/2}}{n^{3/2}} ~\le~ W_0\left(\frac{1}{\sqrt{2}}\sqrt{\frac{\alpha}{P(B_{\alpha,0}; 0)} - 1}\right).
\]
% for a slightly adjusted constant $\mathfrak{D}'_X$. 
The reduction in median bias, hence, leads to a smaller threshold sample size $N_{\alpha}$ after which our intervals are finite-sample valid.

The above argument for asymptotically median unbiased estimators implies that the smallest interval containing $B_{\alpha,0}$ many independent estimators of $\theta_0$ has a finite sample coverage of at least $1 - \alpha$ after a sample size of $N_{\alpha}$. This, however, does not imply coverage validity for the confidence interval returned by Algorithm~\ref{alg:confidence-known-Delta}. This happens because with non-zero probability Algorithm~\ref{alg:confidence-known-Delta} uses $B_{\alpha,0} - 1 < B_{\alpha,0}$ estimators. The following result proves upper and lower bounds on the miscoverage of the confidence interval returned by Algorithm~\ref{alg:confidence-known-Delta} with $\Delta$ whenever the estimation procedure $\mathcal{A}(\cdot)$ has a median bias of $\widetilde{\Delta}$ converging to $\Delta$ as sample size diverges to infinity. For simplicity, the result is stated only for asymptotically median unbiased estimators, i.e., $\Delta = 0$. See Remark~\ref{rem:miscoverage-Delta} and the proof of Theorem~\ref{thm:infinite-order-coverage} for upper and lower bounds on true coverage when Algorithm~\ref{alg:confidence-known-Delta} is applied with $\Delta$ when the estimators has a median bias of at most $\widetilde{\Delta}$.

Recall that $\widehat{\mathrm{CI}}_{\alpha,\Delta}$ is the confidence interval returned by Algorithm~\ref{alg:confidence-known-Delta} when it is applied with $\Delta$ as the median bias parameter. 
Theorem~\ref{thm:infinite-order-coverage} below is proved in Section~\ref{appsec:infinite-order-coverage}.
\begin{thm}\label{thm:infinite-order-coverage}
Suppose $X_1, \ldots, X_n$ are independent random variables. If the estimation procedure $\mathcal{A}(\cdot)$ returns estimators that have a median bias of at most $\widetilde{\Delta} \ge 0$,
then 
\begin{equation}\label{eq:some-order-coverage-upper-bound}
\mathbb{P}(\theta_0\notin\widehat{\mathrm{CI}}_{\alpha,0}) ~\le~ \alpha\left(1 + 2B_{\alpha,0}(B_{\alpha,0} - 1)\widetilde{\Delta}^2(1 + 2\widetilde{\Delta})^{(B_{\alpha,0} - 2)_+}\right)\quad\mbox{for every}\quad\alpha\in(0, 1).
\end{equation}
Furthermore, if $\mathbb{P}(\widehat{\theta}_j \le \theta_0) = \mathbb{P}(\widehat{\theta}_1 \le \theta_0)$ for all $j\ge1$ and $\widehat{\theta}_j, j\ge1$ all have the same median bias of $\widetilde{\Delta}$, then
\begin{equation}\label{eq:some-order-coverage-lower-bound}
\mathbb{P}(\theta_0 \notin \widehat{\mathrm{CI}}_{\alpha,0}) ~\ge~ \alpha,\quad\mbox{for every}\quad\alpha\in(0, 1).
\end{equation}
\end{thm}
{\color{black}Note that the conditions of same median bias $\widetilde{\Delta}$ and same probability of undercoverage are trivally satisfied if all of the batches contain the same number of observations and all the observations are independent and identically distributed. }%Equal number of observations in each batch can be obtained by ignoring some observations, if necessary.}
\begin{rem}\label{rem:miscoverage-Delta}
In Section~\ref{appsec:infinite-order-coverage}, we also consider the case when $\Delta$ is not necessarily 0. 
If the finite-sample median bias $\widetilde{\Delta}$ of the estimation procedure $\mathcal{A}(\cdot)$ is close to $\Delta$ (rather than zero), then 
\[
\mathbb{P}(\theta_0 \notin \widehat{\mathrm{CI}}_{\alpha,\Delta}) \le \alpha\left(1 + 2|\widetilde{\Delta} - \Delta|\right)^{B_{\alpha,\Delta}}.
\]
Further, if $\mathbb{P}(\widehat{\theta}_j = \theta_0) = 0$ and the estimators $\widehat{\theta}_j, j\ge 1$ all have the same median bias $\widetilde{\Delta}$, then
\[
\mathbb{P}(\theta_0 \notin \widehat{\mathrm{CI}}_{\alpha,\Delta}) \ge \alpha\left(1 + 2|\widetilde{\Delta} - \Delta|\right)^{-B_{\alpha,\Delta}}.
\]
These two inequalities imply that if $B_{\alpha,\Delta}|\widetilde{\Delta} - \Delta| = o(1)$, then $\mathbb{P}(\theta_0\notin\widehat{\mathrm{CI}}_{\alpha,\Delta})$ converges to $\alpha$.
\end{rem}
\begin{rem}
The main conclusion of Theorem~\ref{thm:infinite-order-coverage} is that Algorithm~\ref{alg:confidence-known-Delta} can be used with $\Delta = 0$ and it retains asymptotic validity for large sample sizes if  the estimation procedure $\mathcal{A}(\cdot)$ produces asymptotically median unbiased estimators. %Note that estimators which are asymptotically normal are asymptotically median unbiased. 
\end{rem}
Theorem~\ref{thm:infinite-order-coverage} is a finite sample result characterizing explicitly the effect of misspecifying $\Delta$ in Algorithm~\ref{alg:confidence-known-Delta}. The misspecification of $\Delta$ is measured by how far the median bias $\widetilde{\Delta}$ of the estimators $\widehat{\theta}_j, j\ge 1$ is from $\Delta$, the asymptotic median bias. 
To illustrate Theorem~\ref{thm:infinite-order-coverage}, consider the setting under which~\eqref{eq:BE-median-bias} holds true.
Theorem~\ref{thm:infinite-order-coverage} along with~\eqref{eq:BE-median-bias} implies that,
\begin{equation}\label{eq:conclusion-theorem-BE-bound}
\alpha ~\le~ \mathbb{P}(\theta_0 \notin \widehat{\mathrm{CI}}_{\alpha,0}) ~\le~ \alpha\left(1 + \mathfrak{C}'_X\frac{B_{\alpha,0}^3}{n}e^{\mathfrak{C}_XB_{\alpha,0}^{3/2}/n^{1/2}}\right),\quad\mbox{for every}\quad \alpha\in(0, 1).
\end{equation}
In case an estimation procedure $\mathcal{A}(\cdot)$ with reduced median bias is employed in Algorithm~\ref{alg:confidence-known-Delta}, then we get $\widetilde{\Delta} \le \mathfrak{D}_X(B_{\alpha,0}/n)^{3/2}$ for some constant $\mathfrak{D}_X$ and hence, Theorem~\ref{thm:infinite-order-coverage} yields
\begin{equation}\label{eq:conclusion-theorem-Edgeworth-expansion}
\alpha ~\le~ \mathbb{P}(\theta_0 \notin \widehat{\mathrm{CI}}_{\alpha,0}) ~\le~ \alpha\left(1 + \mathfrak{D}_X'\frac{B_{\alpha,0}^5}{n^3}e^{\mathfrak{D}_XB_{\alpha,0}^{5/2}/n^{3/2}}\right),\quad\mbox{for every}\quad \alpha\in(0, 1). 
\end{equation}

Theorem~\ref{thm:infinite-order-coverage} (and the conclusions~\eqref{eq:conclusion-theorem-BE-bound},~\eqref{eq:conclusion-theorem-Edgeworth-expansion}) can be compared to the guarantees offered by classical confidence intervals constructed based on the assumption of asymptotic normality. Under a bound like~\eqref{eq:BE-bound-general}, such confidence intervals only satisfy
\begin{equation}\label{eq:order-of-coverage}
\left|\mathbb{P}\left(\theta_0 \notin \widehat{\mathrm{CI}}^{\texttt{Wald}}_{\alpha}\right) - \alpha\right| \le \frac{\mathfrak{C}_X}{\sqrt{n}},\quad\mbox{for every}\quad \alpha\in(0, 1).
\end{equation}
In other words, the coverage of $\widehat{\mathrm{CI}}^{\texttt{Wald}}_{\alpha}$ differs from $(1-\alpha)$ by a quantity of order $1/\sqrt{n}$ and can significantly miscover if $\alpha \ll 1/\sqrt{n}$. The same comment also applies to the bootstrap and subsampling confidence intervals. Confidence intervals obtained by various methods are often compared in terms of the rate of convergence in~\eqref{eq:order-of-coverage}. In parametric models or, more generally, cases where $\theta_0$ is estimable at an $n^{-1/2}$ rate, confidence intervals which attain a rate of $n^{-1/2}$ in~\eqref{eq:order-of-coverage} are called first-order accurate, those that attain a rate of $n^{-1}$ are called second-order accurate and so on. Asymptotic normality based Wald confidence intervals $\widehat{\mathrm{CI}}_{\alpha}^{\texttt{Wald}}$ are usually first-order accurate. Bootstrap confidence intervals can be constructed to be second-order accurate~\citep{hall1986bootstrap,hall1988theoretical,mammen1992bootstrap}. Subsampling intervals can also be constructed to satisfy second-order accuracy~\citep{bertail2001extrapolation}. In contrast, the \hulc readily obtains second-order accuracy and is valid even if $\alpha$ converges to zero. Further, if we use an estimator with reduced median bias, the \hulc is sixth-order accurate{\color{black}; see}~\eqref{eq:conclusion-theorem-Edgeworth-expansion}. Another important difference is that the \hulc attains relative accuracy (i.e., $|\mathbb{P}(\theta_0\notin\widehat{\mathrm{CI}}_{\alpha,0})/\alpha - 1|$ is small) instead of absolute accuracy as in~\eqref{eq:order-of-coverage}. In the problem of mean estimation, some results for relative accuracy of Wald confidence intervals are available using self-normalized large deviation techniques~\citep{shao1997self,jing2003self}. To our knowledge, such refined results are unavailable for a large class of $M$-estimators.

%%%%%%%%%%%%%%%%%%%%%%%%%%%%%%%%%%%%%%%%%%%%%%%%%%%%%
%%%%%%%%%%%%%%%%%%%%%%%%%%%%%%%%%%%%%%%%%%%%%%%%%%%%% 
\subsection{Comparison with Wald confidence intervals}\label{subsec:comparison}
In this section
we show that
our intervals have lengths close to those of the Wald intervals. In order to facilitate this comparison, we assume in this section that $\alpha \rightarrow 0$ slowly as a function of $n$.
Let $\widehat\theta = {\cal A}(\{X_1,\ldots, X_n\})$ and
$\widehat\theta_j = {\cal A}(\{X_i:\ i\in S_j)$.
Suppose that
\begin{equation}\label{eq:asymptotic-normality}
\sqrt{n}(\widehat{\theta} - \theta_0) \overset{d}{\to} N(0, \sigma^2)\quad\mbox{and}\quad \sqrt{|S_j|}(\widehat{\theta}_j - \theta_0) \overset{d}{\to} N(0, \sigma^2),
\end{equation}
as $n\to\infty$ and $|S_j| \to \infty$ for all $1\le j\le B^*$. Under
this assumption, if $\widehat{\sigma}^2$ is a consistent estimator
$\sigma^2$, then the Wald confidence interval is given by
$\widehat{\mathrm{CI}}^{\texttt{Wald}}_{\alpha} :=
[\widehat{\theta} -
  \widehat{\sigma}z_{\alpha/2}/{\sqrt{n}},\,\widehat{\theta} +
  \widehat{\sigma}z_{\alpha/2}/{\sqrt{n}}],$ where
$z_{\alpha/2}$ is the $(1-\alpha/2)$-th quantile of the standard
Gaussian distribution.  The {\color{black}scaled} width of this confidence interval is given
by
$\sqrt{n}\mbox{Width}(\widehat{\mathrm{CI}}^{\texttt{Wald}}_{\alpha})
= 2z_{\alpha/2}\widehat{\sigma}.$ This converges in probability to
$2z_{\alpha/2}\sigma$. From the properties of the normal distribution,
it follows that $z_{\alpha/2} = \sqrt{2\log(2/\alpha) -  \log(\log(2/\alpha)) - \log(2\pi)} + o(1)$ as $\alpha\to0$. See, for
example, Proposition 4.1 of~\cite{boucheron2012concentration}. Hence,
the width of the Wald confidence interval is asymptotically equal to
$2\sigma\sqrt{2\log(2/\alpha)}/\sqrt{n},$ as $\alpha \rightarrow 0$ and $n \rightarrow \infty$.

To compare this width to the width of the \hulcnospace, for simplicity, we treat $B^*$ as a fixed (i.e., non-stochastic) value and assume that $n$ is a multiple of $B^*$ so that each split has $n/B^*$ many observations. Assumption~\eqref{eq:asymptotic-normality} implies that
%\[
$\sqrt{{n}/{B^*}}(\widehat{\theta}_j - \theta_0) \overset{d}{\to} N(0, \sigma^2)$ for $1\le j\le B^*.$
% \]
Because the estimators are independent and $B^* \le B_{\alpha,\Delta} < \infty$, we get that the convergence is joint for all the estimators $\widehat{\theta}_j, 1\le j\le B^*$. Recall that our confidence interval is the smallest rectangle {\color{black}(interval)} containing these estimators and hence
\begin{equation}\label{eq:width-exact-equation}
\sqrt{\frac{n}{B^*}}\mbox{Width}(\widehat{\mathrm{CI}}_{\alpha,\Delta}) = \sqrt{\frac{n}{B^*}}\left[\max_{1\le j\le B^*}\widehat{\theta}_j - \min_{1\le j\le B^*}\widehat{\theta}_j\right] = \max_{1\le j < k\le B^*}\sqrt{\frac{n}{B^*}}(\widehat{\theta}_j - \widehat{\theta}_k).
\end{equation}
Joint asymptotic convergence of the estimators implies that
\begin{equation}\label{eq:width-convergence}
\sqrt{\frac{n}{B^*}}\mbox{Width}(\widehat{\mathrm{CI}}_{\alpha,\Delta}) \overset{d}{\to} \max_{1\le j < k \le B^*}(G_j - G_k) = \max_{1\le j\le B^*} G_j - \min_{1\le j\le B^*} G_j, 
\end{equation}
where $(G_1, \ldots, G_{B^*})$ is a Gaussian random vector with mean zero and a diagonal covariance matrix with all diagonal entries equal to $\sigma^2$. This shows the first difference in widths. Unlike the classical Wald confidence intervals, the width of our confidence interval does not degenerate after scaling by $\sqrt{n}$; the width after proper scaling converges weakly to a non-degenerate distribution. 
Using~\eqref{eq:width-convergence}, we can control of the width of our confidence region in terms of the width of the convex hull of $B^*$ many independent mean zero Gaussian random variables. Because $G_j$'s are symmetric around zero,
\[
\mathbb{E}\left[\max_{1\le j\le B^*}G_j - \min_{1\le j\le B^*}G_j\right] = 2\mathbb{E}\left[\max_{1\le j\le B^*}G_j\right] = 2\sqrt{2\log(B^*)}\left[1 - \frac{\log\log B^*}{4\log B^*} + O\left(\frac{1}{\log B^*}\right)\right].
\]
The last equality here holds as $\alpha\to0$ and follows from Theorem 1.2 of~\cite{kabluchko2019expected} (and the discussion before that theorem). Therefore, the width of our confidence interval is asymptotically
$2\sigma\sqrt{2B^*\log(B^*)/n}.$
From inequalities~\eqref{eq:inequalities-B-alpha-Delta}, we know that $B_{\alpha,\Delta}$ and $B^*$ are of order $\log_2(2/\alpha)$; note that under asymptotic normality, we can take $\Delta = 0$. Hence, the width of our confidence interval is asymptotically
\begin{equation}\label{eq:width-convex-hull}
2\sigma\sqrt{\frac{2\log(2/\alpha)}{n}}\sqrt{\log_2(\log_2(2/\alpha))}.
\end{equation}
This implies that the ratio of the expected width of our confidence interval to that of the Wald interval is approximately equal to $\sqrt{\log_2(\log_2(2/\alpha))}$.
This is always larger than $1$, and grows very slowly as $\alpha\to0$. For $\alpha\in[0.01, 0.2]$, this ratio ranges between $1.71$ and $1.32$. In a way, this is the price to pay for the generality of the confidence interval. While the Wald confidence interval makes complete use of asymptotic normality, our confidence interval only makes use of the fact that its median is zero; we do not even make use of symmetry. 

Unlike the Wald confidence interval, the \hulc does not explicitly or implicitly estimate the variance of the estimator but its width as given in~\eqref{eq:width-convex-hull} adapts to the unknown standard deviation $\sigma$. The calculation shown above uses asymptotic arguments, but some simple bounds can be obtained using no more than two moments for $\widehat{\theta}_j$. Observe from~\eqref{eq:width-exact-equation} that
\begin{equation}\label{eq:width-simple-calculation}
\begin{split}
\mathbb{E}\left[\mbox{Width}(\widehat{\mathrm{CI}}_{\alpha,\Delta})\right] &= \mathbb{E}\left[\max_{1\le j\le B^*}\widehat{\theta}_j - \min_{1\le j\le B^*}\widehat{\theta}_j\right]\\
&\le 2\mathbb{E}\left[\max_{1\le j\le B^*}|\widehat{\theta}_j - \theta^*|\right] \le 2\left(\mathbb{E}\left[\max_{1\le j\le B^*}|\widehat{\theta}_j - \theta^*|^2\right]\right)^{1/2}\\
&\le 2\left(\mathbb{E}\left[\sum_{j=1}^{B^*}|\widehat{\theta}_j - \theta^*|^2\right]\right)^{1/2} \le 2\sqrt{B^*}\max_{1\le j\le B^*}\left(\mathbb{E}[|\widehat{\theta}_j - \theta^*|^2]\right)^{1/2}.
\end{split}
\end{equation}
Assuming convergence in mean square of $\sqrt{n/B^*}(\widehat{\theta}_j - \theta^*)$ to a distribution with mean zero and variance $\sigma^2$, we get that the expected width is asymptotically bounded by $2\sigma B^*/\sqrt{n}$. This calculation does not require convergence to Gaussianity and shows that the width of our confidence interval, in general, adapts to the standard deviation of the estimators. The calculation~\eqref{eq:width-simple-calculation} can be significantly improved if the estimators are known to have higher moments. In the first inequality of~\eqref{eq:width-simple-calculation} we only use second moment Jensen's inequality. Replacing the second moments by $q$-th moment here will yield $(B^*)^{1/q}$ instead of $\sqrt{B^*}$ in the last line of~\eqref{eq:width-simple-calculation}.

\paragraph{Transformed Parameters.} 
In contrast to Wald intervals,
the \hulc interval
is equivariant to
monotone transformations{\color{black}, assuming that the estimators are equivariant under monotone transformations}. 
It is worth noting that the validity of our confidence interval does
not require any smoothness conditions on the transformation
$g(\cdot).$ In comparison, the delta method requires continuous
differentiability of $g(\cdot)$.
\subsection{Numerical Comparisons}
\subsubsection{Simple Linear Regression}
Figure~\ref{fig:comparison-coverage-lm}
shows the coverage and width of the
95\% \hulc interval
{\color{black}(obtained from Algorithm~\ref{alg:confidence-known-Delta} with $\Delta = 0$)}
and the Wald interval
from ordinary least squares linear
regression. The simulation setting is as follows: for $n \in \{20, 50,
100, 1000\}$, independent observations $(X_i, Y_i), 1\le i\le n$ are
generated from
\begin{equation}\label{eq:gamma-distribution}
X_i\sim \mbox{Uniform}[0, 10],\;\xi_i\sim N(0, 1),\quad\mbox{and}\quad Y_i = 1 + 2X_i + \gamma X_i^{1.7} + \exp(\gamma X_i)\xi_i.
\end{equation}
For $\gamma = 0$, observations $(X_i, Y_i)$ follow the standard linear model and for $\gamma > 0$, observations do not follow a linear model with non-linear mean function and a heteroscedastic error variable. {\color{black}With $\gamma$, misspecification from linear conditional expectation and homoscedasticity increase}. We define the estimator and target as $\widehat{\beta}$ and $\beta^*_{\gamma}$, where
\[
(\widehat{\alpha}, \widehat{\beta}) := \argmin_{\alpha,\beta}\frac{1}{n}\sum_{i=1}^n (Y_i - \alpha - \beta X_i)^2,\quad\mbox{and}\quad (\alpha^*_{\gamma}, \beta^*_{\gamma}) := \argmin_{\alpha,\beta}\mathbb{E}_{\gamma}[(Y - \alpha - \beta X)^2].
\]
Here $\mathbb{E}_{\gamma}[\cdot]$ represents the expectation when $(X, Y)$ are generated from~\eqref{eq:gamma-distribution}.
Note that $\beta^*_{\gamma}$ need not be equal to $2$ for $\gamma > 0$. By Monte-Carlo approximation of $\mathbb{E}_{\gamma}[\cdot]$ with $10^8$ samples, we have $\beta_{0.25}^* = 3.2791, \beta^*_{0.5} = 4.5567, \beta^*_{0.75} = 5.8239,$ and $\beta^*_{1} = 6.8093$. The Wald interval in this case is obtained using the sandwich variance estimator as in~\cite{buja2019models}.
\begin{figure}[!h]
\centering
\includegraphics[width=\textwidth]{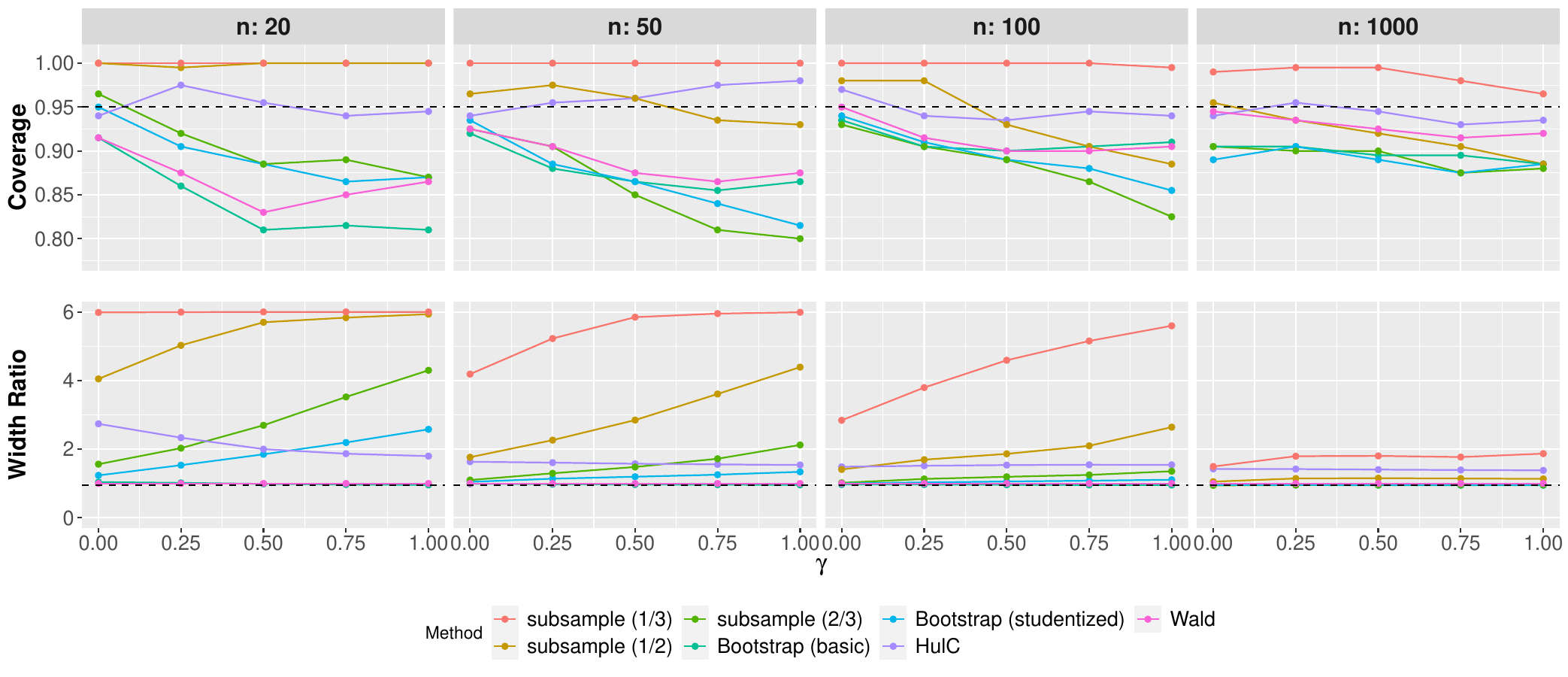}
\vspace{-0.4in}
\caption{{\color{black}Comparison of width and coverage between our confidence interval with Wald's, subsampling, and two bootstraps with a simple linear regression slope estimator as the sample size changes from 20 to 1000 and misspecification parameter $\gamma$ changes from $0$ to $1$. Our method is shown as ``HulC,'' Wald's is shown as ``Wald,'' subsampling with different subsample sizes are shown as ``subsample (1/3)'' (with subsample size $n^{1/3}$), ``subsample (1/2)'' (with subsample size $n^{1/2}$), ``subsample (2/3)'' (with subsample size $n^{2/3}$), and two bootstraps are shown as ``Bootstrap (basic)'' and ``Bootstrap (studentized).'' The empirical coverage in the left plot is computed based on 200 replications. The width ratios are truncated at $6$. The subsampling methods can have much larger confidence interval for smaller sample sizes. The four column plots correspond to four different sample sizes $n = 20, 50, 100, 1000$. The bottom panel shows the ratio of the widths of our confidence interval as well as bootstrap/subsampling confidence intervals to that of the Wald confidence interval; the \hulc yields a $50\%$ larger interval than Wald's.}
% The width ratio plot is truncated at $y = 7$.
}
\label{fig:comparison-coverage-lm}
\end{figure}
\subsubsection{Multiple Linear Regression}
Figure~\ref{fig:width-coverage-comparison-multiple-linear-regression} provides an illustration when the estimator is obtained from multiple linear regression. The setting for Figure~\ref{fig:width-coverage-comparison-multiple-linear-regression} is as follows: for $20 \le n\le 500$, independent observations $(X_i, Y_i)\in\mathbb{R}^6\times \mathbb{R}, 1\le i\le n$ are generated from $Y_i = |\theta_0^{\top}X_i| + \xi_i$, where $\xi_i\sim N(0, 1)$ and $X_i\in\mathbb{R}^6$ is generated according to the following law: $(X_{i,1}, X_{i,2}) \sim \text{Uniform} [-1,1]^2$, $ X_{i,3}:=0.2 X_{i,1}+ 0.2 (X_{i,2}+2)^2+0.2 Z_{i,1}$, $X_{i,4}:=0.1+ 0.1(X_{i,1}+X_{i,2})+0.3(X_{i,1}+1.5)^2+  0.2 Z_{i,2}$, $X_{i,5} \sim \text{Ber}(\exp(X_{i,1})/\{1+\exp(X_{i,1})\}),$ and  $ X_{i,6} \sim \text{Ber}(\exp(X_{i,2})/\{1+\exp(X_{i,2})\})$. Here $(Z_{i,1}, Z_{i,2})\sim\text{Uniform} [-1,1]^2$ are independent of $(X_{i,1},X_{i,2})$ and $\theta_0 = (1.3, -1.3, 1, -0.5, -0.5, -0.5)/\sqrt{5.13}$. This is also a misspecified linear regression model and is taken from~\cite{kuchibhotla2021semiparametric}. Our estimator and target are defined as
\[
(\widehat{\alpha},\widehat{\beta},\widehat{\gamma}) := \argmin_{\alpha,\beta,\gamma}\frac{1}{n}\sum_{i=1}^n (Y_i - \alpha - \beta X_{i,1} - \gamma^{\top}X_{i,-1})^2,\quad\mbox{and}\quad (\alpha^*, \beta^*, \gamma^*) := \argmin_{\alpha,\beta,\gamma}\mathbb{E}[(Y - \alpha - \beta X_1 - \gamma^{\top}X_{-1})^2],
\]
where $X_{i,-1}$ and $X_{-1}$ represent the last 5 coordinates of $X_{i}$ and $X$ respectively. With Monte Carlo approximation of $\mathbb{E}[\cdot]$, we found that $\beta^* = -0.137323$. For level $\alpha = 0.05$, the \hulc {\color{black}(with $\Delta = 0$)} requires splitting the data into approximately 5 parts. This implies that for a sample of size $20$, each part only has 4 observations and one cannot fit uniquely a linear regression estimator because the model has 6 covariates. Interestingly, when we just use the output from \texttt{R} function \texttt{lm()}, the \hulc still covers the true $\beta^*$ with required confidence because in this case \texttt{lm()} simply ignores the last 2 covariates.
\begin{figure}[!h]
\centering
\includegraphics[width=\textwidth]{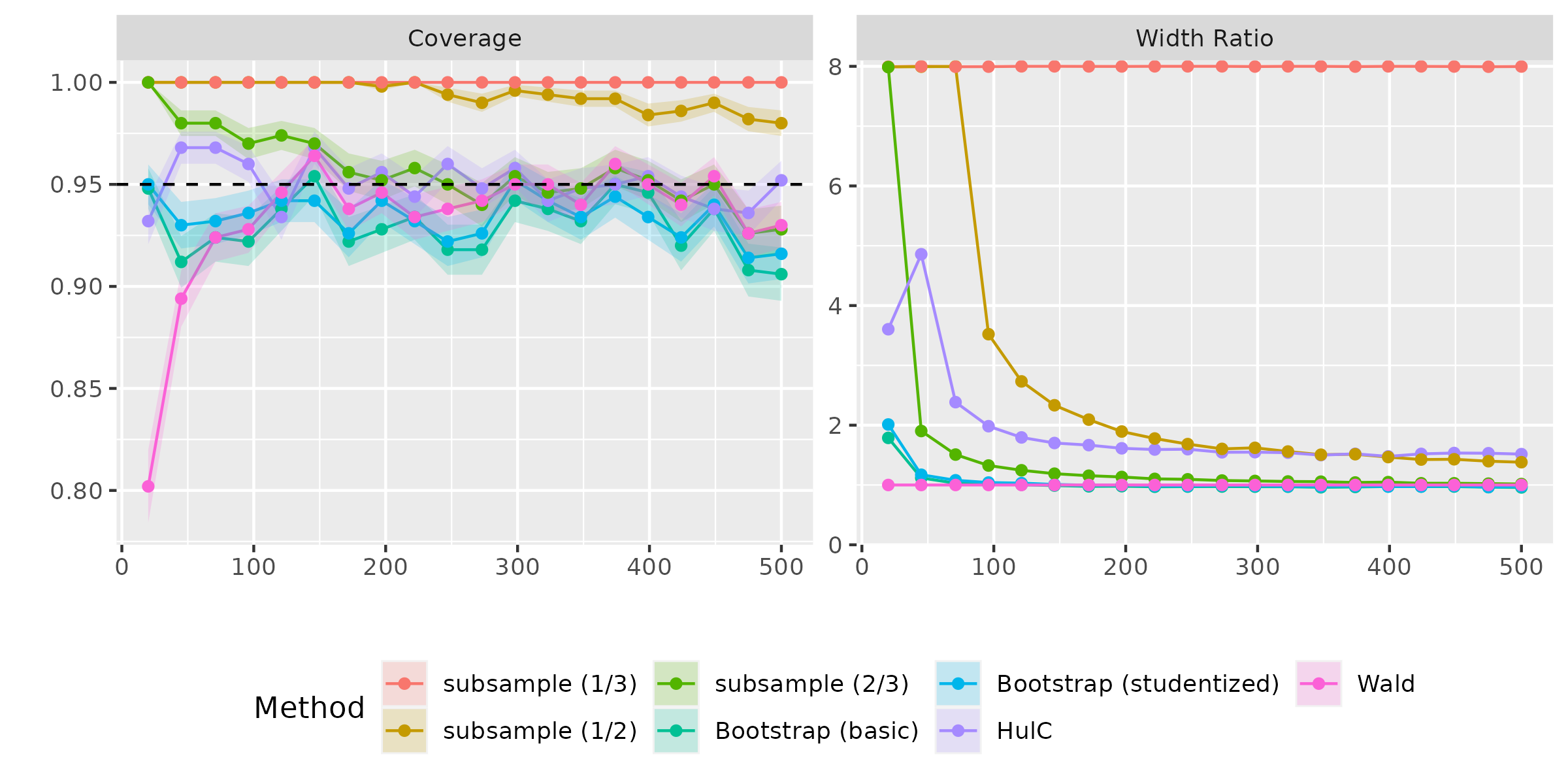}
\vspace{-0.4in}
\caption{{\color{black}Comparison of width and coverage between our confidence interval with Wald's, subsampling, and two bootstraps with a multiple linear regression estimator as the sample size changes from 20 to 500. Our method is shown as ``HulC,'' Wald's is shown as ``Wald,'' subsampling with different subsample sizes are shown as ``subsample (1/3)'' (with subsample size $n^{1/3}$), ``subsample (1/2)'' (with subsample size $n^{1/2}$), ``subsample (2/3)'' (with subsample size $n^{2/3}$), and two bootstraps are shown as ``Bootstrap (basic)'' and ``Bootstrap (studentized).'' The empirical coverage in the left plot is computed based on 200 replications for each sample size between 20 and 500. The width ratios are truncated at $y = 8$. The subsampling methods can have much larger confidence interval for smaller sample sizes.}}%}
\label{fig:width-coverage-comparison-multiple-linear-regression}
\end{figure}
\subsubsection{Quantile Regression}
{\color{black}
In the previous examples of simple and multiple regression both bootstrap and subsampling are well-known to be consistent and their width matches that of Wald interval as the sample size diverges. In this subsection, we consider the example of quantile regression where the theory can be a lot more subtle. We restrict ourselves to the case of quantile regression with one covariate and no intercept, i.e., with data $(X_i, Y_i)\in\mathbb{R}^2, 1\le i\le n$, the estimator $\widehat{\theta}_n$ is given by
\[
\widehat{\theta}_n := \argmin_{\theta\in\mathbb{R}}\sum_{i=1}^n |Y_i - \theta X_i|.
\] 
This is a univariate M-estimation problem with a convex objective function and the results of~\cite{kuchibhotla2021median} imply that $\widehat{\theta}_n$ is asymptotically median unbiased for the population parameter $\theta_0 = \argmin_{\theta\in\mathbb{R}}\mathbb{E}[|Y - \theta X|]$ without requiring any assumptions on the conditional distribution of $Y$ given $X$. On the other hand, asymptotic normality of $\widehat{\theta}_n$ requires existence of the conditional density of $Y$ given $X$; see~\cite{knight1998limiting,knight1999asymptotics,knight2008asymptotics} for details. Further, if the conditional density of $Y$ given $X$ does not exist, then the rate of convergence depends on the smoothness properties of the conditional distribution function. For our discussion, we focus on a particular example discussed in Example 1 of~\cite{knight1999asymptotics}. 

Suppose $(X_i, Y_i), 1\le i\le n$ are independent and identically distributed random vectors obtained via
\[
Y_i = X_i + \varepsilon_i\quad\mbox{where}\quad \mathbb{P}(\varepsilon_i \le t) = 0.5(1 + |t|^{\alpha}\mbox{sgn}(t)),\quad t\in[-1, 1].
\]
Theorem 1 (and Example 1) of~\cite{knight1999asymptotics} imply that $n^{1/(2\alpha)}(\widehat{\beta} - 1)$ converges in distribution to a non-normal distribution defined by a minimization problem. Section 3 of~\cite{knight1999asymptotics} shows that bootstrap is consistent for this problem if and only if $\alpha = 1$ (or equivalently, when the limiting distribution is normal). Subsampling is also not readily applicable because the rate of convergence of the estimator is unknown apriori. But subsampling with estimated rate of convergence as in~\cite{bertail1999subsampling} is applicable. We compare the performance of \hulc, and subsampling with estimated rate of convergence with three different choices of subsample sizes ($n^{1/3}$, $n^{1/2}$, and $n^{2/3}$) across different $\alpha\in(0, 1]$ and different sample sizes $n$. (It is worth mentioning here that \hulc does not involve any tuning parameters where as subsampling with estimated rate of convergence includes more than 10 tuning parameters other than the subsample size.) Figure~\ref{fig:quantile_reg} shows the performance of these procedures based on 200 Monte Carlo replications for each sample size and each $\alpha$.
\begin{figure}[!h]
\centering
\includegraphics[width=\textwidth]{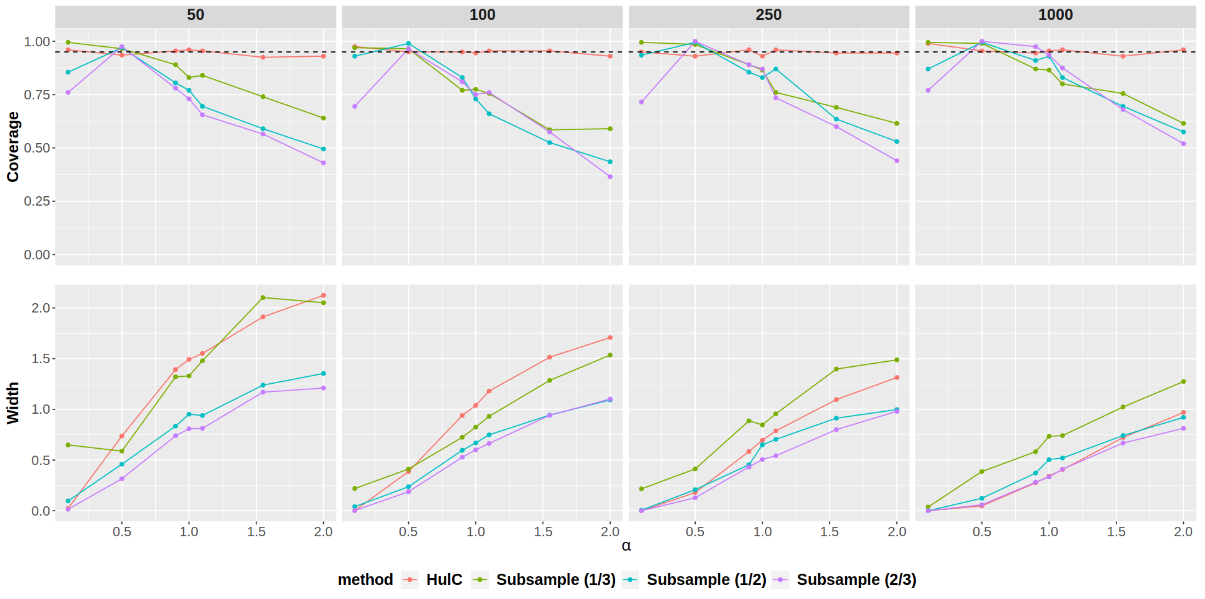}
\vspace{-0.25in}
\caption{Comparison of \hulc and subsampling with estimated rate of convergence in quantile regression under non-standard conditions. The sample size is mentioned at the top of each plot and the smoothness parameter of the distribution $\alpha$ is on the $x$-axis. \hulc maintains the coverage at the nomial level of $0.95$ for all sample sizes, while subsampling with any subsample size fails for larger values of $\alpha$. The width of all the confidence intervals are increasing with $\alpha$ as expected.}
\label{fig:quantile_reg}
\end{figure}
}
\section{Applications to standard problems}\label{sec:applications-parametric-models}
In this section, we present some simple applications including mean estimation, median estimation, and parametric exponential models. 
In parametric and semi-parametric models, regularity conditions and efficiency theory implies the existence of estimators which when centered at the target have an asymptotic mean zero Gaussian distribution. In these cases, often one can modify the estimators to ensure reduced median bias. For some examples of (approximately) median-unbiased estimators, see~\cite{birnbaum1964median,john1974median,pfanzagl1970median,pfanzagl1970asymptotic,pfanzagl1972median,pfanzagl1979optimal,hirji1989median,andrews1987best,kenne2017median}.

{\color{black}In all the examples in this section, we assume that the batch sizes are all the same, i.e., $|S_1| = |S_2| = \ldots = |S_m|$. This can be trivially achieved by ignoring less than $B_{\alpha, \Delta}$ observations, if necessary.}

\subsection{Mean estimation}\label{subsec:mean-estimation}
Suppose $X_1, \ldots, X_n$ are independent real-valued random variables with a common mean $\mu\in\mathbb{R}$. Consider the problem of constructing a confidence interval for $\mu$. Note that the random variables need not be identically distributed. If the random variables have a finite second moment and satisfy the Lindeberg condition, then the sample mean $\widebar{X}_{j} = |S_j|^{-1}\sum_{i\in S_j} X_i$ satisfies
% \[
$\sqrt{|S_j|}(\widebar{X}_j - \mu) \overset{d}{\to} N(0, \sigma_j^2),$ where $\sigma_j^2 = \sum_{i\in S_j}\mbox{Var}(X_i)/{|S_j|}$.
% \]
This implies that the estimator $\widebar{X}_j$ is asymptotically median unbiased and Algorithm~\ref{alg:confidence-known-Delta} with $\Delta = 0$ yields an asymptotically valid confidence interval for $\mu$. In this case, Wald intervals {\color{black}are also asymptotically valid.}
% also have asymptotic coverage.

The setting becomes more interesting when we consider random variables with less than two finite moments. In this case, the limiting distribution of $\widebar{X}_j$ is known to be a stable law and its rate of convergence also changes depending on the tail decay of the random variables. 
If the random variables satisfy 
\begin{equation}\label{eq:symmetric-stable-law}
\lim_{x\to\infty} x^{\alpha}\mathbb{P}(X_i > x) = \lim_{x\to\infty}x^{\alpha}P(X_i < -x)\quad\mbox{for some}\quad\alpha\in[1, 2),
\end{equation} 
then the limiting stable law of $\widebar{X}_j$ is symmetric around zero (see, for instance, Theorem 9.34 in~\cite{breiman1968probability}). In this special case, Algorithm~\ref{alg:confidence-known-Delta} continues to provide asymptotically valid confidence intervals, while Wald intervals and the bootstrap are known to fail for $\alpha < 2$; see, for example,~\cite{athreya1987bootstrap} and~\cite{knight1989bootstrap}. In particular, if the underlying distributions are all symmetric around the mean $\mu$, then without any moment assumptions the confidence interval returned by Algorithm~\ref{alg:confidence-known-Delta} is finite sample valid.
It is worth noting that subsampling~\citep{romano1999subsampling} is still applicable in the case of infinite variance.

If the assumption~\eqref{eq:symmetric-stable-law} does not hold true, then the limiting stable law is not symmetric and the asymmetry depends on the gap between the left and right hand side quantities in~\eqref{eq:symmetric-stable-law}. In this case, the median bias of the limiting distribution is not readily available and the methods presented in previous sections are not applicable. This can be resolved using the \ahulc which we describe in Section~\ref{sec:adaptive-convex-hull}.
\subsection{Median estimation}
Suppose $X_1, \ldots, X_n$ are independent real-valued random variables with common median $m\in\mathbb{R}$. Consider the problem of constructing a confidence interval for $m$. The usual estimator for the population median is the sample median. Set $\widehat{\theta}_j = \mbox{median}(X_i: i\in S_j)$. If the average distribution function $\widebar{F}_j(t) = |S_j|^{-1}\sum_{i\in S_j}\mathbb{P}(X_i \le t)$ has a derivative bounded away from zero at $m$, then it is known~\citep{sen1968asymptotic} that
% \[
$\sqrt{|S_j|}(\widehat{\theta}_j - m) \overset{d}{\to} N(0, \sigma_j^2),$ where $\sigma_j^2 = (4\widebar{f}_j^2(m))^{-1}$.
% \] 
Here $\widebar{f}_j(m)$ is the derivative of $\widebar{F}_j(t)$ at $t = m$. There are several classical methods for constructing confidence intervals for $m$ including Wald's, quantile or rank based intervals. Wald confidence intervals in this case require estimating of the density $\widebar{f}_j(\cdot)$ at $m$ and the quantile based intervals require choosing the appropriate quantiles for end points. Unlike the Wald interval, the quantile based intervals are finite sample valid~\citep{lanke1974interval}. Because the limiting distribution is mean zero Gaussian, the \hulc applies and yields an asymptotically valid confidence interval. 

Once again the setting becomes interesting when the underlying distributions do not satisfy the conditions for normality. For example, if the density $\widebar{f}_j(\cdot)$ is not bounded away from zero at the common median $m$, then the limiting distribution of $\widehat{\theta}_j$ is not Gaussian and hence Wald as well as bootstrap intervals break down. The limiting distribution in this case is explicitly described in~\citet[Section 2]{knight1998limiting}.
% ,~\citet[Section 5]{knight2002second}, and~\citet[Example 2]{geyer1996asymptotics}. 
In this case, the rate of convergence of the median depends on how fast the density decays to zero as $t$ approaches $m$. When the population median $m$ is unique, the sample median computed based on odd number of observations is known to be median unbiased~\citep{mahamunulu1969estimation}. This observation implies that Algorithm~\ref{alg:confidence-known-Delta} with $\Delta = 0$ yields a finite sample valid confidence interval for $m$ if each $S_j$ has an odd number of observations (which can be trivially ensured). In fact, with any given number of observations (even or odd), an estimator that randomly (equally likely) chooses between the $r$-th order statistic and $(|S_j|-r+1)$-th order statistic is median unbiased for $m$ as shown in Section 4 of~\cite{mahamunulu1969estimation}. 

\subsection{Binomial distribution}
Consider $X_1, \ldots, X_n\sim \mbox{Bernoulli}(p)$ for some $p\in(0, 1)$. The problem of constructing confidence intervals for $p$ is a well-studied problem with focus on coverage as $p$ changes with the sample size $n$~\citep{brown2002confidence}. It is well-known that{\color{black}, when properly normalized,} the limiting distribution of Binom$(n, p)$ as $n\to\infty$ changes from a Gaussian to a Poisson distribution depending on whether $np\to\infty$ or $np\to\lambda\in(0,\infty)$. Because of this change, the Wald confidence intervals can undercover $p$ when $p$ is small relative to the sample size $n$~\citep{brown2001interval}. We will now consider the coverage properties of the \hulc when using the sample proportion as an estimator for $p$. For any set $S\subseteq\{1,2,\ldots,n\}$,
% \[
$\sum_{i\in S}X_i \sim \mbox{Binom}(|S|, p).$
% \]
Theorem 10 of~\cite{doerr2018elementary} shows that whenever $p\in[\log(4/3)/|S|, 1 - \log(4/3)/|S|]$, the estimator $\sum_{i\in S}X_i/|S|$ has a median bias of at most $1/4$.
Theorem 1 of~\cite{greenberg2014tight} yields this result for $p\in[1/|S|, 1 - 1/|S|]$. For a more precise result, see Lemma 8 of~\cite{doerr2018elementary}.
Hence, Algorithm~\ref{alg:confidence-known-Delta} with $\Delta=1/4$ yields finite sample coverage {\color{black}of at least $1-\alpha$} for all $p\in[\log(4/3)/m, 1 - \log(4/3)/m]$; here $m$ represents the minimum number of observations in each split of the data. Because $m \asymp n/\log(2/\alpha)$, we get finite sample coverage validity even for $p = \Theta(1/n)$. Note that the binomial distribution is not approximately normal in this case. 

Allowing for some modifications of either the estimator or the final confidence set, we can obtain finite sample coverage for all $p\in[0, 1]$. Firstly, note that {\color{black}the \hulc interval from} Algorithm~\ref{alg:confidence-known-Delta} with $\Delta = 0$ will always cover the true median of the estimators. With the proportion estimator $\sum_{i\in S_j}X_i/|S_j|$, {\color{black}the \hulc interval from} Algorithm~\ref{alg:confidence-known-Delta} with $\Delta = 0$ with a probability of at least $1 - \alpha$ will contain the median of $\mbox{Binom}(m,p)/m$ where $m = |S_j|$ for all $1\le j\le B^*$. \cite{hamza1995smallest} proves that 
\begin{equation}\label{eq:binomial-median-bias}
\left|\mathbb{E}\left[\frac{\mbox{Binom}(m,p)}{m}\right] - \mbox{median}\left(\frac{\mbox{Binom}(m,p)}{m}\right)\right| = \left|p - \mbox{median}\left(\frac{\mbox{Binom}(m,p)}{m}\right)\right| \le \frac{\log(2)}{m}.
\end{equation}
Therefore, if $\widehat{\mathrm{CI}}_{\alpha,0} = [\widehat{L}, \widehat{U}]$ represents the confidence interval from Algorithm~\ref{alg:confidence-known-Delta} with $\Delta = 0$, we get that for all $p\in[0,1]$,
\[
\mathbb{P}\left(p\notin\left[\widehat{L} - \frac{\log(2)}{m},\,\widehat{U} + \frac{\log(2)}{m}\right]\cap[0,1]\right) \le \alpha.
\] 
This is a modification of confidence interval returned by Algorithm~\ref{alg:confidence-known-Delta} but uses the classical binomial proportion estimator. If we modify the estimator, then no changes are required in Algorithm~\ref{alg:confidence-known-Delta} with $\Delta = 0$ to obtain a finite sample coverage. Because the binomial distribution has a monotone likelihood ratio, the results of~\cite{pfanzagl1970median,pfanzagl1972median} can be applied to obtain a median unbiased estimator of $p$. It might be worth noting here that binomial distribution being discrete, any median unbiased estimator has to be randomized; see page 74 of~\cite{pfanzagl2011parametric} for a discussion. The exact median unbiased estimator of \citet{pfanzagl1970median,pfanzagl1972median} is computationally intensive. A simpler estimator for $p$ with reduced median bias can be obtained from~\cite{hirji1989median}, and~\cite{kenne2017median}. These works discuss binary regression and estimating a binomial proportion is the special case when there are no regressors except for an intercept.
\cite{hamza1995smallest} also proves that~\eqref{eq:binomial-median-bias} holds true for $\mbox{Binom}(m,p)$ replaced by $\mbox{Poisson}(m\lambda)$. This implies that the confidence interval $\widehat{\mathrm{CI}}_{\alpha,0}$ from Algorithm~\ref{alg:confidence-known-Delta} inflated by $\log(2)/m$ also has a finite sample coverage {\color{black}of at least $1-\alpha$} for every $\lambda \ge 0$. 
\subsection{Exponential families}
{\color{black}\citet[Sections 3.5 and 5.4]{Lehmann1959} provide median unbiased estimators in monotone likelihood ratio and exponential families with a Lebesgue density. 
Extending this work,} \cite{pfanzagl1979optimal} provides an algorithm to construct an exactly median unbiased estimator for every sample size in a full rank exponential family, even in the presence of nuisance parameters. \cite{pfanzagl1979optimal} considers a more general parametric model than exponential families; see~\cite{read2004median}. A related result for exponential families is also obtained in~\cite{brown1976complete}. For brevity, we will not describe this algorithm here and refer to the papers mentioned above; also, see~\cite{cabrera1997simulation} for some computational methods. With such an estimator, the \hulc can be applied with $\Delta = 0$ to obtain a finite sample valid confidence interval.
\subsection{Squared mean estimation}\label{subsec:mean-square}
Suppose $X_1, \ldots, X_n$ are independent random variables with common mean $\mu$ and common variance $\sigma^2 < \infty$. Consider the estimation of $\theta_0 = \mu^2$. A natural estimator of $\theta_0$ is $\widetilde{\theta} = \widebar{X}_n^2$, the square of the sample mean. The asymptotic distribution of $\widetilde{\theta}$ depends on the true mean and the population variance:
\begin{equation}\label{eq:mean-square-distribution}
\begin{split}
n^{1/2}(\widetilde{\theta} - \theta_0) ~&\overset{d}{\to}~ N(0,4\mu^2\sigma^2),\quad\mbox{if }\theta_0 = \mu^2 \neq 0,
\quad\mbox{and}\quad n(\widetilde{\theta} - \theta_0) ~\overset{d}{\to}~ \sigma^2\rchi^2_1, \quad\mbox{if }\theta_0 = \mu^2 = 0.
\end{split}
\end{equation}
There are two aspects to consider here. First, the rate of convergence changes from $n^{-1/2}$ to $n^{-1}$ as $\mu$ changes from non-zero to zero. Second, the limiting distribution of $\widetilde{\theta}$ becomes one-sided for $\mu = 0$ and this implies that the estimator has an asymptotic median bias of $1/2$ for $\mu = 0$. The first aspect is not an issue for Algorithm~\ref{alg:confidence-known-Delta}, but the second aspect renders Algorithm~\ref{alg:confidence-known-Delta} useless for $\mu$ close to zero because it would require nearly infinite many splits of the data. Alternatively, for each subset $S_j$ of the data, consider the $U$-statistic estimator
\begin{equation}\label{eq:U-stat-prod}
\widehat{\theta}_j = \frac{1}{|S_j|(|S_j| - 1)}\sum_{i\neq k\in S_j} X_iX_k.
\end{equation}
It readily follows that $\mathbb{E}[\widehat{\theta}_j] = \mu^2$ for all $\mu\in\mathbb{R}$, unlike $\widebar{X}_n^2$ which is biased for $\mu$ close to zero. Once again $\widehat{\theta}_j$ has different limiting distributions depending the magnitude of $\mu$. Specifically,
\begin{equation}\label{eq:mean-square-distribution-U-stat}
\begin{split}
\sqrt{|S_j|}(\widehat{\theta}_j - \theta_0) ~&\overset{d}{\to}~ N(0,4\mu^2\sigma^2),\quad\mbox{if }\theta_0 = \mu^2 \neq 0,\\
% \quad\mbox{and}\quad
|S_j|(\widehat{\theta}_j - \theta_0) ~&\overset{d}{\to}~ \sigma^2(\rchi^2_1 - 1), \quad\mbox{if }\theta_0 = \mu^2 = 0.
\end{split}
\end{equation}
The rate of convergence changes between $\mu \neq 0$ and $\mu = 0$, but now the limiting distribution has median bias bounded away from zero. It may be worth pointing out that the limiting distribution in general would be a mixture of normal and Chi-square as $\mu$ becomes close to zero. We can prove the following result on the median bias of $\widehat{\theta}_j$ that is uniform over all $\mu\in\mathbb{R}$. The proof in Section~\ref{appsec:median-bias-mean-square} can be easily extended to accommodate non-identically distributed observations expect for common mean and variance.
\begin{prop}\label{prop:median-bias-mean-square}
Suppose $\xi_i = (X_i - \mu)/\sigma, 1\le i\le n$ are independent and identically distributed. Then for any $\mu$ and $\sigma$, the median bias of $\widehat{\theta}_j$ is bounded by
\begin{equation}\label{eq:median-bias-square-mean}
\begin{split}
&\sup_{\theta\in\mathbb{R}}\left|\frac{1}{2} - \left\{\Phi\left(\frac{-\theta + \sqrt{\theta^2 + |S_j|/(|S_j|-1)^2}}{\sqrt{|S_j|}/(|S_j|-1)}\right) - \Phi\left(\frac{-\theta-\sqrt{\theta^2 + |S_j|/(|S_j|-1)^2}}{\sqrt{|S_j|}/(|S_j|-1)}\right)\right\}\right|\\ 
&\quad+ \sqrt{\frac{4\mathbb{E}[\xi_1^4]\log(|S_j|)}{|S_j|\pi}} + \frac{\mathbb{E}[|\xi_1|^3] + \mathbb{E}[|\xi_1|^6]/(\mathbb{E}[\xi_1^4])^{3/2}}{\sqrt{|S_j|}} + \frac{2}{|S_j|}.
\end{split}
\end{equation}
\end{prop}
Note that the first term on the right hand side of~\eqref{eq:median-bias-square-mean} can be computed given $|S_j|$ without the knowledge of $\mu$ and $\sigma$. Further, the last three terms of~\eqref{eq:median-bias-square-mean} all disappear as $|S_j|\to\infty$ whenever certain moments of $\xi_1$ are bounded away from $0$ and $\infty$. Exact computation for some sample sizes $(|S_j|)$ shows that the supremum in the first term in~\eqref{eq:median-bias-square-mean} is attained at $\theta = 0$ and equals $|\mathbb{P}(\rchi^2_1 \le 1) - 1/2|\approx 0.183$. Hence, Algorithm~\ref{alg:confidence-known-Delta} can be applied with $\Delta = |\mathbb{P}(\rchi^2_1 \le 1) - 1/2|$ and the estimator $\widehat{\theta}_j$ to attain an uniformly valid asymptotic confidence interval for $\theta_0 = \mu^2$. It is worth pointing out that the resulting confidence interval is adaptive in its width as $\mu$ approaches zero, i.e. the expected width of the \hulc interval scales as $n^{-1}$ when $\mu$ is close to 0, and as $n^{-1/2}$ when $\mu$ is large. 
This follows from the fact that $\widehat{\theta}_j$ has an adaptive rate of convergence.

{\color{black}With 100 observations from $N(\mu, 1)$ and varying $\mu$, Figure~\ref{fig:mean-square-subsampling} shows the performance of several confidence intervals for $\theta_0 = \mu^2$ as $\mu$ changes from $0$ to $1.0$. Note that in this problem, subsampling is not readily applicable because the rate of convergence is unknown (as it depends on $\mu$). We use subsampling with estimated rate of convergence from~\cite{bertail1999subsampling}.
\begin{figure}
\centering
\includegraphics[width=\textwidth]{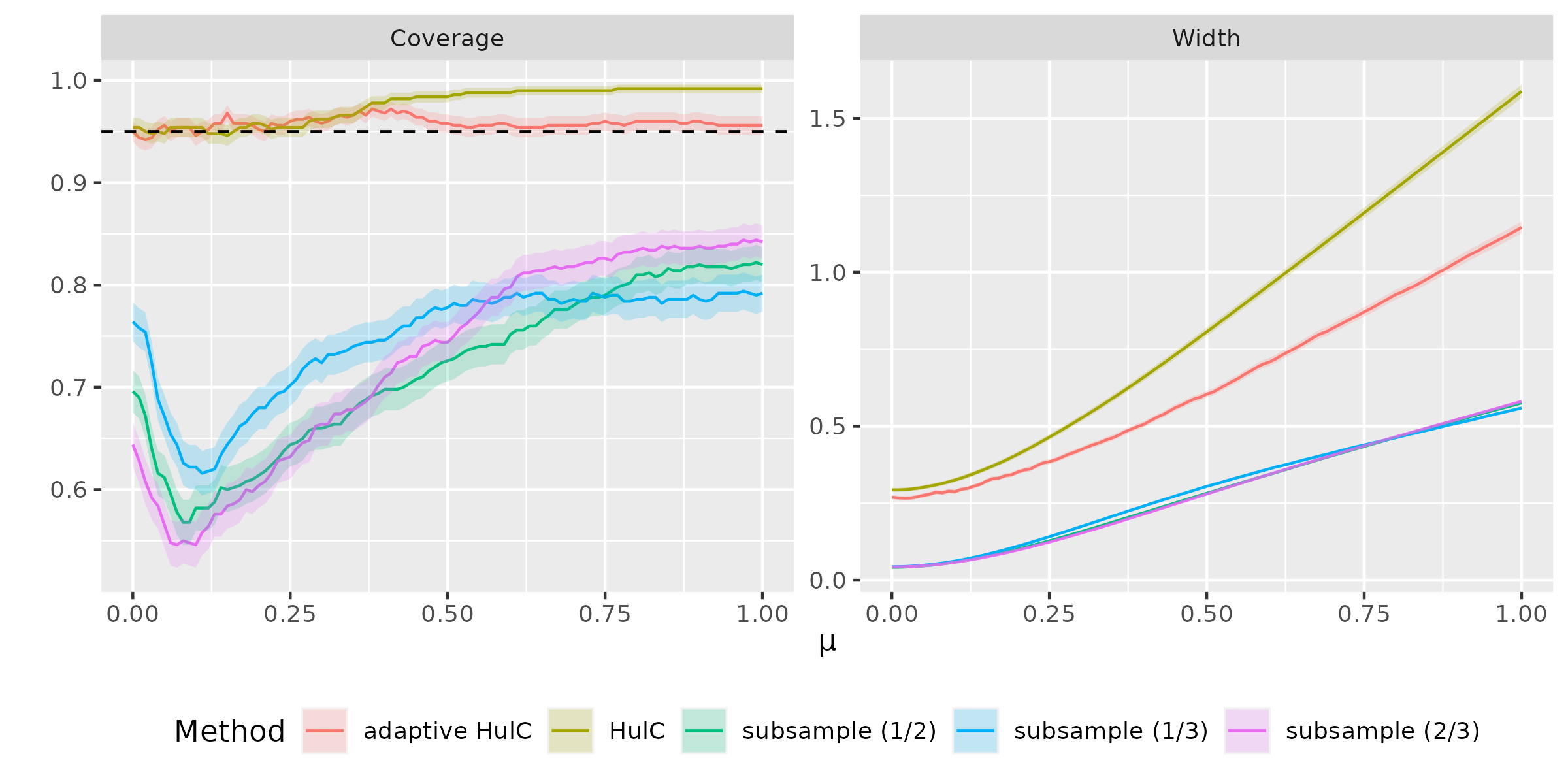}
\vspace{-0.4in}
\caption{The plots show the comparison of HulC (with $\Delta = 0.183$), adaptive HulC (discussed in Section~\ref{sec:adaptive-convex-hull}), and subsampling with different subsample sizes ``subsample (1/3)'' (subsample size of $n^{1/3}$), ``subsample (1/2)'' (subsample size of $n^{1/2}$), ``subsample (2/3)'' (subsample size of $n^{2/3}$). Although the performance of subsampling is better when $\mu = 0$ and when $\mu$ is away from zero, it is not as satisfactory as that of HulC and adaptive HulC.}
\label{fig:mean-square-subsampling}
\end{figure}
}
\subsection{Uniform model}\label{subsec:uniform-model}
Suppose $X_1, \ldots, X_n$ are independent real-valued random variables from $U[0, \theta_0]$, the uniform distribution on $[0,\theta_0]$. The maximum likelihood estimator of $\theta_0$ {\color{black}from the $j$-th batch} is given by $\widetilde{\theta}_j = \max\{X_i:\,i\in S_j\}$ which is both mean and median biased. The median bias is $1/2$ and hence Algorithm~\ref{alg:confidence-known-Delta} would be inapplicable because it requires infinitely many splits of the data. Note that in this model, $\widetilde{\theta}_j$ converges to $\theta_0$ at an $n^{-1}$ rate.

Interestingly, there are estimators of $\theta_0$ that are median unbiased in this case. For instance, with $\widehat{b}_{S_j}$ and $\widehat{a}_{S_{j}}$ representing the largest and the second largest values in $\{X_i:\,i\in S_j\}$, it can be shown that the estimator $\widehat{\theta}_j = 2\widehat{b}_{S_j} - \widehat{a}_{S_j}$ is finite sample median unbiased for $\theta_0$; see Section 3 of~\cite{robson1964estimation} for a proof. Hence, Algorithm~\ref{alg:confidence-known-Delta} can be applied with $\widehat{\theta}_j$ and $\Delta = 0$ to obtain a finite-sample valid confidence set for $\theta_0$. Note that $\widehat{\theta}_j$ also has an $n^{-1}$ rate of convergence. In this case, it is known that the classical bootstrap is invalid, but subsampling works; see e.g.,~\cite{politis1994large} and~\cite{loh1984estimating}.

The estimator $\widehat{\theta}_j$ described above is approximately median unbiased for a large class of distributions of the form $F(x)/F(\theta_0)$ for $x\in[0,\theta_0]$; see~\citet[Section 3]{robson1964estimation}. {\color{black}Here $F(\cdot)$ is an arbitrary distribution function on $[0, \infty)$.} Also, see~\cite{hall1982estimating} for other estimators of $\theta_0$, in a large class of non-parametric distributions, that have a limiting distribution that is symmetric around $\theta_0$.
\subsection{Constrained Estimation}\label{subsec:constrained-estimation}
Suppose $X_1, \ldots, X_n$ are independent real-valued random variables with mean $\mu$. Consider the estimation of $\theta_0 = \mu\mathbbm{1}\{\mu \ge 0\}$. We have seen in Section~\ref{subsec:mean-estimation} how to apply the \hulc for $\mu$. Although $\mu\mapsto\mu\mathbbm{1}\{\mu\ge0\}$ is a simple transformation, it changes the behavior of many commonly used estimators of $\theta_0$. This is because $\theta_0$ is a non-regular functional and hence, there does not exist any regular estimator for $\theta_0$; this follows from~\citet[Theorem 2]{hirano2012impossibility}. The implication is that classical Wald confidence intervals based on the estimator $\widebar{X}_n\mathbbm{1}\{\widebar{X}_n \ge 0\}$ can fail to cover $\theta_0$ for $\mu$ close to zero~\citep[Appendix 1.1]{robins2004optimal}. Further, bootstrap and subsampling are also similarly inconsistent; see~\citet[Section 3.2]{fang2019inference} and~\citet[Section 3]{andrews2000inconsistency} for bootstrap, and~\citet[Eq. (1)--(2)]{andrews2010asymptotic} for subsampling. It is, however, easy to show that the estimator $\widebar{X}_n\mathbbm{1}\{\widebar{X}_n \ge 0\}$ is asymptotically median unbiased for $\theta_0 = \mu\mathbbm{1}\{\mu\ge0\}$ because $\widebar{X}_n$ is asymptotically median unbiased for $\mu$. This follows simply from the fact that $\kappa(t) = t\mathbbm{1}\{t \ge 0\}$ is monotonic in $\mu$ and hence,
\begin{equation}\label{eq:determinisitc-inequality-constrained}
\mathbbm{1}\{\kappa(\widebar{X}_n) \ge \kappa(\mu)\} \ge \mathbbm{1}\{\widebar{X}_n \ge \mu\}\quad\mbox{and}\quad \mathbbm{1}\{\kappa(\widebar{X}_n) \le \kappa(\mu)\} \ge \mathbbm{1}\{\widebar{X}_n \le \mu\}.
\end{equation}
Note that $\kappa(\cdot)$ is not strictly increasing. This implies that
\begin{equation}\label{eq:probabilistic-inequality-constrained}
\mbox{Med-bias}_{\kappa(\mu)}(\kappa(\widebar{X}_n)) ~\le~ \mbox{Med-bias}_{\mu}(\widebar{X}_n).
\end{equation}
Hence, the \hulc with the estimator $\kappa(\widebar{X}_n)$ and $\Delta = 0$ yields a second-order accurate confidence interval for $\kappa(\mu)=\mu\mathbbm{1}\{\mu\ge0\}$. 

Inequalities~\eqref{eq:determinisitc-inequality-constrained} and~\eqref{eq:probabilistic-inequality-constrained} do not require the specific form of the function $\kappa(\cdot)$. They hold for any monotone function $\kappa(\cdot)$ and, in particular, for any piecewise constant function. Theorem 3.2 of~\cite{fang2019inference} implies that bootstrap is inconsistent unless $\kappa(\cdot)$ is differentiable. This shows the wide range of applicability of our confidence interval. Finally, we note that projection to \emph{any} set on the real line is a monotone function and hence our confidence interval from the \hulc is asymptotically valid with the natural estimator that projects the MLE (or any other estimator) to the constraint set. A simple example where this is useful is in the squared mean estimation example of Section~\ref{subsec:mean-square}. The estimator $\widehat{\theta}_j$ in~\eqref{eq:U-stat-prod} is not necessarily non-negative, but its target $\mu^2$ is always non-negative. Using the facts discussed here, we can safely use $\widehat{\theta}_j\mathbbm{1}\{\widehat{\theta}_j \ge 0\}$ instead of $\widehat{\theta}_j$ in the squared mean estimation example.
\subsection{Matching Estimators}
In causal inference, matching estimators for the average treatment effect (ATE) are popular, partly because they are intuitive. Under certain regularity conditions, matching estimators are known to be asymptotically normal centered at the true ATE. Hence, the \hulc with $\Delta = 0$ yields a second order accurate confidence interval for ATE. \cite{abadie2008failure} proved that the bootstrap is inconsistent for matching estimators. They also commented that subsampling can still be used, but given the computational cost of matching, subsampling becomes computationally intensive with larger samples. 
\subsection{Semiparametric Estimation}
In all the examples above, we have cases where the bootstrap and subsampling are either not easily applicable or fail to provide an asymptotically valid confidence interval. There are many cases where all the usual methods apply but the \hulc is much simpler and computationally cheaper.

In non- and semi-parametric problems, when a functional of interest can be estimated by an estimator that is asymptotically normal, two possibilities arise. In the simpler case, the estimator is regular and asymptotically linear with a known (or easily estimable) 
influence function, while in general the estimator may not have a simple asymptotic expansion. In the first case, we may estimate the asymptotic variance consistently via the sample variance of the (estimated) influence function but obtaining finite-sample guarantees (say via a Berry-Esseen bound) typically requires a case-by-case analysis. More generally however, 
the variance often involves more nuisance (non-parametric) components than the functional and hence, variance estimation often requires more assumptions or regularity conditions than estimation of the functional. 
On the other hand, the \hulc requires no more nuisance estimation than required for the estimation of the functional. We give three simple examples to illustrate this.
\begin{enumerate}
	\item \textbf{Integral functionals of density.} Consider the estimation of 
	% \[
	$\theta_0 = \int \phi(f(x), f'(x), \ldots, f^{(k)}(x), x)dx,$
	% \] 
	when $X_1, \ldots, X_n$ are independent and identically distributed observations {\color{black}with a common marginal density} $f$ supported on a compact set. Theorem 2 of~\cite{laurent1997estimation} provides an asymptotically efficient estimator $\widehat{\theta}_n$ for $\theta_0$ that is asymptotically normal under certain smoothness assumptions on $f$. The asymptotic variance of $\widehat{\theta}_n$, however, involves higher order $(\ge k)$ derivatives of $f$ when $k\ge1$. For a more concrete example, consider the Fisher information {\color{black}functional defined as} $\theta_0 = \int_{-\pi}^{\pi} (f'(x))^2/f(x)dx$. The asymptotic variance of the efficient estimator is given by $\int_{-\pi}^{\pi} (2f^{(2)}(x)/f(x) - (f'(x))^2/f(x))^2f(x)dx - (\int_{-\pi}^{\pi} (f'(x))^2/f(x))dx$, which involves the second derivative of $f$. 
	\item \textbf{Single Index Model.} Suppose $(X_i, Y_i), 1\le i\le n$ are independent observations satisfying $\mathbb{E}[Y_i|X_i] = m_0(\theta_0^{\top}X_i)$ when $m_0(\cdot)$ is an unknown convex function. Consider the least squares estimator $(\widehat{m},\widehat{\theta})$ which is obtained as a minimizer of $\sum_{i=1}^n (Y_i - m(\theta^{\top}X_i))^2$ over $m$ that is convex, Lipschitz, and $\theta$ in $\{\eta:\|\eta\|_2 = 1\}$. \cite{kuchibhotla2021semiparametric} prove that $\widehat{\theta}$ is asymptotically normal with an asymptotic variance depending on nuisance components such as the conditional mean of $X$ on $\theta_0^{\top}X$, the derivative of $m_0$, and the conditional variance of $Y$ given $X$.
	\item \textbf{Functionals of Normal Models.} Suppose $X_1, \ldots, X_n$ are independent observations from $N(\mu,\Sigma)$ in the space $E$ (either a Hilbert or a Banach space). Consider the estimation of $\theta_0 = f(\mu)$, if $E$ is Banach or $\theta_0 = f(\mu,\Sigma)$, if $E$ is Hilbert. \cite{koltchinskii2019estimation,koltchinskii2021efficient} provide asymptotically efficient estimators of $\theta_0$ which have a normal limiting distribution. Also, see~\cite{koltchinskii2020estimation}. The asymptotic variance is $\langle \Sigma f'(\mu), f'(\mu)\rangle$ if $\theta_0 = f(\mu)$ and is $\|\Sigma^{1/2}f_{\mu}'(\mu,\Sigma)\|^2 + 2\|\Sigma^{1/2}f_{\Sigma}'(\mu,\Sigma)\Sigma^{1/2}\|_{op}^2$ if $\theta_0 = f(\mu, \Sigma)$. Estimating the asymptotic variance hence requires estimating more complicated functionals of $\mu, \Sigma$. Such variance estimation is not discussed in these works.  
\end{enumerate}
In all of these cases our approach using the \hulc yields a conceptually simpler confidence interval without any additional nuisance estimation.

\section{\ahulcnospace}\label{sec:adaptive-convex-hull}
In this section, we provide a method, the \ahulcnospace, to estimate $\Delta$ based on subsampling~\citep{politis1994large} and consequently, provide a simple method for constructing a valid confidence interval. One might wonder at this point ``why not just use subsampling to construct the confidence interval directly?'' The answer is that the \ahulc does not require the knowledge of the rate of convergence of the estimator. As an example, in the mean estimation case with fewer than two finite moments, we do not know the rate of convergence a priori. Further, we do not need to estimate the rate of convergence as suggested in~\cite{bertail1999subsampling} for subsampling. 

We return now to the univariate parameter setting. Suppose $r_{|S_j|}(\widehat{\theta}_j - \theta_0)$ converges in distribution to $W$ a continuous random variables as $|S_j|\to\infty${\color{black}, for some sequence $\{r_k\}_{k\ge1}$ diverging to $\infty$}. Then it follows that
\[
\mbox{Med-bias}_{\theta_0}(\widehat{\theta}_j) ~\to~ \Delta := \left|\mathbb{P}(W \le 0) - \frac{1}{2}\right|,\quad\mbox{as}\quad |S_j|\to\infty.
\]
Hence, $\Delta$ is the asymptotic median bias and can be estimated using subsampling. Let $S_1^{(b)}, \ldots, S_{K_n}^{(b)}$ denote $K$ random subsamples of size $b = b(n)$ and let $\widehat{\theta}_j^{(b)}, 1\le j\le K_n$ be the estimates based on the subsamples. Let $\widehat{\theta}$ be the estimate based on the full data (of size $n$). Then $\Delta$ can be estimated by
\begin{equation}\label{eq:subsampling-median-bias-distribution}
\widehat{\Delta}_n := |L_n(0) - 1/2|,\quad\mbox{where}\quad L_n(0) := \frac{1}{K_n}\sum_{j=1}^{K_n}\mathbbm{1}\{\widehat{\theta}_j^{(b)} - \widehat{\theta} \le 0\}.
\end{equation}
Given this estimator $\widehat{\Delta}_n$, we can estimate the {\color{black}upper bound on the} miscoverage probability $P(B; \Delta)$ in~\eqref{eq:miscoverage-function-B-Delta} of the convex hull of $B$ estimators by $P(B; \widehat{\Delta}_n)$. 
The results of~\cite{politis1994large} imply that $\widehat{\Delta}_n$ is (asymptotically) consistent for $\Delta$ (see also, our Lemma~\ref{lem:subsampling-delta-control}, which develops finite-sample bounds)
and hence, $B_{\alpha,\Delta} = B_{\alpha,\widehat{\Delta}_n}$ for large enough $n$; see Proposition~\ref{prop:approximate-to-exact-validity}. Therefore, the convex hull based on $B_{\alpha,\widehat{\Delta}_n}$ estimators has an asymptotic miscoverage probability of at most $\alpha$. To avoid conservativeness, one can randomize the number of estimators between $B_{\alpha,\widehat{\Delta}_n}$ and $B_{\alpha,\widehat{\Delta}_n} - 1$ to attain asymptotically exact coverage as shown in Algorithm~\ref{alg:confidence-unknown-Delta}. {\color{black}In other words, the output of Algorithm~\ref{alg:confidence-known-Delta} with $\Delta = \widehat{\Delta}_n$ is the \ahulc interval denoted by $\widehat{\mathrm{CI}}_{\alpha}^{\mathrm{sub}}.$}

\begin{algorithm}[t]
    \caption{Adaptive Confidence Interval with Unknown Median Bias (\ahulcnospace)}
    \label{alg:confidence-unknown-Delta}
    \SetAlgoLined
    \SetEndCharOfAlgoLine{}
    \KwIn{data $X_1, \ldots, X_n$, coverage probability $1 - \alpha$, and an estimation procedure $\mathcal{A}(\cdot)$ that takes as input observations and returns an estimator, subsample size $b$, number of subsamples $K_n$.}
    \KwOut{A confidence interval $\widehat{\mathrm{CI}}_{\alpha}^{\mathrm{sub}}$ such that $\mathbb{P}(\theta_0\in\widehat{\mathrm{CI}}_{\alpha}^{\mathrm{sub}}) \ge 1 - \alpha$ (asymptotically).}
    Draw $K_n$ many subsamples of size $b$ from $X_1, \ldots, X_n$. Apply $\mathcal{A}(\cdot)$ for each subsample and obtain estimators $\widehat{\theta}_b^{(j)}, 1\le j\le K_n$.\;
    Compute the estimator of the (asymptotic) median bias of $\mathcal{A}(\cdot)$ as 
	$\widehat{\Delta}_n := |L_n(0) - {1}/{2}|.$\;
	Use Algorithm~\ref{alg:confidence-known-Delta} with input data $X_1, \ldots, X_n$, coverage probability $1 - \alpha$, the value $\widehat{\Delta}_n$ and the estimation procedure $\mathcal{A}(\cdot)$\;
	\Return the confidence interval obtained as output from Algorithm~\ref{alg:confidence-known-Delta} as $\widehat{\mathrm{CI}}_{\alpha}^{\mathrm{sub}}$.
\end{algorithm}
We now prove bounds on the miscoverage probabilities of the confidence intervals of the \ahulcnospace. The first result provides a bound without using the fact that $\widehat{\Delta}_n$ is obtained from subsampling and then using distributional convergence assumptions, we obtain the final miscoverage bound for $\widehat{\mathrm{CI}}_{\alpha}^{\mathrm{sub}}$. Define ${\Delta}_{n,\alpha}$ as the median bias of $\mathcal{A}(\{X_i:\,i\in S\})$ with ${n}/{(2B_{\alpha,\Delta})}\le |S| \le 2n/B_{\alpha,\Delta}$, i.e.,
\[
{\Delta}_{n,\alpha} ~:=~ \max_{1/2 \le B_{\alpha,\Delta}|S|/n \le 2}\,\mbox{Med-bias}_{\theta_0}(\mathcal{A}(\{X_i:\,i\in S\})).
\]
{\color{black}Note that $\Delta_{n,\alpha}$ in general depends also on $\theta_0$ and the true distribution of the data.}

Consider the following assumption:
\begin{enumerate}[label = \bf(A\arabic*)]
  \item There exists a random variable $W$, a non-decreasing sequence $\{r_m\}_{m\ge1}$, and a non-increasing sequence $\{\delta_m\}_{m\ge1}$ converging to zero such that the estimator $\widehat{\theta}^{(m)}$ obtained by applying $\mathcal{A}(\cdot)$ on $m$ observations satisfies\label{eq:limiting-distribution}
  \[
  \sup_{t\in\mathbb{R}}\left|\mathbb{P}(r_m(\widehat{\theta}^{(m)} - \theta_0) \le t) - \mathbb{P}(W \le t)\right| \le \delta_m.
  \] 
  Further, $0 < \mathbb{P}(W \le 0) < 1$.
\end{enumerate} 
Define the asymptotic median bias of the estimation procedure $\mathcal{A}(\cdot)$ as
% \[
$\Delta := \left(1/2 - \max\{\mathbb{P}(W \le 0), \mathbb{P}(W \ge 0)\}\right)_+.$
% \]
For any $\alpha\in(0, 1)$ and $\Delta \in (0, 1/2)$, define
\[
C_{\alpha,\Delta} := \frac{1}{2}\left[\min\left\{\left(\frac{\alpha}{P(B_{\alpha,\Delta};\Delta)}\right)^{1/B_{\alpha,\Delta}},\,\left(\frac{P(B_{\alpha,\Delta} - 1; \Delta)}{\alpha}\right)^{1/B_{\alpha,\Delta}}\right\} - 1\right],
\]
and for $\Delta = 0$,
\[
C_{\alpha,0} := \frac{1}{B_{\alpha,0}}W_0\left(\frac{1}{\sqrt{2}}\sqrt{\frac{\alpha}{P(B_{\alpha,0}; 0)} - 1}\right).
\]
These quantities are taken from Proposition~\ref{prop:approximate-to-exact-validity} and~\eqref{eq:relaxed-requirement-median-unbiased} which implies that for any $\gamma\in(0, 1/2)$, if $|\gamma - \Delta| \le C_{\alpha,\Delta}$, then $B_{\alpha,\gamma} = B_{\alpha,\Delta}$. Finally, recall that $\widehat{\mathrm{CI}}_{\alpha,\Delta}$ represents the confidence interval returned by the \hulc when the median bias parameter is chosen to be $\Delta$ (irrespective of what the true finite sample median bias is).

To succinctly state our next result we define some additional quantities.
Given an estimate $\widehat{\Delta}_n$ we compute the number of splits, $B_{\alpha,\widehat{\Delta}_n}$. 
We then hypothesize splitting the data twice into $B_{\alpha,\widehat{\Delta}_n}$ and $B_{\alpha,\widehat{\Delta}_n} - 1$ parts with approximately equal number of observations in each split. We then define,
  \[
  \widehat{\mathrm{CI}}_{\alpha}^{(0)} := \left[\min_{1\le j\le B_{\alpha,\widehat{\Delta}_n} - 1}\widehat{\theta}_j, \max_{1\le j\le B_{\alpha,\widehat{\Delta}_n} - 1}\widehat{\theta}_j\right],\quad\mbox{and}\quad \widehat{\mathrm{CI}}_{\alpha}^{(1)} := \left[\min_{1\le j\le B_{\alpha,\widehat{\Delta}_n}}\widehat{\theta}_j, \max_{1\le j\le B_{\alpha,\widehat{\Delta}_n}}\widehat{\theta}_j\right].
  \]
  Here $\widehat{\theta}_j$ are estimators computed based on $\mathcal{A}(\cdot)$. We have the following result:

\begin{thm}\label{thm:subsampling-hull-validity}
Suppose the random variables $X_1, \ldots, X_n$ are independent and assumption~\ref{eq:limiting-distribution} holds true. Then for any $\alpha\in(0, 1)$, the \ahulc confidence intervals satisfy
\begin{equation}\label{eq:fixed-choice-B-unknown-Delta}
\begin{split}
\max\left\{\frac{\mathbb{P}(\theta_0 \notin \widehat{\mathrm{CI}}_{\alpha}^{(0)})}{2},\,\mathbb{P}(\theta_0 \notin \widehat{\mathrm{CI}}_{\alpha}^{(1)})\right\} &\le \mathbb{P}(B_{\alpha,\widehat{\Delta}_n} \neq B_{\alpha,\Delta})
\\
&\quad
+ \alpha\times
\begin{cases}(1 + 2B_{\alpha,0}(B_{\alpha,0} - 1)\Delta_{n,\alpha}^2(1 + 2\Delta_{n,\alpha})^{B_{\alpha,0}}), &\mbox{if }\Delta = 0,\\
(1 + 2|\Delta_{n,\alpha} - \Delta|)^{B_{\alpha,\Delta}}, &\mbox{if }\Delta \neq 0\end{cases},
\end{split}
\end{equation}
and for any $0 \le \eta \le C_{\alpha,\Delta}$,
\begin{equation}\label{eq:random-choice-B-unknown-Delta}
\begin{split}
\left|\mathbb{P}(\theta_0 \notin \widehat{\mathrm{CI}}_{\alpha}^{\mathrm{sub}}) - \mathbb{P}(\theta_0 \notin \widehat{\mathrm{CI}}_{\alpha,\Delta})\right| ~&\le~ 2\mathbb{P}(|\widehat{\Delta}_n - \Delta| \ge \eta)\\ 
&\quad+ 3\alpha\times
\begin{cases}
2\eta^2B_{\alpha,0}^2e^{2\eta B_{\alpha,0}}(1 + 2B_{\alpha,0}^2\Delta_{n,\alpha}^2(1 + 2\Delta_{n,\alpha})^{B_{\alpha,0}}), &\mbox{if }\Delta = 0,\\
2\sqrt{e}\eta B_{\alpha,\Delta}(1 + 2|\Delta_{n,\alpha}-\Delta|)^{B_{\alpha,\Delta}}/(1/2 - \Delta), &\mbox{if }\Delta \neq 0.
\end{cases}
\end{split}
\end{equation}
\end{thm}
Theorem~\ref{thm:subsampling-hull-validity} (proved in Section~\ref{appsec:subsampling-hull-validity}) provides a bound on miscoverage of the confidence interval $\widehat{\mathrm{CI}}_{\alpha}^{\mathrm{sub}}$ from Algorithm~\ref{alg:confidence-unknown-Delta} but does not assume that $\widehat{\Delta}_n$ is obtained from subsampling. 
The miscoverage probabilities of confidence intervals obtained from non-random choices of number of splits $\widehat{\mathrm{CI}}_{\alpha}^{(0)}$ and $\widehat{\mathrm{CI}}_{\alpha}^{(1)}$ only requires controlling the probability of $B_{\alpha,\widehat{\Delta}_n} \neq B_{\alpha,\Delta}$. From Proposition~\ref{prop:approximate-to-exact-validity}, it follows that we do not need $\widehat{\Delta}_n$ to be consistent for $\Delta$. Note that the second term in~\eqref{eq:fixed-choice-B-unknown-Delta} only depends on how close $\Delta_{n,\alpha}$ to $\Delta$ is. 

For the miscoverage probability of $\widehat{\mathrm{CI}}_{\alpha}^{\mathrm{sub}}$ that randomizes the number of splits to avoid overcoverage, we require consistency of $\widehat{\Delta}_n$ to $\Delta$. If $\widehat{\Delta}_n$ is obtained from an independent sample, then we would not require such consistency and can apply Theorem~\ref{thm:infinite-order-coverage} to prove miscoverage. Regarding inequality~\eqref{eq:random-choice-B-unknown-Delta}, we recall from Theorem~\ref{thm:infinite-order-coverage} (and Remark~\ref{rem:miscoverage-Delta}) that $\mathbb{P}(\theta_0 \notin \widehat{\mathrm{CI}}_{\alpha,\Delta})$ can be upper and lower bounded by quantities close to $\alpha$. 
Such lower bounds do not hold true for $\widehat{\mathrm{CI}}_{\alpha}^{(0)}$ and $\widehat{\mathrm{CI}}_{\alpha}^{(1)}$. Finally, because $\widehat{\mathrm{CI}}_{\alpha}^{\mathrm{sub}}$ is a random selection of one of $\widehat{\mathrm{CI}}_{\alpha}^{(0)}$ and $\widehat{\mathrm{CI}}_{\alpha}^{(1)}$, we get
\[
\mathbb{P}(\theta_0 \notin \widehat{\mathrm{CI}}_{\alpha}^{\mathrm{sub}}) \le \mathbb{P}(\theta_0 \notin \widehat{\mathrm{CI}}^{(0)}_{\alpha}) + \mathbb{P}(\theta_0 \notin \widehat{\mathrm{CI}}_{\alpha}^{(1)}),
\] 
and inequality~\eqref{eq:fixed-choice-B-unknown-Delta} can be used to imply that $\widehat{\mathrm{CI}}_{\alpha}^{\mathrm{sub}}$ has an approximate miscoverage probability of $3\alpha$ when $B_{\alpha,\widehat{\Delta}_n} = B_{\alpha,\Delta}$ holds with high probability.

In the multivariate case, one can apply union bound directly on~\eqref{eq:fixed-choice-B-unknown-Delta} with $\alpha$ replaced by $\alpha/d$ to obtain a bound on miscoverage. But using the proof, one can refine this by replacing  $\mathbb{P}(B_{\alpha,\widehat{\Delta}_n} \neq B_{\alpha,\Delta})$  by $\mathbb{P}(B_{\alpha,\widehat{\Delta}_n^{(k)}} \neq B_{\alpha,\Delta^{(k)}} \mbox{ for any } 1\le k\le d)$. Here $\Delta^{(k)}$ is the limiting median bias of the estimator of $k$-th coordinate of $\theta_0$ and $\widehat{\Delta}_n^{(k)}$ is its estimator. Because the second term on the right hand side of~\eqref{eq:fixed-choice-B-unknown-Delta} is multiplicative in $\alpha$, a union bound can be safely applied to obtain a non-trivial guarantee, as in Section~\ref{subsec:multivariate}. Similarly, one can replace the first term on the right hand side of~\eqref{eq:random-choice-B-unknown-Delta} with $\mathbb{P}(|\widehat{\Delta}_n^{(k)} - \Delta^{(k)}| \ge \eta\mbox{ for any }1\le k\le d)$. The second term being multiplicative in $\alpha$ does not affect the applicability of a union bound to obtain a non-trivial bound.

Inequality~\eqref{eq:random-choice-B-unknown-Delta} holds true for all $\eta\in[0, C_{\alpha,\Delta}]$. With $\widehat{\Delta}_n$ a consistent estimator for $\Delta$, one can take $\eta$ converging to zero with sample size. In the following, we will prove a bound on $\mathbb{P}(|\widehat{\Delta}_n - \Delta| \ge \eta)$ when $\widehat{\Delta}_n$ is obtained using subsampling (as in Algorithm~\ref{alg:confidence-unknown-Delta}). It is worth emphasizing that any method of estimating $\Delta$ can be used in Theorem~\ref{thm:subsampling-hull-validity}.

\begin{enumerate}[label = \bf(A\arabic*)]
\setcounter{enumi}{1}
  \item There exists $r^* > 0$ and $\mathfrak{C} < \infty$ such that the distribution function of $W$ satisfies\label{eq:continuity-distribution}
  \[
  0 \le \frac{\mathbb{P}(W \le t) - \mathbb{P}(W \le -t)}{t} \le \mathfrak{C},\quad\mbox{for all}\quad 0 \le t \le r^*.
  \]
  \item The subsample size $b$ satisfies $b/n\to0$ and $r_b/r_n\to0$ as $n\to\infty$. Further, the number of subsamples diverges: $K_n\to\infty$.\label{eq:subsample-size}
\end{enumerate}
These assumptions are similar to those used in the analysis of subsampling. In contrast to the classical analysis of subsampling by \citet{politis1994large} we provide a
finite sample analysis. 
\begin{lem}\label{lem:subsampling-delta-control}
Fix any $t > 0$ such that $r_bt/r_n \le r^*$, then under assumptions~\ref{eq:limiting-distribution},~\ref{eq:continuity-distribution}, and~\ref{eq:subsample-size} 
with probability at least $1 - 2\delta_n - (b + 1)/n - \mathbb{P}(|W| > t)$,
\[
|\widehat{\Delta}_n - \Delta| ~\le~ \sqrt{\frac{\log(2n)}{2K_n}} + \sqrt{\frac{\log(2n/b)}{2[n/b]}} + 2\delta_b + 2\mathfrak{C}\frac{r_bt}{r_n}.
\]
\end{lem}
\noindent The proof follows a similar structure to that of~\cite{politis1994large} and appears in~Section~\ref{appsec:subsampling-delta-control}.

Choosing $t\to\infty$ in Lemma~\ref{lem:subsampling-delta-control} such that $r_bt/r_n\to0$ as $n\to\infty$, we conclude that $|\widehat{\Delta}_n - \Delta| = o_p(1)$; for example, one can take $t = \sqrt{r_n/r_b}$. This combined with Theorem~\ref{thm:subsampling-hull-validity} implies that \ahulc yields an asymptotically valid confidence interval for $\theta_0$ under assumptions~\ref{eq:limiting-distribution},~\ref{eq:continuity-distribution}, and~\ref{eq:subsample-size}.
\section{Applications to non-standard problems}\label{sec:applications-nonparametric-models}
Many commonly used estimators are derived from classical parametric and semi-parametric efficiency theory and {\color{black}after proper normalization} have an asymptotic normal distribution with zero mean. This implies that these estimators have an asymptotic median bias of zero, making them standard problems and allowing for the application of the \hulc with $\Delta = 0$. There do exist estimators that have a non-standard rate of convergence and a non-standard limiting distribution. In this section, we discuss {four} non-standard examples where either the rate of convergence or the limiting distribution or both are unknown in practice. With the rate of convergence unknown, subsampling does not readily apply to yield a confidence interval; one needs to estimate the rate of convergence as in~\cite{bertail1999subsampling}.
\subsection{Squared mean estimation (revisited)}\label{subsec:squared-mean-revisited}
{\color{black}In Section~\ref{subsec:mean-square}, we discussed the application of \hulc in the context of squared mean estimation. In this application, the rate of convergence of the estimator can be $n$ or $\sqrt{n}$ depending on $\mu$ and as $\mu\to0$ as $n\to\infty$ at different rates, other rates of convergence are possible. Figure~\ref{fig:mean-square-subsampling} shows the performance of \hulc (with a conservative median bias bound) and \ahulc and compares them to subsampling. As expected, \hulc yields a conservative confidence interval with coverage at least $1-\alpha$. Subsampling with estimated rate of convergence as in~\cite{bertail1999subsampling} fails to attain correct coverage. This can be due to two reasons. First, we fixed the sample size at 100, which might not be large enough for asymptotics of subsampling. Second, subsampling is not uniformly valid in this example. This means that for each $\mu$ (fixed as $n$ changes), the coverage asymptotically is at least $1-\alpha$ but if $\mu$ is also allowed to change with the sample size, then subsampling asymptotics break down as shown in~\cite{andrews2010asymptotic}. In Figure~\ref{fig:mean-square-subsampling}, this can be seen via the dip in the coverage for $\mu$ close enough to $0$ ($\mu\in[0, 0.25]$ in our setting). It is very interesting and rather surprising to observe that \ahulc maintains coverage for all $\mu$ and is further close to the nominal $1-\alpha$ for $\mu$ close to and farther away from zero. Although subsampling is not valid for construction of confidence intervals, the estimate of median bias from subsampling, from our experiments, seems to be at least as high as the true median bias.}
\subsection{Heavy-tailed mean estimation}\label{subsec:heavy-tailed-mean}
In Section~\ref{subsec:mean-estimation}, we discussed the application of the \hulc in the context of mean estimation when the limiting distribution is symmetric around zero. When the random variables $X_1, \ldots, X_n$ do not have a finite second moment, then the limiting distribution of the sample mean $\widebar{X}_n$ can {\color{black}fail to be} symmetric around the population mean $\mu$ with the amount of asymmetry depending on the tail decay on either side of $\mu$. In this case, the rate of convergence also depends on tail decay and is unknown a priori, which makes subsampling inapplicable. See~\cite{romano1999subsampling} for an application of subsampling using the studentized statistic, which does not require estimating the rate of convergence. Without knowing the rate of convergence, we can apply Algorithm~\ref{alg:confidence-unknown-Delta} to obtain an estimate of the median bias and create a confidence interval for the population mean. In this case, provided that the median bias
is not too close to $1/2$ the \ahulc will yield non-trivial confidence intervals. 
\subsection{Shape constrained regression}\label{subsec:shape-constrained}
Constructing confidence intervals in the context of general non-parametric regression is a difficult task.
In order to obtain optimal estimation rates we aim to explicitly balance (squared) bias and variance. On the other hand, the exact bias is often intractable and difficult to account for in confidence interval construction. As a consequence, often under-smoothing is used to ensure that the squared bias is negligible compared to the variance asymptotically. In practice, however, under-smoothing can be sensitive to the precise choice of tuning parameters.

If the conditional mean function is assumed to satisfy a shape constraint such as monotonicity or convexity, then the least squares estimator of the conditional mean has negligible bias uncomplicating the inference problem. However, the rate of convergence and the limiting distribution depends on the local smoothness of the conditional mean. To be concrete, consider the setting of univariate monotone regression with equi-spaced design, i.e., $Y_i = f_0(i/n) + \varepsilon_i$ where $\varepsilon_i$s are independent and identically distributed with mean zero and finite variance $\sigma^2 > 0$, and $f_0$ is our shape constrained target. Consider the least squares estimator (LSE) of $f_0$ as
\[
\widehat{f}_n ~:=~ \argmin_{f:\,\mbox{increasing}}\,\frac{1}{n}\sum_{i=1}^n (Y_i - f(i/n))^2.
\]
Note that $\widehat{f}_n(\cdot)$ is defined uniquely only at $i/n, 1\le i\le n$, and is, conventionally, defined to as a piecewise constant increasing function on $[0, 1]$. 
In this setting, for any $t\in(0, 1)$ such that $f_0(\cdot)$ has a positive continuous derivative on some neighborhood of $t$, the LSE satisfies
% \[
$n^{1/3}(\widehat{f}_n(t) - f_0(t)) \overset{d}{\to} [4\sigma^2f_0'(t)]^{1/3}\mathbb{C},$
% \]
where $\mathbb{C} = \argmin_{h\in\mathbb{R}}\{\mathbb{W}(h) + h^2\}$ has Chernoff's distribution (here $\mathbb{W}(\cdot)$ is a two-sided Brownian motion starting from $0$). It is important here that $f_0'(t) \neq 0$. If $f_0^{(j)}(t) = 0$ for $1\le j\le p-1$ and $f_0^{(p)}(t) \neq 0$ (for $p\ge1$), where $f_0^{(j)}$ denotes the $j$-th derivative, then $n^{p/(2p + 1)}(\widehat{f}_n(t) - f_0(t))$ converges to a non-degenerate distribution depending on $f^{(p)}_0(t)$ and $\sigma^2$. Finally, if $f_0(\cdot)$ is flat at $t$, then the rate of convergence becomes $n^{1/2}$. These results are known in both the fixed and random design settings (see for instance, \cite{wright1981asymptotic, durot2008monotone,guntuboyina2018nonparametric}).
These rates of convergences imply that the LSE admits an adaptive behavior and for arbitrary monotone functions and consequently it is unclear how to perform inference. The situation becomes more complicated in the multi-dimensional case where the limiting distribution depends on the anisotropic smoothness of $f_0$. Recently~\cite{deng2020confidence} proved that the rates of convergence along with the nuisance parameters in the limiting distributions can be estimated consistently using $\widehat{f}_n$. This theory requires substantial new techniques and still requires estimation of $\sigma^2$. Alternatively, we can use the \ahulc in all of these cases to obtain asymptotically valid confidence intervals. It is worth mentioning that in most of these settings, the median bias of the limiting distribution is also unknown because it depends on the unknown local smoothness. The same discussion also holds true for other shape constrained models such as convex regression and current status regression; see~\citet[Section 4]{guntuboyina2018nonparametric} and~\cite{deng2020inference} for details.

\begin{figure}[h]
\centering
\includegraphics[width=\textwidth,height=2.8in]{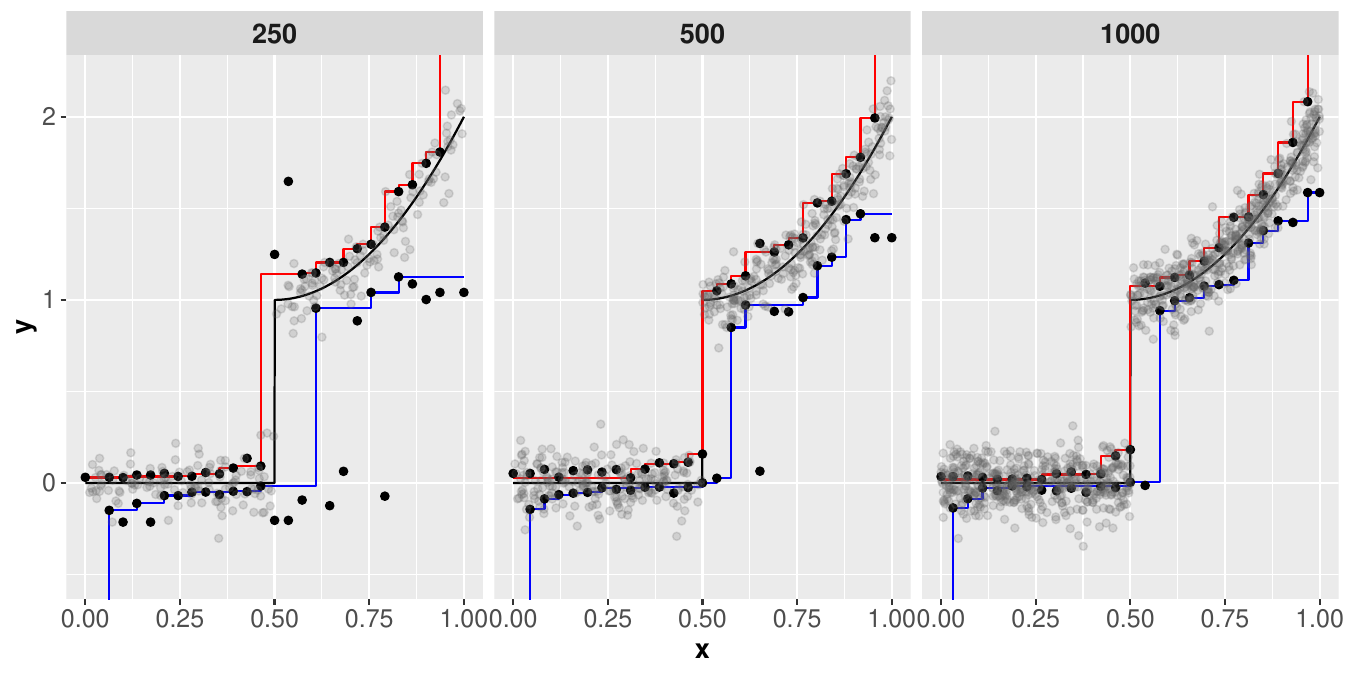}
\vspace{-0.4in}
\caption{Confidence bands for a monotone conditional mean function as sample size changes from 250 to 1000. The black line shows the true function which is a constant on $[0, 0.5]$ and is a (strictly increasing) quadratic on $[0.5, 1]$. The LSE attains an $n^{1/2}$ rate on $[0, 0.5]$ and an $n^{1/3}$ rate on $[0.5, 1]$. The \ahulc simultaneous confidence intervals at $25$ equi-spaced points on $[n^{-1/2},1-n^{-1/2}]$ are shown as dark black points. The confidence band obtained via~\eqref{eq:interval-to-band-monotone} is shown as red and blue lines. The obtained sample for each sample size is plotted in gray.}
\label{fig:CI-monotone-subsampling}
\end{figure}

We note that for shape constrained regression, it is possible to obtain confidence bands from confidence intervals at several points on the domain. For example if we know $\ell(t_1) \le f_0(t_1) \le u(t_1)$ and $\ell(t_2) \le f_0(t_2) \le u(t_2)$ for two points $t_1, t_2\in[0,1]$ in the domain, then using the information {\color{black}that} $f_0(\cdot)$ is non-decreasing we can conclude that $\widebar{\ell}(t) \le f_0(t) \le \widebar{u}(t)$ for all $t\in[0, 1]$, where
\begin{equation}\label{eq:interval-to-band-monotone}
\widebar{\ell}(t) = \begin{cases}-\infty,&\mbox{for }t < t_1,\\
\ell(t_1),&\mbox{for }t_1 \le t < t_2,\\
\ell(t_2), &\mbox{for }t_2 \le t \le 1,\end{cases}\quad\mbox{and}\quad \widebar{u}(t) = \begin{cases}u(t_1),&\mbox{for }t \le t_1,\\
u(t_2),&\mbox{for }t_1 < t \le t_2,\\
\infty, &\mbox{for }t_2 < t \le 1.\end{cases}
\end{equation}
Of course, the more points at which confidence intervals are available, the better the confidence band is. Figure~\ref{fig:CI-monotone-subsampling} shows the simultaneous confidence intervals obtained from the \ahulc (with subsample size $b=n^{1/2}$ and $K_n=1000$) for a monotone conditional mean from observations $Y_i = f_0(X_i) + \varepsilon_i$ where $X_i\sim \mbox{Unif}[0, 1], \varepsilon_i\sim N(0, 0.1^2)$ and $f_0(x) = 1-\mathbbm{1}\{x\le0.5\}+ ((x-0.5)/0.5)^2\mathbbm{1}\{x>0.5\}$. {\color{black}The choice $b = n^{1/2}$ can be obtained by minimizing the bound in Lemma~\ref{lem:subsampling-delta-control} with $r_b = 1/\delta_b = b^{-\gamma}$, which with $\gamma = 1/3$ or $1/2$ represents the rate of convergence of the monotone LSE; see~\cite{han2019berry}.} This figure only shows one replication of the experiment and suggests that the width of the band seems to adapt to the local smoothness of $f_0$. {\color{black}In the following subsection, we present one simulation setting for monotone regression comparing the coverage and width of \hulc with subsampling. For more simulations comparing the coverage and width of \hulc and \ahulc in shape constrained problems, see \url{https://github.com/Arun-Kuchibhotla/HulC/blob/main/R/HulC%20for%20Shape%20Constrained%20Regression.ipynb}.}
\subsubsection{Pointwise Confidence Interval for Monotone Regression}
{\color{black}
Suppose $(X_i, Y_i), 1\le i\le n$ are independent and identically distributed observations satisfying
\[
Y_i = f_0(X_i) + \xi_i,\quad f_0(x) = |x|^{\beta}\mbox{sgn}(x), X_i\sim \mbox{Unif}[-1, 1], \xi_i|X_i \sim N(0, 1).
\]
Clearly, $f_0$ is a monotonically increasing function on $[-1, 1]$ and $|f_0(x) - f_0(0)| = |x|^{\beta}$ which implies that $f_0$ is locally $\beta$-smooth at $0$. From the results of~\cite{wright1981asymptotic}, it follows that the LSE $\widehat{f}_n$ satisfies $n^{\beta/(2\beta + 1)}(\widehat{f}_n(0) - f_0(0))$ converges in distribution to $\mathbb{C}_{\beta}/(\beta+1)^{1/(2\beta + 1)}$ where $\mathbb{C}_{\beta}$ is the slope at zero of the greatest convex minorant of $W(t) + |t|^{\beta + 1}$ with $W(\cdot)$ being the two-sided Weiner-Levy process with variance one per unit time. Hence, the rate of convergence and the limiting distribution depends curcially on the local smoothness parameter $\beta$. Experimentally, we verified that $\mathbb{C}_{\beta}$ is equally likely to be positive or negative and hence, suggests that $\widehat{f}_n(0)$ is asymptotically median unbiased for $f_0(0)$ no matter the value of $\beta$. With this experimental backing, we applied \hulc for this setting and also compared with the performance of subsampling procedure of~\cite{bertail1999subsampling} that estimates the rate of convergence with different choices of subsample sizes. Figure~\ref{fig:isotonic_comparison} shows the comparison of the coverage and width as the sample size changes from 50 to 1000, based on 200 Monte Carlo replications for each sample size and each $\beta$.
\begin{figure}[!h]
\centering
\includegraphics[height=3.2in,width=\textwidth]{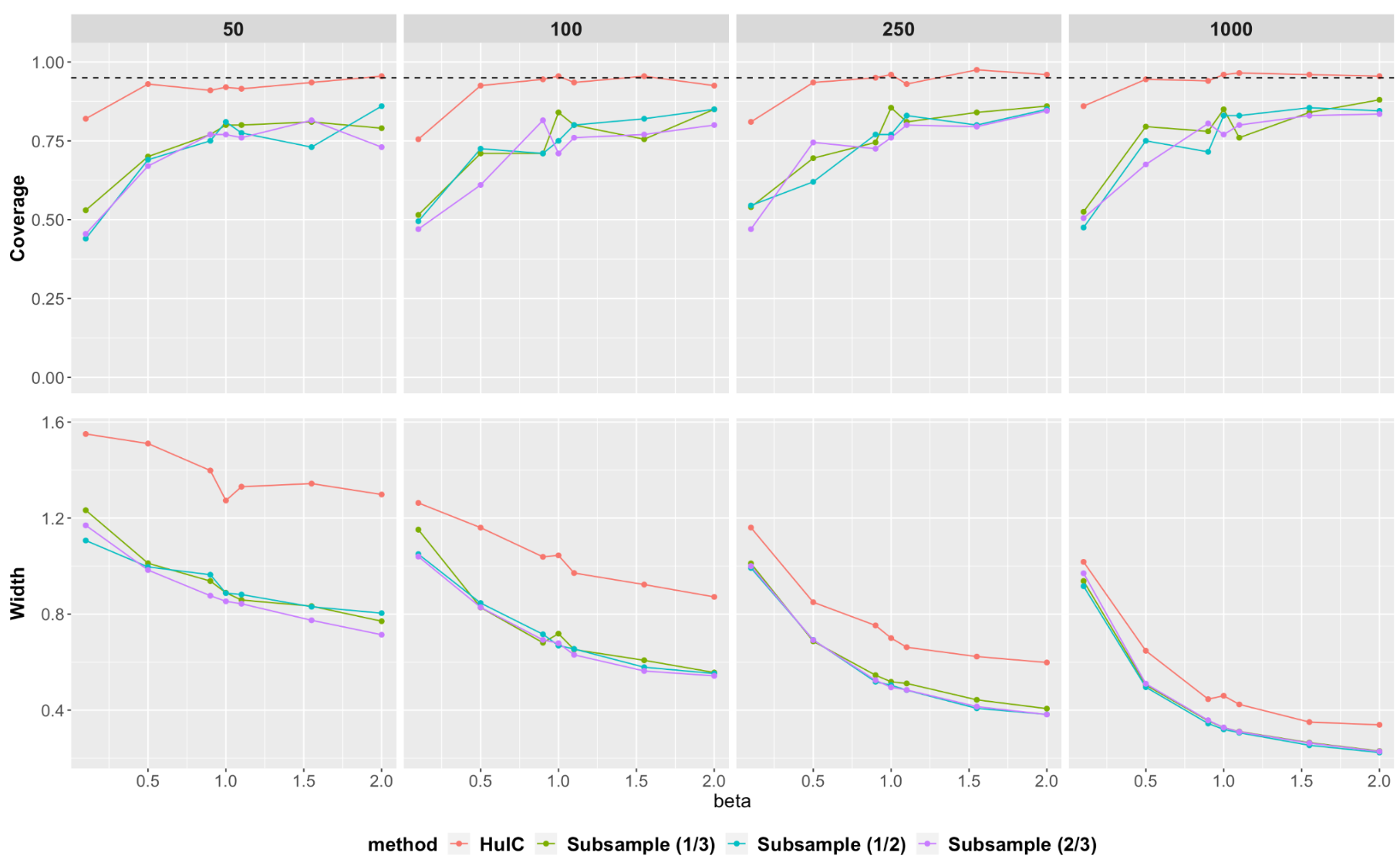}
\caption{Comparison of \hulc and subsamplng with estimated rate of convergence in isotonic regression under varying local smoothness. The sample size is mentioned at the top of each plot and the smoothness parameter $\beta$ is on the $x$-axis. \hulc maintains the coverage at the nominal level of $0.95$ for almost all $\beta$ for all sample sizes, while subsampling with any subsample size fails.}
\label{fig:isotonic_comparison}
\end{figure}
}

\subsection{Nonparametric Regression and Forests}

The \hulc also yields
confidence intervals
for nonparametric regression
even in the presence of unknown
asymptotic {\color{black}mean} bias. We briefly sketch the main ideas, deferring most of the details to future work. We focus on constructing a confidence interval for
the non-parametric regression function $f_0$ at a fixed point $x_0\in\mathbb{R}$.
For example, let
$\widehat{f}_n(x_0)$
be a kernel regression estimator
with bandwidth $h$.
If $h=h_n$ is chosen to balance bias and variance
then
$\sqrt{n h_n}(\widehat{f}_n(x_0)-f_0(x_0))$
converges to a Gaussian law
with mean
$Q = \lim_{n\to\infty} \sqrt{n h_n} \E[\widehat{f}_n(x_0) - f_0(x_0)]$
(which is the asymptotic {\color{black}mean} bias), and finite, non-zero variance. 
As we discussed earlier, classical methods often rely on undersmoothing to ensure that $Q = 0$. However, the \ahulc works as long as $Q$ is finite, since in this case 
the asymptotic median bias is bounded away from $1/2$. We also emphasize that, in contrast to undersmoothing for which there 
are relatively few guidelines on practical implementation, it is more conventional in non-parametric regression to balance (squared) bias and variance, and in many cases
cross-validation methods yield tuning parameters which achieve this balancing under various conditions (see for instance, Theorem 2.2 in~\cite{li2004cross}).

The argument above is not specific to kernel regression estimators. More generally,
let $\widehat{f}_n(x_0)$
be a complicated nonparametric estimator
such as a random forest. The \ahulc yields a valid interval for $f_0(x_0)$ provided that we are able to balance (squared) bias and variance, i.e. so long as we can ensure that for
some possibly unknown 
$r_n$,
we have that
$r_n (\widehat{f}_n(x_0) - f_0(x_0))$
converges to a Gaussian law with possibly
non-zero (but finite) mean, and non-zero, finite variance.

\section{Confidence Regions under Unimodality}\label{sec:unimodality-lanke}
In previous sections, we have considered the construction of confidence intervals based on the median bias of the estimation procedure. In some cases, the estimation procedure has a large
median bias close to $1/2$. For example in the mean square estimation problem, $\widebar{X}_n^2$ has a median bias of $1/2$ when $\mu = 0$ and in the uniform model, the MLE has a median bias of $1/2$. In these cases, 
the \hulc and \ahulc are not useful because they would require infinite splits of the data. Interestingly, in these examples, the limiting distribution of the estimation procedure is unimodal at the true parameter. In the univariate setting, unimodality {\color{black}of an estimator} at $\theta_0$ means that the distribution function {\color{black}of the estimator} is convex on $(\infty, \theta_0]$ and concave on $[\theta_0, \infty)$. {\clr It is important to note that unimodality of a random variable is a global property of the distribution function unlike median bias, which is a local property.}

Using the results of~\cite{lanke1974interval}, we can construct a confidence interval based on unimodality of the estimation procedure. The resulting confidence interval is very similar to the one from the \hulcnospace. 
The \uhulc method is presented in Algorithm~\ref{alg:confidence-unimodal}.
\begin{algorithm}[h]
    \caption{Confidence Interval based on Unimodality and Median Bias (\uhulcnospace)}
    \label{alg:confidence-unimodal}
    \SetAlgoLined
    \SetEndCharOfAlgoLine{}
    \KwIn{data $X_1, \ldots, X_n$, coverage probability $1 - \alpha$, a parameter $t>0$ and an estimation procedure $\mathcal{A}(\cdot)$ that yields estimators asymptotically unimodal at $\theta_0$ and have an asymptotic median bias of $\Delta\in[0, 1/2]$.}
    \KwOut{A confidence interval $\widehat{\mathrm{CI}}_{\alpha}^{\mathrm{mode}}$ such that $\mathbb{P}(\theta_0\in\widehat{\mathrm{CI}}_{\alpha}^{\mathrm{mode}}) \ge 1 - \alpha$ asymptotically.}
    Set $Q(B; t, \Delta) := P(B;\Delta)(1 + t)^{-B + 1},$
    and find the smallest integer $B = B_{\alpha,t,\Delta}\ge1$ such that $Q(B; t, \Delta) \le \alpha$. Recall $P(B; \Delta)$ from~\eqref{eq:miscoverage-function-B-Delta}.\;
    % \hspace{0.1in} 
    Generate a random variable $U$ from Uniform distribution on $[0,1]$ and set
    \begin{equation}\label{eq:definition-eta}
    \eta_{\alpha,t} ~:=~ \frac{Q(B_{\alpha,t,\Delta}-1;t,\Delta) - \alpha}{Q(B_{\alpha,t,\Delta}-1;t,\Delta) - Q(B_{\alpha,t,\Delta};t,\Delta)}\quad\mbox{and}\quad B^* := \begin{cases} B_{\alpha,t,\Delta}, &\mbox{if }U \le \eta_{\alpha,t},\\
    B_{\alpha,t,\Delta} - 1, &\mbox{if }U > \eta_{\alpha,t}.
    \end{cases} 
    \end{equation}\vspace{-0.2in}\;
    % \hspace{0.1in} 
  	Randomly split the data $X_1, \ldots, X_n$ into $B^*$ many disjoint sets $\{\{X_i:i\in S_j\}: 1\le j\le B^*\}$ of approximately equal sizes.\;
  	Compute estimators $\widehat{\theta}_j := \mathcal{A}(\{X_i:\,i\in S_j\})$, for $1\le j\le B^*$ and set
  	\[
  	\widehat{\theta}_{\max} ~:=~ \max_{1\le j\le B^*}\widehat{\theta}_j,\quad\mbox{and}\quad \widehat{\theta}_{\min} ~:=~ \min_{1\le j\le B^*}\widehat{\theta}_j.
  	\]\vspace{-0.2in}\;
      \Return the confidence interval 
      % \begin{equation}\label{eq:confidence-interval-known-Delta}
      $\widehat{\mathrm{CI}}_{\alpha}^{\mathrm{mode}} := [\widehat{\theta}_{\min} - t(\widehat{\theta}_{\max} - \widehat{\theta}_{\min}),\, \widehat{\theta}_{\max} + t(\widehat{\theta}_{\max} - \widehat{\theta}_{\min})].$
      % \end{equation}
\end{algorithm}

The \uhulc can be seen as a generalization of the \hulc where we also use the unimodality of estimators, if available. {\color{black}\uhulc involves a tuning parameter $t$ indicates how much to inflate the convex hull confidence interval.} Taking $t = 0$ in the \uhulc gives exactly the \hulcnospace. The confidence interval with $t = 0$ need not have coverage validity if the asymptotic median bias is $1/2$ and by taking $t > 0$, we get asymptotic coverage when the limiting distribution is unimodal even if the asymptotic median bias is $1/2$. {\color{black}Even if $\Delta < 1/2$, using $t > 0$ yields a reduction in the number of estimators used for convex hull. Formally, set $Q(B; t, \Delta) := P(B;\Delta)(1 + t)^{-B + 1}$, which is upper bounded by $P(B; \Delta)$. Then $\uhulc$ only requires $B_{\alpha,t,\Delta}$ many independent estimators, where $B_{\alpha,t,\Delta}$ is the smallest integer such that $Q(B_{\alpha,t,\Delta}; t, \Delta) \le \alpha$.} 

In the \uhulcnospace, we assume that the limiting median bias $\Delta$ is known, but one can always substitute $\Delta = 1/2$ if median bias is unknown; recall that $P(B;1/2) = 1$ for all $B$. 
Instead of assuming a known $\Delta$, one can use the subsampling approach from Section~\ref{sec:adaptive-convex-hull} to replace $\Delta$ with the subsampling estimator. We leave it to future work to derive a final miscoverage bound for this subsampling-based  procedure.

The following theorem (proved in Section~\ref{appsec:unimodal-CI-coverage}) shows that the confidence interval returned by the \uhulc has a miscoverage probability bounded asymptotically by $\alpha$. 
\begin{thm}\label{thm:unimodal-CI-coverage}
Suppose the estimators $\widehat{\theta}_j$ in the \uhulc are independent and are constructed based on approximately equal sized samples. Further, suppose the estimators are continuously distributed and satisfy
\begin{align}
\label{eqn:near-unimodal}
\sup_{u\in\mathbb{R}}|\mathbb{P}(r_{n,\alpha}(\widehat{\theta}_j - \theta_0) \le u) - \mathbb{P}(W \le u)| ~\le~ \delta_{n,\alpha},
\end{align}
for some sequence $\{r_{n,\alpha}\}_{n\ge1}$ and a continuous random variable $W$ that is unimodal at $0$ and has a median bias of $\Delta$ (i.e., $\Delta = |1/2 - \mathbb{P}(W \le 0)|$). Then for all $t \ge 0$, $\Delta\in[0, 1/2]$, and $\alpha\in[0, 1]$,
\begin{equation}\label{eq:coverage-mode}
\mathbb{P}\left(\theta_0 \notin \widehat{\mathrm{CI}}_{\alpha}^{\mathrm{mode}}\right) ~\le~ \alpha\left(1 - 10B_{\alpha,t,\Delta}(1 + t)\delta_{n,\alpha}\right)_+^{-1}.
\end{equation}
\end{thm}
Similar to Theorem~\ref{thm:infinite-order-coverage}, Theorem~\ref{thm:unimodal-CI-coverage} shows that the confidence interval from the \uhulc {\color{black}has its miscoverage probability bounded by $\alpha$}
% attains the required miscoverage probability 
up to a multiplicative error. Once again, this is unlike the coverage guarantee for Wald's interval. 
Because of the multiplicative error, the guarantee from Theorem~\ref{thm:unimodal-CI-coverage} is also suitable for an application of the union bound to obtain a valid multivariate confidence region, as discussed previously in Section~\ref{subsec:multivariate}.

Note that the right hand side of~\eqref{eq:coverage-mode} is finite if and only if $10B_{\alpha,t,\Delta}(1 + t)\delta_{n,\alpha} < 1$. It is easy to prove that $B_{\alpha,t,\Delta} = O(\log(1/\alpha))$ when either $t > 0$ or $\Delta < 1/2$ and in many cases, $\delta_{n,\alpha} = O(\sqrt{\log(1/\alpha)/n})$. Hence, the condition for finiteness would hold true as long as $n \gg \log^3(1/\alpha)$; this is similar to the requirement in the \hulcnospace. The importance of the \uhulc stems from its ability to tackle problems where the median bias of the estimator is large (near $1/2$). 

The width of the confidence interval $\widehat{\mathrm{CI}}_{\alpha}^{\mathrm{mode}}$ is given by $(1 + 2t)(\widehat{\theta}_{\max} - \widehat{\theta}_{\min})$. This is $1 + 2t$ times larger than the width of confidence interval from the \hulcnospace. The confidence interval has
{\color{black}asymptotically valid coverage of at least $1-\alpha$}
 % a coverage 
for any parameter $t \ge 0$ and as $t$ increases, the number of splits $B$ in the \uhulc decreases leading to a smaller value of $\widehat{\theta}_{\max} - \widehat{\theta}_{\min}$. Similar to the map $\Delta \mapsto B_{\alpha,\Delta}$, the map $(t,\Delta)\mapsto B_{\alpha,t,\Delta}$ is a piecewise constant function. 

\subsection{Standard problems (revisited)}
The \uhulc can make use of both asymptotic unimodality and asymptotic median unbiasedness which holds true for most of the standard problems where the limiting distribution is Gaussian (a symmetric unimodal distribution). In many of the examples discussed in Section~\ref{sec:applications-parametric-models}, one can use the \uhulc to (potentially) obtain a tighter confidence interval. Note that $B_{\alpha,t,\Delta}$ in the \uhulc is always smaller than $B_{\alpha,\Delta}=B_{\alpha,0,\Delta}$ in the \hulcnospace. Once again the advantage is that we do not need to estimate the limiting variance of the estimators being used and need not know the rate of convergence.  
\subsection{Application 2: shape constrained regression (revisited)}
In Section~\ref{subsec:shape-constrained}, we used subsampling to estimate the median bias of the LSE in shape constrained regression. Experimentally, we found that the distribution of the LSE is unimodal at the true value. 
Consider the regression problem $Y_i = f_0(i/n) + \varepsilon_i$ where $\varepsilon_i\sim N(0, 1)$ and $f_0(x) \equiv 0$. The histograms of the LSE error $\widehat{f}_n(x_0) - f_0(x_0)$ {for $x_0$ ranging from $0.1$ to $0.9$} are shown in Figure~\ref{fig:isoplot} when the estimator is computed based on $10^6$ samples and 1000 replications. {\color{black}Figure~\ref{fig:isoplot} indicates that the distribution of $\widehat{f}_n(x_0) - f_0(x_0)$ has a unique mode and that mode is close to zero, and this property is unaffected by the location of $x_0\in[0, 1]$.}
\begin{figure}[!h]
    \centering
    \includegraphics[height=3in,width=\textwidth]{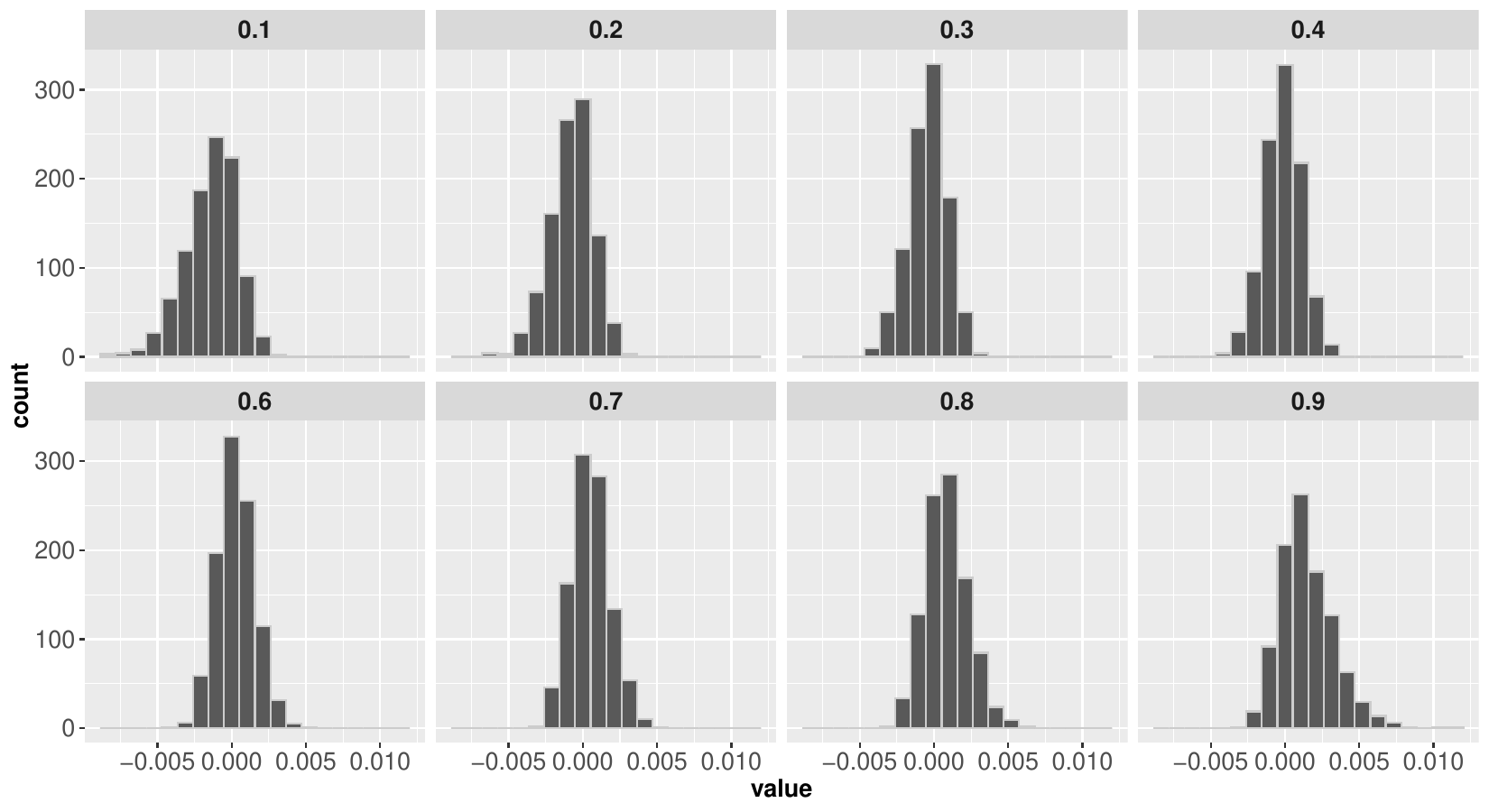}
    \vspace{-0.4in}
    \caption{Histogram of (centered) monotone LSE $\widehat{f}_n(x_0) - f_0(x_0)$ at $x_0 \in \{0.1, 0.2, \ldots, 0.9\}$ when the true function is identically zero. Clearly, the distributions are asymmetric but the mode is at the right place (0). The distribution at $0.1$ is left skewed and the one at $0.9$ is right skewed. The farther we move from the center of the support $[0, 1]$, the more asymmetric the distribution becomes.} 
    \label{fig:isoplot}
\end{figure}
We are not aware of a result proving unimodality of the limiting distribution of the LSE in general (when higher derivatives of $f_0$ may vanish at $x_0$). But, motivated by our experimental results, we apply the \uhulc to construct a confidence band, and leave a more rigorous investigation of its validity to future work.

The performance of the \uhulc for monotone regression is shown in Figure~\ref{fig:isotone-confidence}{\color{black}; for illustration, we use $t = 1/2$}. It shows adaptation and shows higher uncertainty around the change point. The confidence band here is noticeably larger than the one from the \ahulcnospace.
\begin{figure}[!h]
    \centering
    \includegraphics[width=\textwidth,height=2.8in]{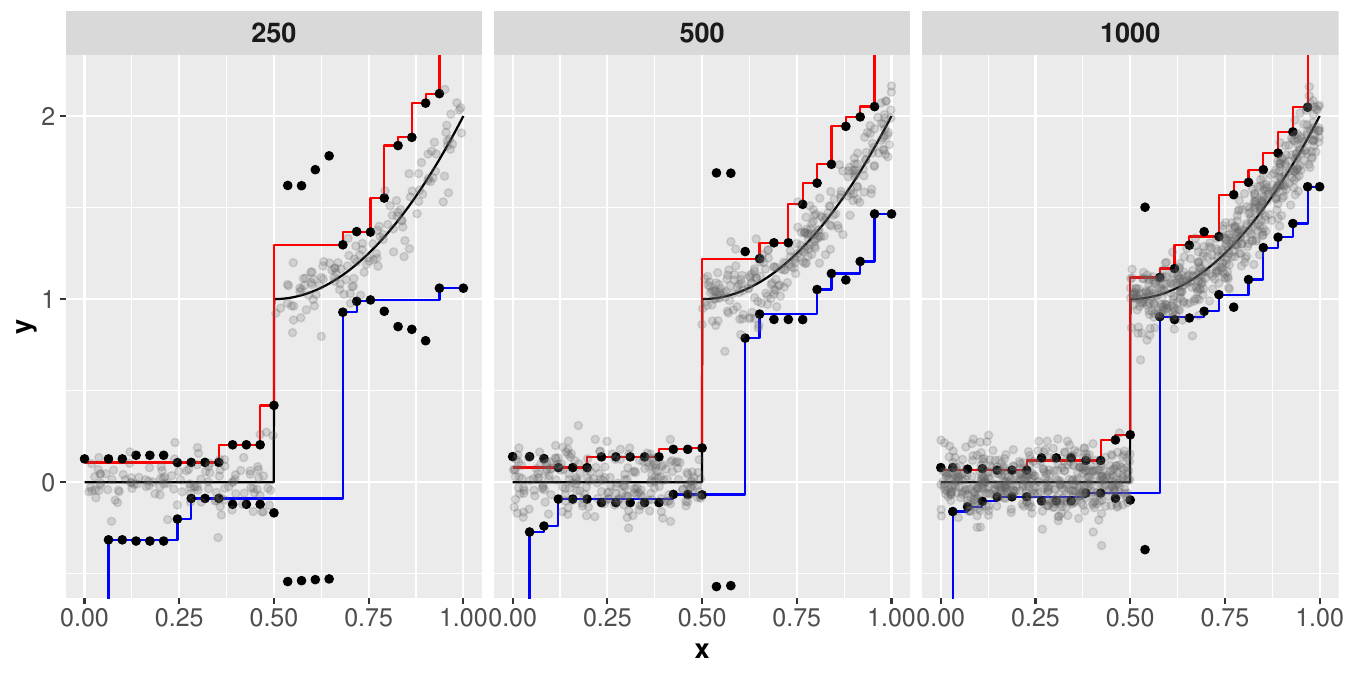}
    \vspace{-0.4in}
    \caption{Performance of the \uhulc (with $t = 1/2$ and $\Delta =1/2$) as sample size increases. The black line is the true monotone function which is a constant $0$ on $[0, 0.5]$ and is smooth on $[0.5, 1]$. The confidence intervals using the \uhulc along with union bound at 25 equi-spaced points on $[n^{-1/2}, 1 - n^{-1/2}]$ are shown in black points. The confidence bands from these simultaneous confidence intervals using~\eqref{eq:interval-to-band-monotone} are shown in red and blue.}
    \label{fig:isotone-confidence}
\end{figure}

\section{\hulc for multivariate parameters}\label{subsec:multivariate}
Suppose now that
$\theta_0\in\mathbb{R}^d$.
A slight modification of the \hulc still works if we replace the interval
with either the convex hull
or the rectangular hull
of the estimators.

As before, let $\widehat{\theta}_j, 1\le j\le B$ be independent estimators of
$\theta_0\in\mathbb{R}^d$.
The convex hull of a set of points in $\mathbb{R}^d$ is
the smallest convex set containing these points. The smallest
rectangle containing the estimators $\widehat{\theta}_j, 1\le j\le B$,
which we call the \emph{rectangular hull},
is
\[
\mbox{RectHull}(\{\widehat{\theta}_j: 
1\le j\le B\}) ~:=~ \bigotimes_{k=1}^d 
\left[\min_{1\le j\le B}e_k^{\top}\widehat{\theta}_j,\,\max_{1\le j\le B}e_k^{\top}\widehat{\theta}_j\right],
\] 
where $e_k, 1\le k\le d$ represent the canonical basis vectors in
$\mathbb{R}^d$; and $\bigotimes$ denotes the Cartesian product. 

To compactly state our next result, we define the following coordinate-wise maximum median bias,
\begin{align}
\label{eq:coordinate-max-median-bias}
\Delta = \max_{1\le k\le d}\max_{1\le j\le B}\mbox{Med-bias}_{e_k^{\top}\theta_0}(e_k^{\top}\widehat{\theta}_j).
\end{align}

Similar to Lemma~\ref{lem:fixed-B-coverage-result}, we have the following result (proved in Section~\ref{appsec:fixed-B-coverage-result-multivariate}) on the coverage of the convex hull and the smallest rectangle.
\begin{lem}\label{lem:fixed-B-coverage-result-multivariate}
Suppose $\widehat{\theta}_j, 1\le j\le B$ are independent estimators of $\theta_0\in\mathbb{R}^d$.
\begin{enumerate}
  \item If $\mathbb{P}(c^{\top}(\widehat{\theta}_j - \theta_0) \le 0) = 1/2$ for all $c\in\mathbb{R}^d\setminus\{0\}$, then for $B \ge d + 1$,
  \begin{equation}\label{eq:half-space-symmetry-coverage}
  \mathbb{P}\left(\theta_0 \notin \mathrm{ConvHull}(\{\widehat{\theta}_j: 1\le j\le B\})\right) ~=~ \frac{1}{2^{B-1}}\sum_{i=0}^{B-d-1}\binom{B-1}{i}.
  \end{equation}
  \item Recall the definition of $\Delta$ in~\eqref{eq:coordinate-max-median-bias}. For all $B \ge 1$,
  \begin{equation}\label{eq:each-coordinate-median-bias-coverage}
  \mathbb{P}\left(\theta_0 \notin \mathrm{RectHull}(\{\widehat{\theta}_j: 1\le j\le B\})\right) ~\le~ d\left\{\left(\frac{1}{2} + \Delta\right)^B + \left(\frac{1}{2} - \Delta\right)^B\right\}.
  \end{equation}
\end{enumerate} 
\end{lem} 
The proof of~\eqref{eq:half-space-symmetry-coverage} follows from the works of~\cite{wendel1962problem} and~\cite{wagner2001continuous}. The requirement of more than $d + 1$ estimators can be restrictive in practice. This is especially so in near high dimensional problems where the dimension $d$ can grow with the sample size $n$. The proof of~\eqref{eq:each-coordinate-median-bias-coverage} follows by using the union bound on the univariate confidence region in Lemma~\ref{lem:fixed-B-coverage-result}. Furthermore, to obtain~\eqref{eq:each-coordinate-median-bias-coverage} we only assume that the coordinate-wise median bias of the estimators is bounded and this condition is much weaker than the corresponding condition used 
to obtain~\eqref{eq:half-space-symmetry-coverage}.

Inequality~\eqref{eq:each-coordinate-median-bias-coverage} is written with a single parameter $\Delta$ as a bound on the median bias for all coordinates.
It is, however, easy to obtain similar bounds when the median bias is different for different coordinates; see the proof of Lemma~\ref{lem:fixed-B-coverage-result-multivariate} for details. Similarly, we also note that, one need not compute $B$ random vector estimators. One might construct a different number of estimators for $e_k^{\top}\theta_0$ and construct univariate confidence intervals along each coordinate to obtain a multivariate confidence rectangle for $\theta_0$. Formally, if $\widehat{\mathrm{CI}}_{\alpha/d,\Delta}^{(k)}$ is a confidence interval of level $1 - \alpha/d$ for $e_k^{\top}\theta_0$ constructed using Algorithm~\ref{alg:confidence-known-Delta} with (a known) $\Delta$, then 
\begin{equation}\label{eq:union-bound-confidence-region}
\mathbb{P}\left(\theta_0 \in \bigotimes_{k=1}^d \widehat{\mathrm{CI}}_{\alpha/d,\Delta}^{(k)}\right) \ge 1 - \alpha.
\end{equation}
Note that construction of $\widehat{\mathrm{CI}}_{\alpha/d,\Delta}^{(k)}$ requires $B_{\alpha/d,\Delta}$ estimators of $e_k^{\top}\theta_0$. Following inequalities~\eqref{eq:inequalities-B-alpha-Delta}, we conclude that
\[
\max\left\{\left\lceil\frac{\log(d/\alpha)}{\log(2/(1 + 2\Delta))}\right\rceil, \left\lceil\frac{\log(2d/\alpha)}{\log(2)}\right\rceil \right\} \le B_{\alpha/d,\Delta} \le \left\lceil\frac{\log(2d/\alpha)}{\log(2/(1 + 2\Delta))}\right\rceil\quad\Rightarrow\quad B_{\alpha/d,\Delta} \asymp \log(2d/\alpha).
\]
This implies that one only needs to split the original data $X_1, \ldots, X_n$ into (about) $\log(2d/\alpha)$ many batches. In Lemma~\ref{lem:fixed-B-coverage-result-multivariate},~\eqref{eq:each-coordinate-median-bias-coverage} requires $B \ge C\log(2d/\alpha)$ for a constant $C$ for a coverage of $1 - \alpha$. This can be compared with the requirement $B \ge d + 1$ for the validity of~\eqref{eq:half-space-symmetry-coverage}. Hence, for moderate to high dimensional problems, the smallest rectangle is an economical choice. 

Similar to the univariate case, one need not know median bias of $e_k^{\top}\widehat{\theta}_j$ exactly. It suffices to know it approximately as dictated by Proposition~\ref{prop:approximate-to-exact-validity}. The conclusions from Proposition~\ref{prop:approximate-to-exact-validity} continue to hold true even with a growing dimension. For instance, for asymptotically median unbiased estimators, if
\begin{equation}\label{eq:dimension-growing-sample-size-requirement}
\mathfrak{C}'_X\frac{\log^3(2d/\alpha)}{n} ~\le~ \frac{\alpha}{dP(B_{\alpha/d,0}; 0)} - 1,
\end{equation}
for some constant $\mathfrak{C}'_X$, then irrespective of the dimension $d\ge1$, we obtain
\[
\mathbb{P}\left(\theta_0 \in \mbox{RectHull}(\{\widehat{\theta}_j:\,1\le j\le B_{\alpha/d,0}\})\right) \ge 1 - \alpha.
\] 
It follows from Figures~\ref{fig:right-hand-side-of-consistency-subsampling-scaled} \&~\ref{fig:right-hand-side-of-consistency-subsampling-median-unbiased} that the right hand side of~\eqref{eq:dimension-growing-sample-size-requirement} can be as large as $0.4$ for certain choices of $\alpha$, even for $d\gg n$. Similarly, Theorem~\ref{thm:infinite-order-coverage} (in particular its implication~\eqref{eq:conclusion-theorem-BE-bound}) yields 
\begin{equation}\label{eq:usual-median-bias-union-bnd}
\mathbb{P}\left(\theta_0 \notin \bigotimes_{k=1}^d \widehat{\mathrm{CI}}_{\alpha/d,0}^{(k)}\right) \le \sum_{k=1}^d \mathbb{P}\left(\theta_0 \notin \widehat{\mathrm{CI}}_{\alpha/d,0}^{(k)}\right) \le \sum_{k=1}^d \frac{\alpha}{d}\left(1 + \mathfrak{C}_X'\frac{B_{\alpha/d,0}^3}{n}\right) = \alpha\left(1 + \mathfrak{C}_X'\frac{B_{\alpha/d,0}^3}{n}\right).
\end{equation}
Here the union bound is valid irrespective of what the dimension $d$ is relative to the sample size $n$. If the estimators are median bias reduced, then the same argument as above yields
\begin{equation}\label{eq:less-median-bias-union-bnd}
\mathbb{P}\left(\theta_0 \notin \bigotimes_{k=1}^d \widehat{\mathrm{CI}}_{\alpha/d,0}^{(k)}\right) ~\le~ \alpha\left(1 + \mathfrak{C}_X'\frac{B_{\alpha/d,0}^5}{n^3}\right).
\end{equation}
Note, once again, that $B_{\alpha/d,0} \asymp \log(2d/\alpha)$ and hence the miscoverage probabilities are bounded by $\alpha(1 + \mathfrak{C}_X'\log^3(d/\alpha)/n)$ (under~\eqref{eq:BE-median-bias}) and $\alpha(1 +  \mathfrak{C}_X'\log^5(d/\alpha)/n^3)$ (under median bias reduction).

We see that we only require
$\log(d/\alpha) = o(n^{1/3})$ 
(or $\log(d/\alpha) = o(n^{3/5})$ in the bias reduced case)
and we do not require joint/multivariate distributional 
convergence whatsoever. This can be contrasted with the results from
the literature on high-dimensional central limit
theorems~\citep{belloni2018high,koike2020notes,fang2020high,deng2020slightly,chernozhukov2020nearly}.
These results %(except for~\cite{belloni2018high}) 
concern
Gaussian approximation for high-dimensional averages. Under certain
moment assumptions on the \emph{joint} distribution, these results
imply a joint Gaussian approximation with a minimum requirement
of $\log(d) = o(n^{1/3})$~\citep[Proposition
  1.1]{fang2020high}. \citet[Theorem 2.3]{belloni2018high} uses the union bound based on moderate deviations but still
  requires joint moment conditions and the 
  condition that  
$\log(d) =
o(n^{1/3})$. With usual
estimators, the \hulc also has the same dimensionality requirement while
only making use of marginal median bias. With median bias reduced
estimators, the \hulc only requires $\log(d) = o(n^{3/5})$ and this is even
weaker than $\log(d) = o(n^{1/2})$, which is the best possible
dimension restriction for a Gaussian approximation~\citep[Theorem
  3]{das2020central}. It is worth mentioning that by slightly
enlarging the bootstrap confidence regions, the dimensionality
requirement can be reduced to $\log(d) =
o(n)$ in the case of mean estimation~\citep{deng2020slightly}.

Although we have used the union bound above to obtain a coverage of $1-\alpha$ for a multivariate parameter, we only required \emph{asymptotic} median unbiasedness of $e_k^{\top}\widehat{\theta}_j$ marginally. There is no requirement whatsoever on the asymptotic joint convergence or symmetry of $\widehat{\theta}_j\in\mathbb{R}^d$. Interestingly, such a result is not possible with the usual confidence intervals. This point is discussed further in Section~\ref{subsec:union-bound-Wald} of the supplementary material.

\section{Conclusions and Future Directions}\label{sec:conclusions}
In this paper, we developed and analyzed a simple and broadly applicable method, the \hulc, 
for constructing confidence sets, using the convex hull of estimates constructed on independent subsamples
of the data. 
{\color{black}All the \hulc intervals presented have an asymptotic coverage either equal to or greater than $1-\alpha$.}
The \hulc bypasses the difficult problem of estimating nuisance components in the limiting distribution,
requires fewer regularity conditions than the bootstrap and unlike subsampling 
does not require knowledge of the rate of convergence of the underlying estimates on which it is based. 
These advantages, in many cases, come at a surprisingly small price in the width of the interval. The width of the intervals
are determined in general by the accuracy of the underlying estimators, as well as their median bias. 
We also present two variants, 
the \ahulc which estimates the median bias using subsampling, and the \uhulc which can be useful even in cases when the median bias is large so
long as the limiting distribution
is unimodal. Beyond these methodological contributions, we also studied several challenging confidence set construction problems 
and showed how our methods can often provide simple solutions to these problems. 

From a computational standpoint, the \hulc only requires
computing the estimator $B$ times where $B$ is typically around 5 or
10, and so is less computationally intensive than the bootstrap. In cases where the underlying estimator has computational 
complexity which is super-linear in the number of samples, computing $B$ estimates on $n/B$ samples can in fact be cheaper
than computing a single estimate on the whole dataset. {\color{black}These considerations are relevant no matter the dimension of the estimator.}

Intuitively, the \hulc is also quite robust in the sense that we only use
qualitative properties, such as an upper bound on the median bias, of the
limiting distribution rather than its exact form. In finite samples, the
distribution of a statistic might be close to symmetric
even if it does not resemble a Gaussian distribution. 

Our analysis has not discussed the important problem
of obtaining confidence intervals with uniform coverage.
This is especially
important in irregular problems
where the rate of convergence
and limiting distribution can vary
across the parameter space. Developing this understanding is an important open problem that 
we plan to address in future work.

\section*{Acknowledgements}
We are grateful to Prof. Hannes Leeb for helpful comments, and for pointing out a gap in an earlier version of Theorem~\ref{thm:coverage-algorithm-known-Delta} and Proposition~\ref{prop:approximate-to-exact-validity}. We also thank Jin-Hong Du for his help with Figures~\ref{fig:comparison-coverage-lm}--\ref{fig:mean-square-subsampling}. We also thank the reviewers and the associate editor for their constructive comments that led to several improvements of the paper.
\bibliographystyle{apalike}
\bibliography{references}

\begin{thebibliography}{}

\bibitem[Abadie and Imbens, 2008]{abadie2008failure}
Abadie, A. and Imbens, G.~W. (2008).
\newblock On the failure of the bootstrap for matching estimators.
\newblock {\em Econometrica}, 76(6):1537--1557.

\bibitem[Andrews, 2000]{andrews2000inconsistency}
Andrews, D. W.~K. (2000).
\newblock Inconsistency of the bootstrap when a parameter is on the boundary of
  the parameter space.
\newblock {\em Econometrica}, 68(2):399--405.

\bibitem[Andrews and Guggenberger, 2010]{andrews2010asymptotic}
Andrews, D. W.~K. and Guggenberger, P. (2010).
\newblock Asymptotic size and a problem with subsampling and with the {$m$} out
  of {$n$} bootstrap.
\newblock {\em Econometric Theory}, 26(2):426--468.

\bibitem[Andrews and Phillips, 1987]{andrews1987best}
Andrews, D. W.~K. and Phillips, P. C.~B. (1987).
\newblock Best median-unbiased estimation in linear regression with bounded
  asymmetric loss functions.
\newblock {\em J. Amer. Statist. Assoc.}, 82(399):886--893.

\bibitem[Athreya, 1987]{athreya1987bootstrap}
Athreya, K.~B. (1987).
\newblock Bootstrap of the mean in the infinite variance case.
\newblock {\em Ann. Statist.}, 15(2):724--731.

\bibitem[Belloni et~al., 2018]{belloni2018high}
Belloni, A., Chernozhukov, V., Chetverikov, D., Hansen, C., and Kato, K.
  (2018).
\newblock High-dimensional econometrics and regularized {GMM}.
\newblock {\em arXiv preprint arXiv:1806.01888}.

\bibitem[Bentkus, 2005]{bentkus2005lyapunov}
Bentkus, V. (2005).
\newblock A {L}yapunov type bound in {${\bf R}^d$}.
\newblock {\em Theory Probab. Appl.}, 49(2):311--323.

\bibitem[Bentkus et~al., 1997]{bentkus1997berry}
Bentkus, V., Bloznelis, M., and G\"{o}tze, F. (1997).
\newblock A {B}erry-{E}ss\'{e}en bound for {$M$}-estimators.
\newblock {\em Scand. J. Statist.}, 24(4):485--502.

\bibitem[Bertail and Politis, 2001]{bertail2001extrapolation}
Bertail, P. and Politis, D.~N. (2001).
\newblock Extrapolation of subsampling distribution estimators: the i.i.d. and
  strong mixing cases.
\newblock {\em Canad. J. Statist.}, 29(4):667--680.

\bibitem[Bertail et~al., 1999]{bertail1999subsampling}
Bertail, P., Politis, D.~N., and Romano, J.~P. (1999).
\newblock On subsampling estimators with unknown rate of convergence.
\newblock {\em J. Amer. Statist. Assoc.}, 94(446):569--579.

\bibitem[Bickel, 1982]{bickel1982adaptive}
Bickel, P.~J. (1982).
\newblock On adaptive estimation.
\newblock {\em Ann. Statist.}, 10(3):647--671.

\bibitem[Bickel et~al., 1993]{bickel1993efficient}
Bickel, P.~J., Klaassen, C. A.~J., Ritov, Y., and Wellner, J.~A. (1993).
\newblock {\em Efficient and adaptive estimation for semiparametric models}.
\newblock Johns Hopkins Series in the Mathematical Sciences. Johns Hopkins
  University Press, Baltimore, MD.

\bibitem[Birnbaum, 1964]{birnbaum1964median}
Birnbaum, A. (1964).
\newblock Median-unbiased estimators.
\newblock {\em Bull. Math. Statist.}, 11(1-2):25--34.

\bibitem[Borges, 6970]{borges1970approximation}
Borges, R. (1969/70).
\newblock Eine {A}pproximation der {B}inomialverteilung durch die
  {N}ormalverteilung der {O}rdnung {$1/n$}.
\newblock {\em Z. Wahrscheinlichkeitstheorie und Verw. Gebiete}, 14:189--199.

\bibitem[Borges, 1971]{borges1971derivation}
Borges, R. (1971).
\newblock Derivation of normalizing transformations with an error of order
  {$1/n$}.
\newblock {\em Sankhy\={a} Ser. A}, 33:441--460.

\bibitem[Boucheron and Thomas, 2012]{boucheron2012concentration}
Boucheron, S. and Thomas, M. (2012).
\newblock Concentration inequalities for order statistics.
\newblock {\em Electron. Commun. Probab.}, 17:no. 51, 12.

\bibitem[Breiman, 1992]{breiman1968probability}
Breiman, L. (1992).
\newblock {\em Probability}, volume~7 of {\em Classics in Applied Mathematics}.
\newblock Society for Industrial and Applied Mathematics (SIAM), Philadelphia,
  PA.
\newblock Corrected reprint of the 1968 original.

\bibitem[Brown et~al., 2001]{brown2001interval}
Brown, L.~D., Cai, T.~T., and DasGupta, A. (2001).
\newblock Interval estimation for a binomial proportion.
\newblock {\em Statist. Sci.}, 16(2):101--133.
\newblock With comments and a rejoinder by the authors.

\bibitem[Brown et~al., 2002]{brown2002confidence}
Brown, L.~D., Cai, T.~T., and DasGupta, A. (2002).
\newblock Confidence intervals for a binomial proportion and asymptotic
  expansions.
\newblock {\em Ann. Statist.}, 30(1):160--201.

\bibitem[Brown et~al., 1976]{brown1976complete}
Brown, L.~D., Cohen, A., and Strawderman, W.~E. (1976).
\newblock A complete class theorem for strict monotone likelihood ratio with
  applications.
\newblock {\em Ann. Statist.}, 4(4):712--722.

\bibitem[Buja et~al., 2019]{buja2019models}
Buja, A., Brown, L., Berk, R., George, E., Pitkin, E., Traskin, M., Zhang, K.,
  and Zhao, L. (2019).
\newblock Models as approximations {I}: consequences illustrated with linear
  regression.
\newblock {\em Statist. Sci.}, 34(4):523--544.

\bibitem[Cabrera and Watson, 1997]{cabrera1997simulation}
Cabrera, J. and Watson, G. (1997).
\newblock Simulation methods for mean and median bias reduction in parametric
  estimation.
\newblock {\em Journal of statistical planning and inference}, 57(1):143--152.

\bibitem[Chakravarti et~al., 2019]{chakravarti2019gaussian}
Chakravarti, P., Balakrishnan, S., and Wasserman, L. (2019).
\newblock Gaussian mixture clustering using relative tests of fit.
\newblock {\em arXiv preprint arXiv:1910.02566}.

\bibitem[Chernozhukov et~al., 2020]{chernozhukov2020nearly}
Chernozhukov, V., Chetverikov, D., and Koike, Y. (2020).
\newblock Nearly optimal central limit theorem and bootstrap approximations in
  high dimensions.
\newblock {\em arXiv preprint arXiv:2012.09513}.

\bibitem[Das and Lahiri, 2021]{das2020central}
Das, D. and Lahiri, S. (2021).
\newblock Central {L}imit {T}heorem in high dimensions: the optimal bound on
  dimension growth rate.
\newblock {\em Trans. Amer. Math. Soc.}, 374(10):6991--7009.

\bibitem[Deng, 2020]{deng2020slightly}
Deng, H. (2020).
\newblock Slightly conservative bootstrap for maxima of sums.
\newblock {\em arXiv preprint arXiv:2007.15877}.

\bibitem[Deng et~al., 2022]{deng2020inference}
Deng, H., Han, Q., and Sen, B. (2022).
\newblock Inference for local parameters in convexity constrained models.
\newblock {\em Journal of the American Statistical Association},
  (just-accepted):1--33.

\bibitem[Deng et~al., 2021]{deng2020confidence}
Deng, H., Han, Q., and Zhang, C.-H. (2021).
\newblock Confidence intervals for multiple isotonic regression and other
  monotone models.
\newblock {\em Ann. Statist.}, 49(4):2021--2052.

\bibitem[Desu and Rodine, 1969]{mahamunulu1969estimation}
Desu, M.~M. and Rodine, R.~H. (1969).
\newblock Estimation of the population median.
\newblock {\em Scandinavian Actuarial Journal}, 1969(1-2):67--70.

\bibitem[Doerr, 2018]{doerr2018elementary}
Doerr, B. (2018).
\newblock An elementary analysis of the probability that a binomial random
  variable exceeds its expectation.
\newblock {\em Statist. Probab. Lett.}, 139:67--74.

\bibitem[D\"{u}mbgen, 1993]{dumbgen1993nondifferentiable}
D\"{u}mbgen, L. (1993).
\newblock On nondifferentiable functions and the bootstrap.
\newblock {\em Probab. Theory Related Fields}, 95(1):125--140.

\bibitem[Durot, 2008]{durot2008monotone}
Durot, C. (2008).
\newblock Monotone nonparametric regression with random design.
\newblock {\em Math. Methods Statist.}, 17(4):327--341.

\bibitem[Efron, 1979]{efron1979bootstrap}
Efron, B. (1979).
\newblock Bootstrap methods: another look at the jackknife.
\newblock {\em Ann. Statist.}, 7(1):1--26.

\bibitem[Efron, 1982]{efron1982transformation}
Efron, B. (1982).
\newblock Transformation theory: how normal is a family of distributions?
\newblock {\em Ann. Statist.}, 10(2):323--339.

\bibitem[Fang and Koike, 2021]{fang2020high}
Fang, X. and Koike, Y. (2021).
\newblock High-dimensional central limit theorems by {S}tein's method.
\newblock {\em Ann. Appl. Probab.}, 31(4):1660--1686.

\bibitem[Fang and Santos, 2019]{fang2019inference}
Fang, Z. and Santos, A. (2019).
\newblock Inference on directionally differentiable functions.
\newblock {\em Rev. Econ. Stud.}, 86(1):377--412.

\bibitem[Firth, 1993]{firth1993bias}
Firth, D. (1993).
\newblock Bias reduction of maximum likelihood estimates.
\newblock {\em Biometrika}, 80(1):27--38.

\bibitem[Gebhardt, 1969]{gebhardt1969some}
Gebhardt, F. (1969).
\newblock Some numerical comparisons of several approximations to the binomial
  distribution.
\newblock {\em J. Amer. Statist. Assoc.}, 64(328):1638--1646.

\bibitem[Greenberg and Mohri, 2014]{greenberg2014tight}
Greenberg, S. and Mohri, M. (2014).
\newblock Tight lower bound on the probability of a binomial exceeding its
  expectation.
\newblock {\em Statist. Probab. Lett.}, 86:91--98.

\bibitem[Guntuboyina and Sen, 2018]{guntuboyina2018nonparametric}
Guntuboyina, A. and Sen, B. (2018).
\newblock Nonparametric shape-restricted regression.
\newblock {\em Statist. Sci.}, 33(4):568--594.

\bibitem[Hall, 1982]{hall1982estimating}
Hall, P. (1982).
\newblock On estimating the endpoint of a distribution.
\newblock {\em Ann. Statist.}, 10(2):556--568.

\bibitem[Hall, 1986]{hall1986bootstrap}
Hall, P. (1986).
\newblock On the bootstrap and confidence intervals.
\newblock {\em Ann. Statist.}, 14(4):1431--1452.

\bibitem[Hall, 1988]{hall1988theoretical}
Hall, P. (1988).
\newblock Theoretical comparison of bootstrap confidence intervals.
\newblock {\em Ann. Statist.}, 16(3):927--985.
\newblock With a discussion and a reply by the author.

\bibitem[Hall, 1992]{hall2013bootstrap}
Hall, P. (1992).
\newblock {\em The bootstrap and {E}dgeworth expansion}.
\newblock Springer Series in Statistics. Springer-Verlag, New York.

\bibitem[Hamza, 1995]{hamza1995smallest}
Hamza, K. (1995).
\newblock The smallest uniform upper bound on the distance between the mean and
  the median of the binomial and {P}oisson distributions.
\newblock {\em Statist. Probab. Lett.}, 23(1):21--25.

\bibitem[Han and Kato, 2022]{han2019berry}
Han, Q. and Kato, K. (2022).
\newblock {B}erry--{E}sseen bounds for {C}hernoff-type nonstandard asymptotics
  in isotonic regression.
\newblock {\em Ann. Appl. Probab.}, 32(2):1459--1498.

\bibitem[Hartigan, 1969]{hartigan1969using}
Hartigan, J.~A. (1969).
\newblock Using subsample values as typical values.
\newblock {\em J. Amer. Statist. Assoc.}, 64:1303--1317.

\bibitem[Hartigan, 1970]{hartigan1970exact}
Hartigan, J.~A. (1970).
\newblock Exact confidence intervals in regression problems with independent
  symmetric errors.
\newblock {\em Ann. Math. Statist.}, 41:1992--1998.

\bibitem[Hayakawa et~al., 2021]{hayakawa2021estimating}
Hayakawa, S., Lyons, T., and Oberhauser, H. (2021).
\newblock Estimating the probability that a given vector is in the convex hull
  of a random sample.
\newblock {\em arXiv preprint arXiv:2101.04250}.

\bibitem[Hirano and Porter, 2012]{hirano2012impossibility}
Hirano, K. and Porter, J.~R. (2012).
\newblock Impossibility results for nondifferentiable functionals.
\newblock {\em Econometrica}, 80(4):1769--1790.

\bibitem[Hirji et~al., 1989]{hirji1989median}
Hirji, K.~F., Tsiatis, A.~A., and Mehta, C.~R. (1989).
\newblock Median unbiased estimation for binary data.
\newblock {\em Amer. Statist.}, 43(1):7--11.

\bibitem[Hoeffding, 1963]{hoeffding1963probability}
Hoeffding, W. (1963).
\newblock Probability inequalities for sums of bounded random variables.
\newblock {\em J. Amer. Statist. Assoc.}, 58:13--30.

\bibitem[Ibragimov and M\"{u}ller, 2010]{Ibragimov2010}
Ibragimov, R. and M\"{u}ller, U.~K. (2010).
\newblock {$t$}-statistic based correlation and heterogeneity robust inference.
\newblock {\em J. Bus. Econom. Statist.}, 28(4):453--468.

\bibitem[Jing et~al., 2003]{jing2003self}
Jing, B.-Y., Shao, Q.-M., and Wang, Q. (2003).
\newblock Self-normalized {C}ram\'{e}r-type large deviations for independent
  random variables.
\newblock {\em Ann. Probab.}, 31(4):2167--2215.

\bibitem[John, 1974]{john1974median}
John, S. (1974).
\newblock Median-unbiased most acceptable estimates of {P}oisson, binomial and
  negative-binomial distributions.
\newblock {\em Comm. Statist.}, 3:1155--1159.

\bibitem[Kabluchko and Zaporozhets, 2019]{kabluchko2019expected}
Kabluchko, Z. and Zaporozhets, D. (2019).
\newblock Expected volumes of {G}aussian polytopes, external angles, and
  multiple order statistics.
\newblock {\em Trans. Amer. Math. Soc.}, 372(3):1709--1733.

\bibitem[Kenne~Pagui et~al., 2017]{kenne2017median}
Kenne~Pagui, E.~C., Salvan, A., and Sartori, N. (2017).
\newblock Median bias reduction of maximum likelihood estimates.
\newblock {\em Biometrika}, 104(4):923--938.

\bibitem[Kim, 2016]{kim2016higher}
Kim, K.~I. (2016).
\newblock Higher order bias correcting moment equation for m-estimation and its
  higher order efficiency.
\newblock {\em Econometrics}, 4(4):48.

\bibitem[Knight, 1989]{knight1989bootstrap}
Knight, K. (1989).
\newblock On the bootstrap of the sample mean in the infinite variance case.
\newblock {\em Ann. Statist.}, 17(3):1168--1175.

\bibitem[Knight, 1998]{knight1998limiting}
Knight, K. (1998).
\newblock Limiting distributions for {$L_1$} regression estimators under
  general conditions.
\newblock {\em Ann. Statist.}, 26(2):755--770.

\bibitem[Knight, 1999]{knight1999asymptotics}
Knight, K. (1999).
\newblock Asymptotics for ${L}_1$-estimators of regression parameters under
  heteroscedasticityy.
\newblock {\em Canadian Journal of Statistics}, 27(3):497--507.

\bibitem[Knight, 2008]{knight2008asymptotics}
Knight, K. (2008).
\newblock Asymptotics of the regression quantile basic solution under
  misspecification.
\newblock {\em Applications of Mathematics}, 53(3):223--234.

\bibitem[Koike, 2021]{koike2020notes}
Koike, Y. (2021).
\newblock Notes on the dimension dependence in high-dimensional central limit
  theorems for hyperrectangles.
\newblock {\em Jpn. J. Stat. Data Sci.}, 4(1):257--297.

\bibitem[Koltchinskii, 2020]{koltchinskii2020estimation}
Koltchinskii, V. (2020).
\newblock Estimation of smooth functionals in high-dimensional models:
  bootstrap chains and gaussian approximation.
\newblock {\em arXiv preprint arXiv:2011.03789}.

\bibitem[Koltchinskii and Zhilova, 2021a]{koltchinskii2021efficient}
Koltchinskii, V. and Zhilova, M. (2021a).
\newblock Efficient estimation of smooth functionals in {G}aussian shift
  models.
\newblock {\em Ann. Inst. Henri Poincar\'{e} Probab. Stat.}, 57(1):351--386.

\bibitem[Koltchinskii and Zhilova, 2021b]{koltchinskii2019estimation}
Koltchinskii, V. and Zhilova, M. (2021b).
\newblock Estimation of smooth functionals in normal models: bias reduction and
  asymptotic efficiency.
\newblock {\em Ann. Statist.}, 49(5):2577--2610.

\bibitem[Kosmidis and Firth, 2009]{kosmidis2009bias}
Kosmidis, I. and Firth, D. (2009).
\newblock Bias reduction in exponential family nonlinear models.
\newblock {\em Biometrika}, 96(4):793--804.

\bibitem[Kosmidis et~al., 2020]{kosmidis2020mean}
Kosmidis, I., Kenne~Pagui, E.~C., and Sartori, N. (2020).
\newblock Mean and median bias reduction in generalized linear models.
\newblock {\em Stat. Comput.}, 30(1):43--59.

\bibitem[Kuchibhotla, 2021]{kuchibhotla2021median}
Kuchibhotla, A.~K. (2021).
\newblock Median bias of {M}-estimators.
\newblock {\em arXiv preprint arXiv:2106.00164}.

\bibitem[Kuchibhotla et~al., 2023]{kuchibhotlamedian}
Kuchibhotla, A.~K., Balakrishnan, S., and Wasserman, L. (2023).
\newblock Median regularity and honest inference.
\newblock {\em Biometrika}, Accepted.

\bibitem[Kuchibhotla et~al., 2021]{kuchibhotla2021semiparametric}
Kuchibhotla, A.~K., Patra, R.~K., and Sen, B. (2021).
\newblock Semiparametric efficiency in convexity constrained single index
  model.
\newblock {\em J. Amer. Statist. Assoc.}

\bibitem[Lam, 2022]{lam2022cheap}
Lam, H. (2022).
\newblock A cheap bootstrap method for fast inference.
\newblock {\em arXiv preprint arXiv:2202.00090}.

\bibitem[Lanke, 1974]{lanke1974interval}
Lanke, J. (1974).
\newblock Interval estimation of a median.
\newblock {\em Scand. J. Statist.}, 1(1):28--32.

\bibitem[Laurent, 1997]{laurent1997estimation}
Laurent, B. (1997).
\newblock Estimation of integral functionals of a density and its derivatives.
\newblock {\em Bernoulli}, 3(2):181--211.

\bibitem[Lehmann, 1959]{Lehmann1959}
Lehmann, E.~L. (1959).
\newblock {\em Testing statistical hypotheses}.
\newblock John Wiley \& Sons, Inc., New York; Chapman \& Hall, Ltd., London.

\bibitem[Lehmann and Casella, 1998]{lehmann2006theory}
Lehmann, E.~L. and Casella, G. (1998).
\newblock {\em Theory of point estimation}.
\newblock Springer Texts in Statistics. Springer-Verlag, New York, second
  edition.

\bibitem[Li and Racine, 2004]{li2004cross}
Li, Q. and Racine, J. (2004).
\newblock Cross-validated local linear nonparametric regression.
\newblock {\em Statist. Sinica}, 14(2):485--512.

\bibitem[Loh, 1984]{loh1984estimating}
Loh, W.-Y. (1984).
\newblock Estimating an endpoint of a distribution with resampling methods.
\newblock {\em Ann. Statist.}, 12(4):1543--1550.

\bibitem[Mammen, 1992]{mammen1992bootstrap}
Mammen, E. (1992).
\newblock Bootstrap, wild bootstrap, and asymptotic normality.
\newblock {\em Probab. Theory Related Fields}, 93(4):439--455.

\bibitem[Massart, 1990]{massart1990tight}
Massart, P. (1990).
\newblock The tight constant in the {D}voretzky-{K}iefer-{W}olfowitz
  inequality.
\newblock {\em Ann. Probab.}, 18(3):1269--1283.

\bibitem[Pfanzagl, 1970a]{pfanzagl1970median}
Pfanzagl, J. (1970a).
\newblock Median unbiased estimates for {M}. {L}. {R}.-families.
\newblock {\em Metrika}, 15:30--39.

\bibitem[Pfanzagl, 1970b]{pfanzagl1970asymptotic}
Pfanzagl, J. (1970b).
\newblock On the asymptotic efficiency of median unbiased estimates.
\newblock {\em Ann. Math. Statist.}, 41:1500--1509.

\bibitem[Pfanzagl, 1971]{pfanzagl1971berry}
Pfanzagl, J. (1971).
\newblock The {B}erry-{E}sseen bound for minimum contrast estimates.
\newblock {\em Metrika}, 17:82--91.

\bibitem[Pfanzagl, 7172]{pfanzagl1972median}
Pfanzagl, J. (1971/72).
\newblock On median unbiased estimates.
\newblock {\em Metrika}, 18:154--173.

\bibitem[Pfanzagl, 1973a]{pfanzagl1973accuracy}
Pfanzagl, J. (1973a).
\newblock Asymptotic expansions related to minimum contrast estimators.
\newblock {\em Ann. Statist.}, 1:993--1026.

\bibitem[Pfanzagl, 1973b]{pfanzagl1973asymptotic}
Pfanzagl, J. (1973b).
\newblock Asymptotic expansions related to minimum contrast estimators.
\newblock {\em The Annals of Statistics}, pages 993--1026.

\bibitem[Pfanzagl, 1979]{pfanzagl1979optimal}
Pfanzagl, J. (1979).
\newblock On optimal median unbiased estimators in the presence of nuisance
  parameters.
\newblock {\em Ann. Statist.}, 7(1):187--193.

\bibitem[Pfanzagl, 1994]{pfanzagl2011parametric}
Pfanzagl, J. (1994).
\newblock {\em Parametric statistical theory}.
\newblock De Gruyter Textbook. Walter de Gruyter \& Co., Berlin.
\newblock With the assistance of R. Hamb\"{o}ker.

\bibitem[Pfanzagl, 2017]{pfanzagl2017optimality}
Pfanzagl, J. (2017).
\newblock Optimality of unbiased estimators: Nonasymptotic theory.
\newblock In {\em Mathematical Statistics}, Springer Series in Statistics,
  pages 83--106. Springer-Verlag, Berlin.
\newblock Essays on history and methodology, Springer Series in Statistics.
  Perspectives in Statistics.

\bibitem[Pinelis, 2017]{pinelis2017optimal}
Pinelis, I. (2017).
\newblock Optimal-order uniform and nonuniform bounds on the rate of
  convergence to normality for maximum likelihood estimators.
\newblock {\em Electron. J. Stat.}, 11(1):1160--1179.

\bibitem[Politis and Romano, 1994]{politis1994large}
Politis, D.~N. and Romano, J.~P. (1994).
\newblock Large sample confidence regions based on subsamples under minimal
  assumptions.
\newblock {\em Ann. Statist.}, 22(4):2031--2050.

\bibitem[Read, 2004]{read2004median}
Read, C.~B. (2004).
\newblock Median unbiased estimators.
\newblock {\em Encyclopedia of statistical sciences}, 7.

\bibitem[Rinaldo et~al., 2019]{rinaldo2019bootstrapping}
Rinaldo, A., Wasserman, L., and G'Sell, M. (2019).
\newblock Bootstrapping and sample splitting for high-dimensional,
  assumption-lean inference.
\newblock {\em Ann. Statist.}, 47(6):3438--3469.

\bibitem[Robins, 2004]{robins2004optimal}
Robins, J.~M. (2004).
\newblock Optimal structural nested models for optimal sequential decisions.
\newblock In {\em Proceedings of the {S}econd {S}eattle {S}ymposium in
  {B}iostatistics}, volume 179 of {\em Lect. Notes Stat.}, pages 189--326.
  Springer, New York.

\bibitem[Robson and Whitlock, 1964]{robson1964estimation}
Robson, D.~S. and Whitlock, J.~H. (1964).
\newblock Estimation of a truncation point.
\newblock {\em Biometrika}, 51:33--39.

\bibitem[Romano and Wolf, 1999]{romano1999subsampling}
Romano, J.~P. and Wolf, M. (1999).
\newblock Subsampling inference for the mean in the heavy-tailed case.
\newblock {\em Metrika}, 50(1):55--69.

\bibitem[Sen, 1968]{sen1968asymptotic}
Sen, P.~K. (1968).
\newblock Asymptotic normality of sample quantiles for {$m$}-dependent
  processes.
\newblock {\em Ann. Math. Statist.}, 39:1724--1730.

\bibitem[Shao and Tu, 1995]{shao2012jackknife}
Shao, J. and Tu, D.~S. (1995).
\newblock {\em The jackknife and bootstrap}.
\newblock Springer Series in Statistics. Springer-Verlag, New York.

\bibitem[Shao, 1997]{shao1997self}
Shao, Q.-M. (1997).
\newblock Self-normalized large deviations.
\newblock {\em Ann. Probab.}, 25(1):285--328.

\bibitem[Sherman and Carlstein, 2004]{sherman1997omnibus}
Sherman, M. and Carlstein, E. (2004).
\newblock Confidence intervals based on estimators with unknown rates of
  convergence.
\newblock {\em Comput. Statist. Data Anal.}, 46(1):123--139.

\bibitem[Stigler, 2007]{stigler2007epic}
Stigler, S.~M. (2007).
\newblock The epic story of maximum likelihood.
\newblock {\em Statist. Sci.}, 22(4):598--620.

\bibitem[van~der Vaart, 1998]{van2000asymptotic}
van~der Vaart, A.~W. (1998).
\newblock {\em Asymptotic statistics}, volume~3 of {\em Cambridge Series in
  Statistical and Probabilistic Mathematics}.
\newblock Cambridge University Press, Cambridge.

\bibitem[Wagner and Welzl, 2001]{wagner2001continuous}
Wagner, U. and Welzl, E. (2001).
\newblock A continuous analogue of the upper bound theorem.
\newblock volume~26, pages 205--219.
\newblock ACM Symposium on Computational Geometry (Hong Kong, 2000).

\bibitem[Wasserman et~al., 2020]{wasserman2020universal}
Wasserman, L., Ramdas, A., and Balakrishnan, S. (2020).
\newblock Universal inference.
\newblock {\em Proc. Natl. Acad. Sci. USA}, 117(29):16880--16890.

\bibitem[Wendel, 1962]{wendel1962problem}
Wendel, J.~G. (1962).
\newblock A problem in geometric probability.
\newblock {\em Math. Scand.}, 11:109--111.

\bibitem[Wright, 1981]{wright1981asymptotic}
Wright, F.~T. (1981).
\newblock The asymptotic behavior of monotone regression estimates.
\newblock {\em Ann. Statist.}, 9(2):443--448.

\bibitem[Zhang and Liang, 2011]{zhang2011berry}
Zhang, J.-J. and Liang, H.-Y. (2011).
\newblock Berry-{E}sseen type bounds in heteroscedastic semi-parametric model.
\newblock {\em J. Statist. Plann. Inference}, 141(11):3447--3462.

\end{thebibliography}
\newpage
% \newpage
% \appendix
% \newpage
\setcounter{section}{0}
\setcounter{equation}{0}
\setcounter{figure}{0}
\setcounter{section}{0}
\setcounter{equation}{0}
\setcounter{figure}{0}
\renewcommand{\thesection}{S.\arabic{section}}
\renewcommand{\theequation}{E.\arabic{equation}}
\renewcommand{\thefigure}{S.\arabic{figure}}
% \setcounter{page}{1}
% \cftsetindents{section}{1em}{2.5em}
% \cftsetindents{subsection}{1.5em}{3em}
\begin{center}
\Large {\bf Supplement to\\ ``The \hulcnospace: Confidence Regions from Convex Hulls''}
\end{center}
       
\begin{abstract}
This supplement contains the proofs of all the main results in the paper. 
\end{abstract}

\section{Union bound with Wald intervals}\label{subsec:union-bound-Wald}
If, for each $1\le k\le d$, the estimators $e_k^{\top}\widehat{\theta}_j$ are asymptotically normal, then asymptotic normality implies that for all $\gamma\in(0, 1)$,
\[
\left|\mathbb{P}\left(e_k^{\top}\theta_0 \in \widehat{\mathrm{CI}}_{\gamma}^{\texttt{Wald}, k}\right) - (1 - \gamma)\right| \le \delta_n,
\]
for some $\delta_n$ converging to zero as $n\to\infty$; an example is~\eqref{eq:BE-bound-general}. First order accurate confidence intervals (such as Wald's) satisfy $\delta_n = O(n^{-1/2})$, second order accurate ones satisfy $\delta_n = O(n^{-1})$ and so on. Taking $\gamma = \alpha/d$ and applying the union bound, we only obtain
\[
\mathbb{P}\left(\bigcup_{k=1}^d\left\{e_k^{\top}\theta_0 \notin \widehat{\mathrm{CI}}_{\alpha/d}^{\texttt{Wald}, k}\right\}\right) \le \sum_{k=1}^d \mathbb{P}\left(e_k^{\top}\theta_0 \notin \widehat{\mathrm{CI}}_{\alpha/d}^{\texttt{Wald}, k}\right) \le \alpha + d\delta_n.
\]
In order for the right hand side to be $\alpha$ asymptotically, one needs $d\delta_n = o(1)$. This is a very stringent requirement, especially when the dimension $d$ grows faster than the sample size $n$. 

There is a simple way to resolve this issue following our proposed methodology. The idea is to construct $1/2$ Wald confidence regions for each coordinate $e_k^{\top}\theta_0$ from each of $e_k^{\top}\widehat{\theta}_j, 1\le j\le B$ and then take the union of these regions. Formally, set $\widebar{B}_{d,\alpha} = \lceil\log(d/\alpha)/\log(2)\rceil$. For $1\le j\le \widebar{B}_{d,\alpha}$, suppose $\widehat{\mathrm{CI}}_{j}^{\texttt{Wald}, k}$ is the Wald confidence region of coverage $1/2$ based on $e_k^{\top}\widehat{\theta}_j$. This means that
\[
\left|\mathbb{P}\left(e_k^{\top}\theta_0 \in \widehat{\mathrm{CI}}_{j}^{\texttt{Wald}, k}\right) - \frac{1}{2}\right| \le \delta_{n,d},\quad\mbox{for all}\quad 1\le j\le \widebar{B}_{d,\alpha}, 1\le k\le d.
\]
The right hand side $\delta_{n,d}$ here depends on the dimension $d$ because $\widehat{\theta}_j$ is computed based on $n/\widebar{B}_{d/\alpha}$ many observations. 
Under these conditions, the following result provides a valid $1 - \alpha$ confidence regions using Wald's confidence intervals and a union bound. The interesting aspect (similar to~\eqref{eq:infinite-order-coverage}) is that the coverage implied by Proposition~\ref{prop:Wald-union-bound} is eventually finite sample (i.e., holds after some sample size) even though the coverage of $\widehat{\mathrm{CI}}_j^{\texttt{Wald}, k}$ is asymptotic.
\begin{prop}\label{prop:Wald-union-bound}
If
\begin{equation}\label{eq:Wald-sample-size-condition}
\widebar{B}_{d,\alpha}\delta_{n,d} ~\le~ \widebar{B}_{d,\alpha}\left[\left(\frac{\alpha}{d}\right)^{1/\widebar{B}_{d,\alpha}} - \frac{1}{2}\right],
\end{equation}
then
\[
\mathbb{P}\left(\theta_0 \notin \bigotimes_{k=1}^d\bigcup_{j=1}^{\widebar{B}_{d,\alpha}}\widehat{\mathrm{CI}}_{j}^{\texttt{Wald}, k}\right) ~\le~ \alpha.
\]
\end{prop}
\begin{proof}%[Proof of Proposition~\ref{prop:Wald-union-bound}]
Following the proof of Lemma~\ref{lem:fixed-B-coverage-result}, we obtain
\[
\mathbb{P}\left(e_k^{\top}\theta_0 \notin \bigcup_{j=1}^{\widebar{B}_{d,\alpha}} \widehat{\mathrm{CI}}_{j}^{\texttt{Wald}, k}\right) ~\le~ \left(\frac{1}{2} + \delta_{n,d}\right)^{\widebar{B}_{d,\alpha}} \le \frac{\alpha}{d}.
\]
The last inequality here follows from~\eqref{eq:Wald-sample-size-condition}. Hence the union bound applies and proves the result.
\end{proof}
Similar to the right hand side of~\eqref{eq:consistency-subsampling}, the right hand side of~\eqref{eq:Wald-sample-size-condition} is be bounded away from zero unless $d/\alpha$ is a power of $2$. It is worth noting that, unlike~\eqref{eq:infinite-order-coverage}, the construction of the confidence region $\otimes_{k=1}^d\cup_{j=1}^{\widebar{B}_{d,\alpha}}\widehat{\mathrm{CI}}_{j}^{\texttt{Wald}, k}$ requires estimation of variance of the estimators and its validity requires (marginal) distributional convergence of the estimators $e_k^{\top}\widehat{\theta}_j$.
% \appendix
\section{Proof of Lemma~\ref{lem:fixed-B-coverage-result}}\label{appsec:fixed-B-coverage-result}
It is clear that
\begin{align*}
    \mathbb{P}\left(\theta_0 \notin \left[\min_{1\le j\le B}\widehat{\theta}_j,\,\max_{1\le j\le B}\widehat{\theta}_j\right]\right) &= \mathbb{P}\left(\theta_0 < \min_{1\le j\le B}\widehat{\theta}_j(k)\right) + \mathbb{P}\left(\max_{1\le j\le B}\widehat{\theta}_j < \theta_0\right)\\
    &= \prod_{j=1}^B \mathbb{P}(\widehat{\theta}_j > \theta_0) + \prod_{j=1}^B \mathbb{P}(\widehat{\theta}_j < \theta_0)\\
    &= \prod_{j=1}^B \left\{1 - \mathbb{P}(\widehat{\theta}_j \le \theta_0)\right\} + \prod_{j=1}^B \left\{1 - \mathbb{P}(\widehat{\theta}_j \ge \theta_0)\right\}.
\end{align*}
The result now follows from the definition~\eqref{eq:delta-definition-main-lemma} of $\Delta$. Note that the inequality in result stems from the fact that $\Delta$ in~\eqref{eq:delta-definition-main-lemma} is the maximum value over all estimators. Furthermore, we use the fact that $\mathbb{P}(\widehat{\theta}_j < \theta_0) + \mathbb{P}(\widehat{\theta}_j > \theta_0) \le 1$. If this inequality is strict, then the inequality in the coverage is strict.
\section{Proof of Theorem~\ref{thm:coverage-algorithm-known-Delta}}\label{appsec:coverage-algorithm-known-Delta}
The confidence interval from the \hulc is either based on $B_{\alpha,\Delta}$ estimators or based on $B_{\alpha,\Delta} - 1$ estimators. The number of estimators used depends on the realization of the uniform random variable $U$ in step 2 of the \hulcnospace. With probability $\tau_{\alpha,\Delta}$, the confidence interval will be based on $B_{\alpha,\Delta} - 1$ estimators and with probability $1 - \tau_{\alpha,\Delta}$, the confidence interval will be based on $B_{\alpha,\Delta}$ estimators. Hence from Lemma~\ref{lem:fixed-B-coverage-result}, we obtain
\begin{equation}\label{eq:miscoverage-randomized}
\begin{split}
\mathbb{P}\left(\theta_0\notin\widehat{\mathrm{CI}}_{\alpha,\Delta}\right) &= \tau_{\alpha,\Delta}\mathbb{P}\left(\theta_0 \notin \left[\min_{1\le j\le B_{\alpha,\Delta}-1}\widehat{\theta}_j, \max_{1\le j\le B_{\alpha,\Delta}-1}\widehat{\theta}_j\right]\right)\\
&\quad+ (1-\tau_{\alpha,\Delta})\mathbb{P}\left(\theta_0\notin\left[\min_{1\le j\le B_{\alpha,\Delta}}\widehat{\theta}_j, \max_{1\le j\le B_{\alpha,\Delta}}\widehat{\theta}_j\right]\right)\\
&\le \tau_{\alpha,\Delta} P(B_{\alpha,\Delta} - 1; \Delta) + (1-\tau_{\alpha,\Delta}) P(B_{\alpha,\Delta}; \Delta) ~=~ \alpha.
\end{split}
\end{equation}
This proves~\eqref{eq:over-coverage-in-general}. 
Under the additional assumptions for~\eqref{eq:exact-coverage-same-median-bias}, the only inequality in~\eqref{eq:miscoverage-randomized} becomes equality. Firstly, $\mathbb{P}(\widehat{\theta}_j = \theta_0) = 0$ for all $j\ge1$ implies that
\[
\mathbb{P}\left(\theta_0 \notin \left[\min_{1\le j\le B}\widehat{\theta}_j, \max_{1\le j\le B}\widehat{\theta}_j\right]\right) = \prod_{j=1}^B \mathbb{P}(\widehat{\theta}_j > \theta_0) + \prod_{j=1}^B (1 - \mathbb{P}(\widehat{\theta}_j > \theta_0)).
\]
Secondly, the assumption that all the estimators have a median bias of $\Delta$ exactly, implies that this probability is exactly $(1/2 - \Delta)^B + (1/2 + \Delta)^B$. This proves~\eqref{eq:exact-coverage-same-median-bias}.
\section{Proof of Proposition~\ref{prop:approximate-to-exact-validity}}\label{appsec:approximate-to-exact-validity}
% {\clr 
% \begin{proof}[New Proof]
If $\widetilde{\Delta} = \Delta$, then the conclusion is obvious. Consider the case, $\widetilde{\Delta} \neq \Delta$.
Note that $P(B; \Delta)$ is an increasing function of $\Delta$. Hence, for any $0 \le \Delta \le \widetilde{\Delta} < 1/2$, we have
\begin{equation}\label{eq:simple-ratio-inequality}
\begin{split}
1 \le \frac{P(B; \widetilde{\Delta})}{P(B; \Delta)} &= \frac{(1 + 2\widetilde{\Delta})^B + (1 - 2\widetilde{\Delta})^B}{(1 + 2\Delta)^B + (1 - 2\Delta)^B}\\
&\le \max\left\{\frac{(1 + 2\widetilde{\Delta})^B}{(1 + 2\Delta)^B},\,\frac{(1 - 2\widetilde{\Delta})^B}{(1 - 2\Delta)^B}\right\} = \left(\frac{1 + 2\widetilde{\Delta}}{1 + 2\Delta}\right)^B = \left(1 + \frac{2(\widetilde{\Delta} - \Delta)}{(1 + 2\Delta)}\right)^B\\ 
&\le \left(1 + 2(\widetilde{\Delta} - \Delta)\right)^B.
\end{split}
\end{equation}
Reversing the roles of $\widetilde{\Delta}$ and $\Delta$, we conclude that if $0 \le \widetilde{\Delta} \le \Delta < 1/2$, then
\[
\left(1 + 2(\Delta - \widetilde{\Delta})\right)^{-B} \le \frac{P(B; \widetilde{\Delta})}{P(B; \Delta)} \le 1.
\]
Hence, for all $\widetilde{\Delta}, \Delta\in[0,1/2)$, and all $B\ge1$, we have
\begin{equation}\label{eq:main-inequality-P-B-Delta}
P(B; \Delta)\left(1 + 2|\Delta - \widetilde{\Delta}|\right)^{-B} ~\le~ P(B; \widetilde{\Delta}) ~\le~ P(B; \Delta)\left(1 + 2|\Delta - \widetilde{\Delta}|\right)^B. 
\end{equation}
Because $P(B_{\alpha,\Delta};\Delta) < \alpha$, using~\eqref{eq:main-inequality-P-B-Delta} with $B = B_{\alpha,\Delta}$, we get $P(B_{\alpha,\Delta}; \widetilde{\Delta}) \le \alpha$ if
\begin{equation}\label{eq:upper-bound}
\left(1 + 2|\Delta - \widetilde{\Delta}|\right)^{B_{\alpha,\Delta}} \le \frac{\alpha}{P(B_{\alpha,\Delta}; \Delta)}.
\end{equation}
Similarly, $P(B_{\alpha,\Delta} - 1; \Delta) > \alpha$, using~\eqref{eq:main-inequality-P-B-Delta} with $B = B_{\alpha,\Delta} - 1$, we get $P(B_{\alpha,\Delta} - 1; \widetilde{\Delta}) > \alpha$ if
\begin{equation}\label{eq:lower-bound}
\left(1 + 2|\Delta - \widetilde{\Delta}|\right)^{B_{\alpha,\Delta} - 1} < \frac{P(B_{\alpha,\Delta} - 1; \Delta)}{\alpha}.
\end{equation}
Note that both the ratios on the right hand side of~\eqref{eq:upper-bound} and~\eqref{eq:lower-bound} are at least $1$. Furthermore, inequality~\eqref{eq:lower-bound} will be satisfied if $(1 + 2|\Delta - \widetilde{\Delta}|)^{B_{\alpha,\Delta}} \le {P(B_{\alpha,\Delta} - 1; \Delta)}/{\alpha}.$
Combining inequalities~\eqref{eq:upper-bound} and~\eqref{eq:lower-bound}, we get that if~\eqref{eq:consistency-subsampling} holds true, then $P(B_{\alpha,\Delta}; \widetilde{\Delta}) \le \alpha < P(B_{\alpha,\Delta}-1;\widetilde{\Delta})$ and hence $B_{\alpha,\Delta} = B_{\alpha,\widetilde{\Delta}}$.

Note that in inequality~\eqref{eq:simple-ratio-inequality}, we used an inequality for the ratio $P(B;\widetilde{\Delta})/P(B;\Delta)$. Observe that if $\Delta = 0$, then $P(B; \widetilde{\Delta})/P(B; \Delta)$ has zero derivative at $\widetilde{\Delta} = 0$. Observe that, with $\Delta = 0$, 
\begin{align*}
1\le \frac{P(B; \widetilde{\Delta})}{P(B; \Delta)} &= \frac{(1 + 2\widetilde{\Delta})^B + (1 - 2\widetilde{\Delta})^B}{2}\\ 
&= 1 + 0\widetilde{\Delta} + 2B(B-1)\widetilde{\Delta}^2\left[\frac{(1 + 2\widetilde{\Delta}^*)^{B - 2} + (1 - 2\widetilde{\Delta}^*)^{B - 2}}{2}\right],
\end{align*}
for some $\widetilde{\Delta}^*\in[0, \widetilde{\Delta}]$. For $B = 1$, the right most term is zero and for $B \ge 2,$ we have that
\[
\frac{(1 + 2\widetilde{\Delta}^*)^{B - 2} + (1 - 2\widetilde{\Delta}^*)^{B - 2}}{2} \le (1 + 2\widetilde{\Delta}^*)^{B - 2}.
\] 
Because $\widetilde{\Delta}^* \le \Delta$, we obtain for all $B \ge 1$, 
\begin{equation}\label{eq:Delta-zero-inequality}
1 ~\le~ \frac{P(B; \widetilde{\Delta})}{P(B; 0)} ~\le~ 1 + 2B(B-1)\widetilde{\Delta}^2(1 + 2\widetilde{\Delta})^{(B - 2)_+}.
\end{equation}
By definition $P(B_{\alpha,0} - 1; 0) > \alpha$. Taking $B = B_{\alpha,0} - 1$ in~\eqref{eq:Delta-zero-inequality}, we obtain
\begin{equation}\label{eq:lower-bound-zero-Delta}
P(B_{\alpha,0} - 1; \widetilde{\Delta}) \ge P(B_{\alpha,0} - 1; 0) > \alpha \quad\Rightarrow\quad P(B_{\alpha, 0} - 1; \widetilde{\Delta}) > \alpha.
\end{equation}
Again by definition $P(B_{\alpha,0}; 0) \le \alpha$. Taking $B = B_{\alpha,0}$ in~\eqref{eq:Delta-zero-inequality}, we obtain
\begin{equation}\label{eq:upper-bound-zero-Delta}
P(B_{\alpha,0}; \widetilde{\Delta}) \le \alpha,\quad\mbox{if}\quad 2B_{\alpha,0}(B_{\alpha,0} - 1)\widetilde{\Delta}^2(1 + 2\widetilde{\Delta})^{(B_{\alpha,0} - 2)_+} \le \frac{\alpha}{P(B_{\alpha,0};0)} - 1.
\end{equation}
Combining~\eqref{eq:lower-bound-zero-Delta} and~\eqref{eq:upper-bound-zero-Delta}, we obtain~\eqref{eq:consistency-subsampling-asymptotic-median-unbiased}.
\begin{figure}[!h]
\centering
\includegraphics[width=\textwidth, height=2.9in]{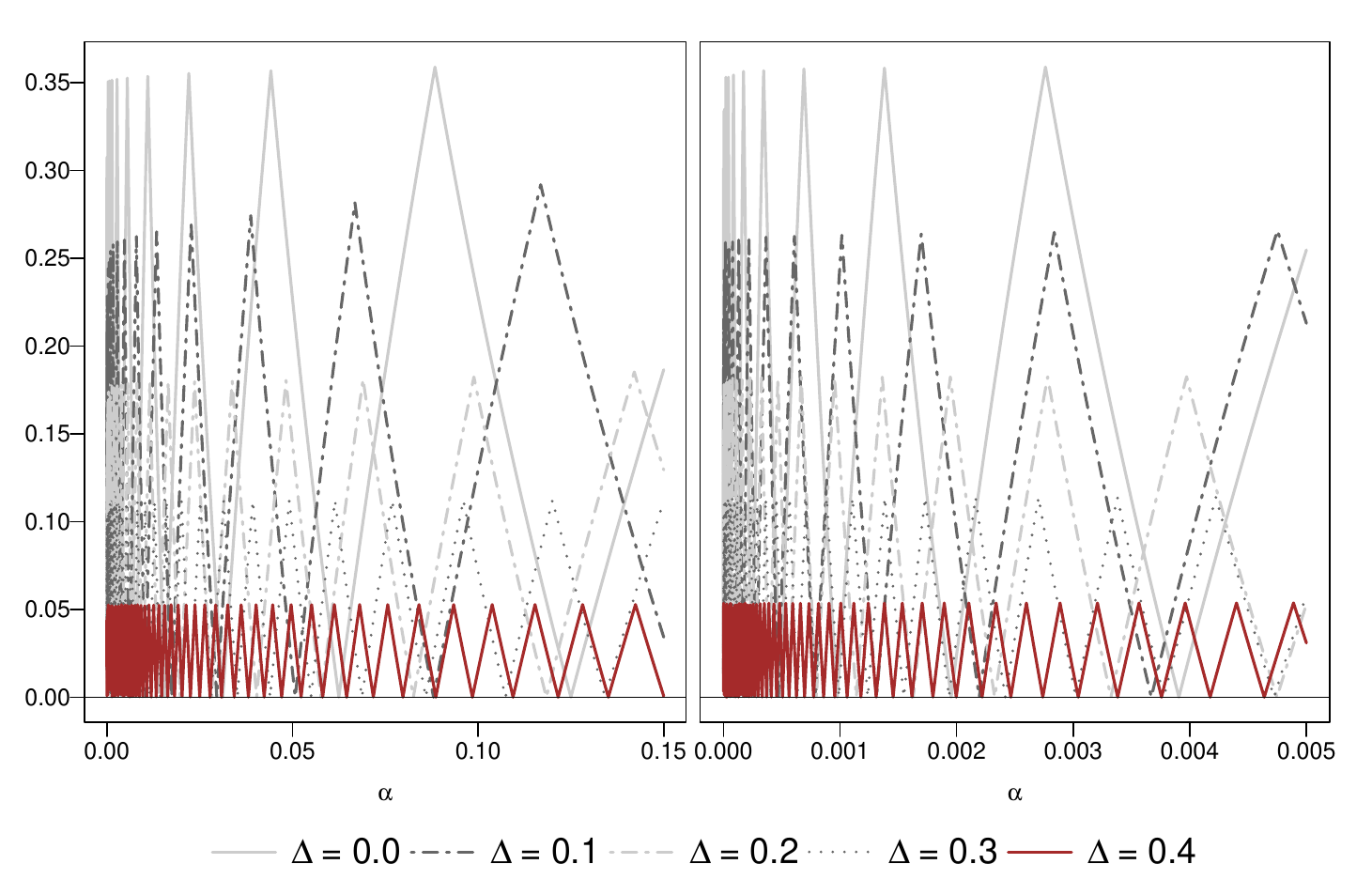}
\vspace{-0.2in}
\caption{The plots show the right hand side of~\eqref{eq:consistency-subsampling} on the $y$-axis as $\alpha$ changes from $0$ to $0.15$ and $\Delta\in\{0, 0.1, 0.2, 0.3, 0.4\}$. In the left panel, we show the plot for $\alpha\in(0, 0.15)$ and in the right panel, we show the plot for $\alpha\in(0, 0.005)$. The $y$-axis limits remain the same for both plots.}
\label{fig:right-hand-side-of-consistency-subsampling-scaled}
\end{figure}
\begin{figure}[!h]
\centering
\includegraphics[width=\textwidth, height=2.9in]{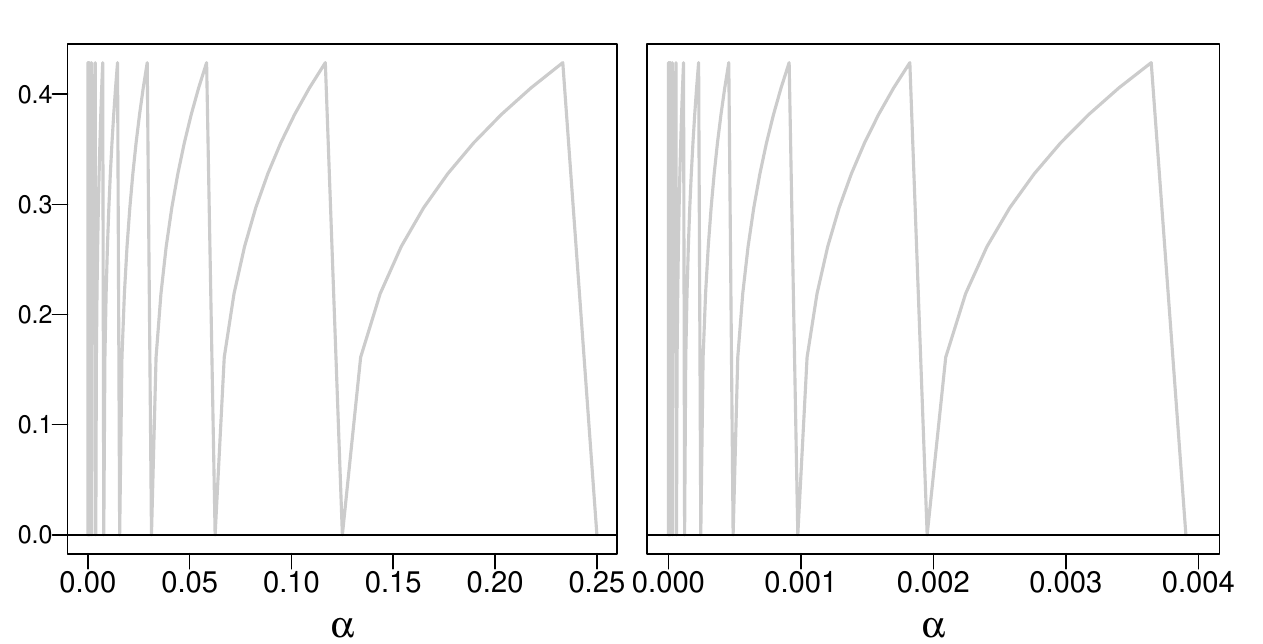}
% wrong requirement {Codes_and_Plots/requirement_proposition_1_part_2_final.pdf}
\vspace{-0.2in}
\caption{The plots show the right hand side of~\eqref{eq:relaxed-requirement-median-unbiased} on the $y$-axis as $\alpha$ changes from $0$ to $0.15$. In the left panel, we show the plot for $\alpha\in(0, 0.15)$ and in the right panel, we show the plot for $\alpha\in(0, 0.004)$. The $y$-axis limits remain the same for both plots.}
\label{fig:right-hand-side-of-consistency-subsampling-median-unbiased}
\end{figure}

\section{Proof of Theorem~\ref{thm:infinite-order-coverage}}\label{appsec:infinite-order-coverage}
Because the median bias of estimators from $\mathcal{A}(\cdot)$ is $\widetilde{\Delta}$, we set
\[
\widetilde{\Delta} ~\ge~ \max_{1\le j\le B^*}\left(\frac{1}{2} - \max\left\{\mathbb{P}(\widehat{\theta}_j \ge \theta_0),\,\mathbb{P}(\widehat{\theta}_j \le \theta_0)\right\}\right)_+.
\]
Lemma~\ref{lem:fixed-B-coverage-result} implies that
\[
\mathbb{P}(\theta_0 \notin \widehat{\mathrm{CI}}_{\alpha,0}) \le \mathbb{E}\left[\left(\frac{1}{2} + \widetilde{\Delta}\right)^{B^*} + \left(\frac{1}{2} - \widetilde{\Delta}\right)^{B^*}\right].
\]
The right hand side involves an expectation because $B^*$ is a random variable satisfying
\[
\mathbb{P}(B^* = B_{\alpha,0}) = 1 - \tau_{\alpha,0},\quad\mbox{and}\quad \mathbb{P}(B^* = B_{\alpha,0} - 1) = \tau_{\alpha,0}.
\]
This follows from~\eqref{eq:definition-tau}. Therefore,
\begin{equation}\label{eq:penultimate-asymptotic-coverage}
\mathbb{P}(\theta_0 \notin \widehat{\mathrm{CI}}_{\alpha,0}) \le \tau_{\alpha,0}P(B_{\alpha,0} - 1; \widetilde{\Delta}) + (1 - \tau_{\alpha,0})P(B_{\alpha,0}; \widetilde{\Delta}).
\end{equation}
From the proof of Proposition~\ref{prop:approximate-to-exact-validity} (in particular~\eqref{eq:Delta-zero-inequality}), it follows that
\begin{equation}\label{eq:ratio-inequality-median-unbiased}
\max\left\{\frac{P(B_{\alpha,0}; \widetilde{\Delta})}{P(B_{\alpha,0}; 0)},\,\frac{P(B_{\alpha,0} - 1; \widetilde{\Delta})}{P(B_{\alpha,0}-1;0)}\right\} \le 1 + 2B_{\alpha,0}(B_{\alpha,0} - 1)\widetilde{\Delta}^2(1 + 2\widetilde{\Delta})^{(B_{\alpha,0} - 2)_+}.
\end{equation}
Substituting this inequality in~\eqref{eq:penultimate-asymptotic-coverage} yields
\begin{align*}
\mathbb{P}(\theta_0 \notin \widehat{\mathrm{CI}}_{\alpha,0}) &\le \tau_{\alpha,0}P(B_{\alpha,0} - 1; 0)\left(1 + 2B_{\alpha,0}(B_{\alpha,0} - 1)\widetilde{\Delta}^2(1 + 2\widetilde{\Delta})^{(B_{\alpha,0} - 2)_+}\right)\\ 
&\qquad+ (1-\tau_{\alpha,0})P(B_{\alpha,0}; 0)\left(1 + 2B_{\alpha,0}(B_{\alpha,0} - 1)\widetilde{\Delta}^2(1 + 2\widetilde{\Delta})^{(B_{\alpha,0} - 2)_+}\right)\\
&= \left[\tau_{\alpha,0}P(B_{\alpha,0}-1;0) + (1-\tau_{\alpha,0})P(B_{\alpha,0}; 0)\right]\left(1 + 2B_{\alpha,0}(B_{\alpha,0} - 1)\widetilde{\Delta}^2(1 + 2\widetilde{\Delta})^{(B_{\alpha,0} - 2)_+}\right)\\
&= \alpha\left(1 + 2B_{\alpha,0}(B_{\alpha,0} - 1)\widetilde{\Delta}^2(1 + 2\widetilde{\Delta})^{(B_{\alpha,0} - 2)_+}\right).
\end{align*}
The last equality follows from the definition~\eqref{eq:definition-tau} of $\tau_{\alpha,0}$. This completes the proof of upper bound in~\eqref{eq:some-order-coverage-upper-bound}. 

From the proof of Lemma~\ref{lem:fixed-B-coverage-result}, it follows, under the assumption of $\mathbb{P}(\widehat{\theta}_j = \theta_0) = 0$ and the exact median bias of $\widetilde{\Delta}$, that
\[
\mathbb{P}(\theta_0 \notin \widehat{\mathrm{CI}}_{\alpha,\Delta}) = \mathbb{E}\left[\left(\frac{1}{2} - \widetilde{\Delta}\right)^{B^*} + \left(\frac{1}{2} - \widetilde{\Delta}\right)^{B^*}\right] \ge \mathbb{E}\left[\frac{2}{2^{B^* - 1}}\right] = \alpha,
\]
the last equality follows again from the definition of $\tau_{\alpha,0}$ in~\eqref{eq:definition-tau}. This completes the proof of~\eqref{eq:some-order-coverage-lower-bound}.

For the case where the estimators have an asymptotic median bias of $\Delta$ ($\neq 0$), we use the \hulc to obtain $\widehat{\mathrm{CI}}_{\alpha,\Delta}$, while the true finite sample median bias is bounded by $\widetilde{\Delta}$. In this case, to prove the upper bound, we use the inequality
\[
\max\left\{\frac{P(B_{\alpha,\Delta}; \widetilde{\Delta})}{P(B_{\alpha,\Delta}; \Delta)},\,\frac{P(B_{\alpha,\Delta}-1; \widetilde{\Delta})}{P(B_{\alpha,\Delta} - 1; \Delta)}\right\} \le \left(1 + 2|\widetilde{\Delta} - \Delta|\right)^{B_{\alpha,\Delta}},
\]
in place of~\eqref{eq:ratio-inequality-median-unbiased}. Using this inequality in~\eqref{eq:penultimate-asymptotic-coverage} (with $\tau_{\alpha,0}$ replaced by $\tau_{\alpha,\Delta}$), we obtain
\begin{align*}
\mathbb{P}(\theta_0 \notin \widehat{\mathrm{CI}}_{\alpha,\Delta}) &\le \left[\tau_{\alpha,\Delta}P(B_{\alpha,\Delta} - 1; \Delta) + (1 - \tau_{\alpha,\Delta})P(B_{\alpha,\Delta}; \Delta)\right]\left(1 + 2|\widetilde{\Delta} - \Delta|\right)^{B_{\alpha,\Delta}}\\ 
&= \alpha\left(1 + 2|\widetilde{\Delta} - \Delta|\right)^{B_{\alpha,\Delta}}.
\end{align*}
From the proof of Lemma~\ref{lem:fixed-B-coverage-result}, it follows, under the assumption of $\mathbb{P}(\widehat{\theta}_j = \theta_0) = 0$ and the exact median bias of $\widetilde{\Delta}$, that
\begin{align*}
\mathbb{P}(\theta_0 \notin \widehat{\mathrm{CI}}_{\alpha,\Delta}) 
&= P(B_{\alpha,\Delta}; \widetilde{\Delta})(1 - \tau_{\alpha,\Delta}) + P(B_{\alpha,\Delta} - 1; \widetilde{\Delta})\tau_{\alpha,\Delta}\\
&\ge \left[P(B_{\alpha,\Delta}; \Delta)(1 - \tau_{\alpha,\Delta}) + P(B_{\alpha,\Delta} - 1; \Delta)\tau_{\alpha,\Delta}\right]\left(1 + 2|\widetilde{\Delta} - \Delta|\right)^{-B_{\alpha,\Delta}},\quad\mbox{using~\eqref{eq:main-inequality-P-B-Delta},}\\
&= \alpha\left(1 + 2|\widetilde{\Delta} - \Delta|\right)^{-B_{\alpha,\Delta}}.
\end{align*}
This proves the upper and lower bounds for a non-zero asymptotic median bias of $\Delta$.
\section{Proof of Lemma~\ref{lem:fixed-B-coverage-result-multivariate}}\label{appsec:fixed-B-coverage-result-multivariate}
Equality~\eqref{eq:half-space-symmetry-coverage} follows from Theorem 3 of~\cite{hayakawa2021estimating}. This result was originally proved in~\cite{wendel1962problem} under symmetry of $\widehat{\theta}_j - \theta_0$. \cite{wagner2001continuous} proved that the miscoverage probability is lower bounded by the quantity on the right hand side whenever $\widehat{\theta}_j - \theta_0$ has an absolutely continuous distribution; this does not require the assumption of $\mathbb{P}(c^{\top}(\widehat{\theta}_j - \theta_0) \le 0) = 1/2$ for all $c\in\mathbb{R}^d\setminus\{0\}$.

Equality~\eqref{eq:each-coordinate-median-bias-coverage} follows readily from Lemma~\ref{lem:fixed-B-coverage-result} and union bound. Formally,
\begin{align*}
\mathbb{P}\left(\theta_0 \notin \mbox{RectHull}(\{\widehat{\theta}_j: 1\le j\le B\})\right) &= \mathbb{P}\left(\bigcup_{k=1}^d \left\{e_k^{\top}\theta_0 \notin \left[\min_{1\le j\le B}e_k^{\top}\widehat{\theta}_j,\,\max_{1\le j\le B}e_k^{\top}\widehat{\theta}_j\right]\right\}\right)\\
&\le \sum_{k=1}^d \mathbb{P}\left(e_k^{\top}\theta_0 \notin \left[\min_{1\le j\le B}e_k^{\top}\widehat{\theta}_j,\,\max_{1\le j\le B}e_k^{\top}\widehat{\theta}_j\right]\right)\\
&\le \sum_{k=1}^d \left\{\left(\frac{1}{2} - \Delta_k\right)^{B} + \left(\frac{1}{2} + \Delta_k\right)^B\right\},
\end{align*}
where $\Delta_k$ is an upper bound on the median bias of $e_k^{\top}\widehat{\theta}_j$ for $1\le j\le B$. Hence inequality~\eqref{eq:each-coordinate-median-bias-coverage} follows.

\section{Proof of Proposition~\ref{prop:median-bias-mean-square}}\label{appsec:median-bias-mean-square}

For notational convenience and without loss of generality, we will prove the result when $\widehat{\theta}_j$ is computed based on $n$ observations. Set
\[
\widehat{T} = \frac{1}{\binom{n}{2}} \sum_{i < j} X_i X_j. 
\]
This can be rewritten as
\[
\widehat{T} = \frac{n}{(n-1)} \left( \frac{1}{n} \sum_{i=1}^n X_i \right)^2 - \frac{1}{(n-1)} \left( \frac{1}{n} \sum_{i=1}^n X_i^2  \right).
\]
In terms of $\xi_i$ and $Z = n^{-1/2}\sum_{i=1}^n \xi_i$, this becomes
\[
\widehat{T} - \mu^2 = \frac{2 \mu \sigma Z}{\sqrt{n}} + \frac{\sigma^2}{n-1} \left[ Z^2 - 1 \right] + \frac{\sigma^2}{n-1}\left[ 1 - \frac{1}{n} \sum_{i=1}^n \xi_i^2\right].
\]
It is clear that
\[
\frac{\widehat{T}}{\sigma^2} - \frac{\mu^2}{\sigma^2} ~=~ 2\frac{\mu}{\sigma}\frac{Z}{\sqrt{n}} + \frac{Z^2 - 1}{n-1} + \frac{1}{\sqrt{n}(n-1)}\left[\sqrt{n}\left(\frac{1}{n}\sum_{i=1}^n {\xi_i^2} - 1\right)\right].
\]
Set $W = \sqrt{n}(n^{-1}\sum_{i=1}^n \xi_i^2 - 1) = O_p(1)$. By the univariate Berry--Esseen bound, we get
\begin{equation}\label{eq:BE-bound-quadratic}
\left|\mathbb{P}\left(W \le t\right) - \mathbb{P}\left(N(0, \mathbb{E}[\xi^4]) \le t\right)\right| \le \frac{\mathbb{E}[\xi^6]}{(\mathbb{E}[\xi^4])^{3/2}\sqrt{n}}.
\end{equation}
Similarly,
\begin{equation}\label{eq:BE-bound-Z}
\left|\mathbb{P}(Z \le t) - \mathbb{P}(N(0, 1) \le t)\right| \le \frac{\mathbb{E}[|\xi|^3]}{(\mathbb{E}[\xi^2])^{3/2}\sqrt{n}}.
\end{equation}
Inequality~\eqref{eq:BE-bound-quadratic} implies that 
\[
\mathbb{P}\left(|W| > (3\mathbb{E}[\xi^4]\log n)^{1/2}\right) \le \frac{1}{n} + \frac{\mathbb{E}[\xi^6]}{(\mathbb{E}[\xi^4])^{3/2}\sqrt{n}}.
\]
Clearly, 
\begin{equation}\label{eq:less-median-bias}
\begin{split}
\mathbb{P}\left(\widehat{T} \le \mu^2\right) &= \mathbb{P}\left(2\frac{\mu}{\sigma} Z + \frac{\sqrt{n}}{n-1}(Z^2 - 1) + \frac{W}{n-1} \le 0\right)\\
&= \mathbb{P}\left(2\frac{\mu}{\sigma} Z + \frac{\sqrt{n}}{n-1}(Z^2 - 1) + \frac{W}{n-1} \le 0,\, |W| \le (3\mathbb{E}[\xi^4]\log n)^{1/2}\right)\\
&\quad+ \mathbb{P}\left(2\frac{\mu}{\sigma} Z + \frac{\sqrt{n}}{n-1}(Z^2 - 1) + \frac{W}{n-1} \le 0,\, |W| > (3\mathbb{E}[\xi^4]\log n)^{1/2}\right)\\
&\le \mathbb{P}\left(2\frac{\mu}{\sigma} Z + \frac{\sqrt{n}}{n-1}(Z^2 - 1) \le \frac{(2\mathbb{E}[\xi^4]\log n)^{1/2}}{n-1}\right)\\ 
&\quad+ \mathbb{P}(|W| > (3\mathbb{E}[\xi^4]\log n)^{1/2}).
\end{split}
\end{equation}
Similarly,
\begin{equation}\label{eq:more-median-bias}
\begin{split}
\mathbb{P}(\widehat{T} \le \mu^2) &\ge \mathbb{P}\left(2\frac{\mu}{\sigma} Z + \frac{\sqrt{n}}{n-1}(Z^2 - 1) \le -\frac{(3\mathbb{E}[\xi^4]\log n)^{1/2}}{n-1}\right)\\ 
&\quad- \mathbb{P}(|W| > (3\mathbb{E}[\xi^4]\log n)^{1/2}).
\end{split}
\end{equation}
Note that
\begin{equation}\label{eq:event-equivalence}
\begin{split}
\{a Z + b(Z^2 - 1) \le c\} &\equiv \left\{\frac{-a - \sqrt{a^2 + 4b(b+c)}}{2b} \le Z \le \frac{-a + \sqrt{a^2 + 4b(b+c)}}{2b}\right\}\\
&\equiv \left\{\left|Z + \frac{a}{2b}\right| \le \frac{\sqrt{a^2 + 4b(b + c)}}{2b}\right\}.
\end{split}
\end{equation}
Using inequality~\eqref{eq:BE-bound-Z}, we obtain
\begin{align*}
\mathbb{P}\left(a Z + b(Z^2 - 1) \le c\right) &= \mathbb{P}\left(\left|Z + \frac{a}{2b}\right| \le \frac{\sqrt{a^2 + 4b(b + c)}}{2b}\right)\\
&= \mathbb{P}\left(\left|N(0, 1) + \frac{a}{2b}\right| \le \frac{\sqrt{a^2 + 4b(b + c)}}{2b}\right) \pm \frac{\mathbb{E}[\xi^3]}{(\mathbb{E}[\xi^2])^{3/2}\sqrt{n}}\\
&= \Phi\left(\frac{-a + \sqrt{a^2 + 4b(b+c)}}{2b}\right) - \Phi\left(\frac{-a-\sqrt{a^2 + 4b(b+c)}}{2b}\right)\\ 
&\quad\pm \frac{\mathbb{E}[|\xi|^3]}{(\mathbb{E}[\xi^2])^{3/2}\sqrt{n}}.
\end{align*}
Finally, note that
\begin{align*}
\left|\Phi\left(\frac{-a \pm \sqrt{a^2 + 4b(b\pm c)}}{2b}\right) - \Phi\left(\frac{-a \pm \sqrt{a^2 + 4b^2}}{2b}\right)\right| &\le \frac{1}{\sqrt{2\pi}}\left|\sqrt{\frac{a^2 + 4b^2 \pm 4bc}{4b^2}} - \sqrt{\frac{a^2 + 4b^2}{4b^2}}\right|\\
&\le \frac{1}{\sqrt{2\pi}}\times\frac{|c|}{b}. 
\end{align*}
Combining the two inequalities above, we get
\begin{equation}\label{eq:final-inequality-quadratic}
\begin{split}
\mathbb{P}\left(a Z + b(Z^2 - 1) \le c\right) &= \Phi\left(\frac{-a + \sqrt{a^2 + 4b^2}}{2b}\right) - \Phi\left(\frac{-a-\sqrt{a^2 + 4b^2}}{2b}\right)\\
&\quad\pm \frac{2}{\sqrt{2\pi}}\frac{|c|}{b} \pm \frac{\mathbb{E}[|\xi|^3]}{(\mathbb{E}[\xi^2])^{3/2}\sqrt{n}}\\
&= \mathbb{P}(aZ + b(Z^2 - 1) \le 0) \pm \frac{2}{\sqrt{2\pi}}\frac{|c|}{b} \pm \frac{\mathbb{E}[|\xi|^3]}{(\mathbb{E}[\xi^2])^{3/2}\sqrt{n}}.
\end{split}
\end{equation}
Substituting these inequalities in~\eqref{eq:less-median-bias} and~\eqref{eq:more-median-bias}, we conclude
\begin{equation}\label{eq:final-bound-median-bias}
\begin{split}
&\left|\mathbb{P}(\widehat{T} \le \mu^2) - \left\{\Phi\left(\frac{-a + \sqrt{a^2 + 4b^2}}{2b}\right) - \Phi\left(\frac{-a-\sqrt{a^2 + 4b^2}}{2b}\right)\right\}\right|\\ 
&\quad\le \sqrt{\frac{2}{\pi}}\frac{|c|}{b} + \frac{\mathbb{E}[|\xi|^3]}{(\mathbb{E}[\xi^2])^{3/2}\sqrt{n}} + \frac{2}{n} + \frac{\mathbb{E}[\xi^6]}{(\mathbb{E}[\xi^4])^{3/2}\sqrt{n}}.
\end{split}
\end{equation}
% The last term $1/n$ follows as a bound on $\mathbb{P}(|W| > (2\mathbb{E}[\xi^4]\log n)^{1/2})$. 
Here 
\[
a = 2\frac{\mu}{\sigma},\quad b = \frac{\sqrt{n}}{n-1},\quad\mbox{and}\quad c = \pm\frac{\sqrt{3\mathbb{E}[\xi^4]\log n}}{n-1}.
\]
This implies
\[
\frac{|c|}{b} = \sqrt{\frac{3\mathbb{E}[\xi^4]\log n}{n}}.
\]
Note that the right hand side of~\eqref{eq:final-bound-median-bias} is of order $n^{-1/2}$ and does not depend on $\mu$; it only depends on $\mathbb{E}[|\xi|^j], j = 3, 4, 6$. Inequality~\eqref{eq:final-bound-median-bias} implies that the median bias of $\widehat{T}$ can be obtained by taking the maximum over all $\theta\in\mathbb{R}$ of
\begin{align*}
&\left|\frac{1}{2} - \left\{\Phi\left(\frac{-2\theta + \sqrt{4\theta^2 + 4n/(n-1)^2}}{2\sqrt{n}/(n-1)}\right) - \Phi\left(\frac{-2\theta-\sqrt{4\theta^2 + 4n/(n-1)^2}}{2\sqrt{n}/(n-1)}\right)\right\}\right|\\
&\quad= \left|\frac{1}{2} - \left\{\Phi\left(\frac{-\theta + \sqrt{\theta^2 + n/(n-1)^2}}{\sqrt{n}/(n-1)}\right) - \Phi\left(\frac{-\theta-\sqrt{\theta^2 + n/(n-1)^2}}{\sqrt{n}/(n-1)}\right)\right\}\right|.
\end{align*}
It seems the maximum is attained at $\theta = 0$ for any $n$. 
% \end{proof}
\section{Proof of Theorem~\ref{thm:subsampling-hull-validity}}\label{appsec:subsampling-hull-validity}
Throughout the proof, we write $\widehat{\Delta}$ instead of $\Delta$ for convenience. 
Define the event
\[
\mathcal{E} := \{B_{\alpha,\widehat{\Delta}_n} = B_{\alpha,\Delta}\}.
\]
and set
\begin{equation}\label{eq:CIs-under-with-true-Delta}
\widetilde{\mathrm{CI}}_{\alpha}^{(0)} := \left[\min_{1\le j\le B_{\alpha,\Delta} - 1}\widehat{\theta}_j, \min_{1\le j\le B_{\alpha,\Delta} - 1}\widehat{\theta}_j\right],\quad\mbox{and}\quad \widetilde{\mathrm{CI}}_{\alpha}^{(1)} := \left[\min_{1\le j\le B_{\alpha,\Delta}}\widehat{\theta}_j, \min_{1\le j\le B_{\alpha,\Delta}}\widehat{\theta}_j\right].
\end{equation}
Recall $\tau_{\alpha,\Delta}$ from~\eqref{eq:definition-tau}. On the event $\mathcal{E}$, we get that 
\[
\tau_{\alpha,\widehat{\Delta}} = \frac{\alpha - P(B_{\alpha,\Delta}; \widehat{\Delta})}{P(B_{\alpha,\Delta} - 1; \widehat{\Delta}) - P(B_{\alpha,\Delta}; \widehat{\Delta})},\quad\mbox{and}\quad \widehat{\mathrm{CI}}_{\alpha}^{(0)} = \widetilde{\mathrm{CI}}_{\alpha}^{(0)},\quad \widehat{\mathrm{CI}}_{\alpha}^{(1)} = \widetilde{\mathrm{CI}}_{\alpha}^{(1)}.
\]
We first bound the miscoverage probabilities of $\widehat{\mathrm{CI}}_{\alpha}^{(0)}$ and $\widehat{\mathrm{CI}}_{\alpha}^{(1)}$ when event $\mathcal{E}$ occurs. From the definition of $\Delta_{n,\alpha}$ and Lemma~\ref{lem:fixed-B-coverage-result},
\begin{equation}\label{eq:miscoverage-B-minus-1-subsampling}
\begin{split}
\mathbb{P}(\{\theta_0 \notin \widehat{\mathrm{CI}}_{\alpha}^{(0)}\}\cap\mathcal{E}) &\le \mathbb{P}(\theta_0 \notin \widetilde{\mathrm{CI}}_{\alpha}^{(0)})\\
&\le P(B_{\alpha,\Delta} - 1; \Delta_{n,\alpha})\\
&= \frac{P(B_{\alpha,\Delta} - 1; \Delta_{n,\alpha})}{P(B_{\alpha,\Delta} - 1; \Delta)}P(B_{\alpha,\Delta} - 1; \Delta)\\
&\le P(B_{\alpha,\Delta} - 1; \Delta)\times\begin{cases}(1 + 2(B_{\alpha,0} - 1)(B_{\alpha,0} - 2)\Delta_{n,\alpha}^2(1 + 2\Delta_{n,\alpha})^{B_{\alpha,0}}), &\mbox{if }\Delta = 0,\\
(1 + 2|\Delta_{n,\alpha} - \Delta|)^{B_{\alpha,\Delta} - 1}, &\mbox{if }\Delta \neq 0\end{cases}\\ 
&\le {2\alpha}\times\begin{cases}(1 + 2(B_{\alpha,0} - 1)(B_{\alpha,0} - 2)\Delta_{n,\alpha}^2(1 + 2\Delta_{n,\alpha})^{B_{\alpha,0}}), &\mbox{if }\Delta = 0,\\
(1 + 2|\Delta_{n,\alpha} - \Delta|)^{B_{\alpha,\Delta} - 1}, &\mbox{if }\Delta \neq 0\end{cases}.
\end{split}
\end{equation}
The first inequality follows from proof of Proposition~\ref{prop:approximate-to-exact-validity} (in particular~\eqref{eq:main-inequality-P-B-Delta}) while the second inequality follows from the fact that $P(B_{\alpha,\Delta}; \Delta) \le \alpha$ and $P(B_{\alpha,\Delta} - 1; \Delta)/P(B_{\alpha,\Delta}; \Delta) \le 2$ for any $\Delta\in[0, 1/2]$. Similarly,
\begin{equation}\label{eq:miscoverage-B-subsampling}
\begin{split}
&\mathbb{P}(\{\theta_0 \notin \widehat{\mathrm{CI}}_{\alpha}^{(1)}\}\cap\mathcal{E})\\
&\le P(B_{\alpha,\Delta}; \Delta_{n,\alpha})\\ 
&= P(B_{\alpha,\Delta}; \Delta)\frac{P(B_{\alpha,\Delta}; \Delta_{n,\alpha})}{P(B_{\alpha,\Delta}; \Delta)} \le \alpha\times\begin{cases}(1 + 2B_{\alpha,0}(B_{\alpha,0} - 1)\Delta_{n,\alpha}^2(1 + 2\Delta_{n,\alpha})^{B_{\alpha,0}}), &\mbox{if }\Delta = 0,\\
(1 + 2|\Delta_{n,\alpha} - \Delta|)^{B_{\alpha,\Delta}}, &\mbox{if }\Delta \neq 0\end{cases}.
\end{split}
\end{equation}
This completes the proof of~\eqref{eq:fixed-choice-B-unknown-Delta}. 
We will now prove a bound on the miscoverage probability of $\widehat{\mathrm{CI}}_{\alpha}^{\mathrm{sub}}$.
Writing $\mathbb{E}_{U}[\cdot]$ to represent the expectation with respect to $U$ which is a uniform random variable independent of the data, we get
\begin{equation}\label{eq:miscoverage-subsampling}
\mathbb{E}_U[\mathbbm{1}\{\theta_0 \notin \widehat{\mathrm{CI}}_{\alpha}^{\mathrm{sub}}\}]\mathbbm{1}\{\mathcal{E}\} ~=~ \tau_{\alpha,\widehat{\Delta}}\mathbbm{1}\{\theta_0 \notin \widehat{\mathrm{CI}}_{\alpha}^{(0)}\}\mathbbm{1}\{\mathcal{E}\} + (1 - \tau_{\alpha,\widehat{\Delta}})\mathbbm{1}\{\theta_0 \notin \widehat{\mathrm{CI}}_{\alpha}^{(1)}\}\mathbbm{1}\{\mathcal{E}\}.
\end{equation}
Note that the miscoverage probability of $\widehat{\mathrm{CI}}_{\alpha}^{\mathrm{sub}}$ is $\mathbb{E}[\mathbb{E}_{U}[\mathbbm{1}\{\theta_0 \notin \widehat{\mathrm{CI}}_{\alpha}^{\mathrm{sub}}\}]]$. Using inequalities~\eqref{eq:miscoverage-B-minus-1-subsampling} and~\eqref{eq:miscoverage-B-subsampling}, we can readily obtain
\[
\mathbb{P}(\theta_0 \notin \widehat{\mathrm{CI}}_{\alpha}^{\mathrm{sub}}) \le \mathbb{P}(\mathcal{E}^c) + 3\alpha\times\begin{cases}(1 + 2B_{\alpha,0}(B_{\alpha,0} - 1)\Delta_{n,\alpha}^2(1 + 2\Delta_{n,\alpha})^{B_{\alpha,0}}), &\mbox{if }\Delta = 0,\\
(1 + 2|\Delta_{n,\alpha} - \Delta|)^{B_{\alpha,\Delta}}, &\mbox{if }\Delta \neq 0\end{cases}.
\]
If $\Delta = 0$, this only implies approximately $3\alpha$ miscoverage probability. If $\widehat{\Delta} - \Delta$ converges to zero in probability, then $\tau_{\alpha,\widehat{\Delta}} - \tau_{\alpha,\Delta}$ converges to zero and we obtain asymptotically $\alpha$ miscoverage probability for $\widehat{\mathrm{CI}}_{\alpha}^{\mathrm{sub}}$. This is the aim in proving~\eqref{eq:random-choice-B-unknown-Delta}. For this, we cannot readily use inequality~\eqref{eq:fixed-choice-B-unknown-Delta} because $\tau_{\alpha,\widehat{\Delta}}$ depend on the same data as $\widehat{\mathrm{CI}}_{\alpha}^{(0)},$ and $\widehat{\mathrm{CI}}_{\alpha}^{(1)}$. In order to overcome this, we show $\tau_{\alpha,\widehat{\Delta}}$ is close to $\tau_{\alpha,\Delta}$ in a relative error sense. For this, it is not sufficient to know $B_{\alpha,\widehat{\Delta}}=B_{\alpha,\Delta}$. We need $|\widehat{\Delta} - \Delta|$ to be small. Consider the event
\[
\mathcal{E}_0 ~:=~ \{|\widehat{\Delta} - \Delta| \le \eta\},
\]
for some $\eta \in [0, C_{\alpha,\Delta}]$. Recall from Proposition~\ref{prop:approximate-to-exact-validity} implies that if $|\widehat{\Delta} - \Delta| \le C_{\alpha,\Delta}$, then $B_{\alpha,\widehat{\Delta}} = B_{\alpha,\Delta}$. Hence, if $\eta \le C_{\alpha,\Delta}$, then $\mathcal{E}_0\subseteq\mathcal{E}$ and $\mathbb{P}(\mathcal{E}^c) \le \mathbb{P}(\mathcal{E}_0^c)$. To control the difference between $\tau_{\alpha,\widehat{\Delta}}$ and $\tau_{\alpha,\Delta}$, we write $\tau_{\alpha,\widehat{\Delta}}$ as
\[
\tau_{\alpha,\widehat{\Delta}} = \frac{1 - \widehat{a}}{\widehat{b} - \widehat{a}},
\]
where $\widehat{a} = P(B_{\alpha,\Delta}; \widehat{\Delta})/\alpha$ and $\widehat{b} = P(B_{\alpha,\Delta} - 1; \widehat{\Delta})/\alpha$. Similarly, we can write $\tau_{\alpha,\Delta} = (1 - a)/(b - a)$ for $a, b$ defined similar to $\widehat{a}, \widehat{b}$ with $\Delta$ replacing $\widehat{\Delta}$. From inequalities~\eqref{eq:main-inequality-P-B-Delta} and~\eqref{eq:Delta-zero-inequality}, we get that
\begin{align*}
(1 + 2|\widehat{\Delta} - \Delta|)^{-B_{\alpha,\Delta}} &\le \frac{\widehat{a}}{a} = \frac{P(B_{\alpha,\Delta}; \widehat{\Delta})}{P(B_{\alpha,\Delta}; \Delta)} \le (1 + 2|\widehat{\Delta} - \Delta|)^{B_{\alpha,\Delta}},\quad\mbox{if }\Delta\neq 0,\\
1 &\le \frac{\widehat{a}}{a} = \frac{P(B_{\alpha,\Delta}; \widehat{\Delta})}{P(B_{\alpha,\Delta}; \Delta)} \le (1 + 2B_{\alpha,0}(B_{\alpha,0} - 1)\widehat{\Delta}^2(1 + 2\widehat{\Delta})^{B_{\alpha,0}}),\quad\mbox{if }\Delta = 0.
\end{align*}
The same inequalities also hold true for $\widehat{b}/b$. For notational convenience, let us write $\ell \le \widehat{a}/a \le u$ and the same for $\widehat{b}/b$. It is easy to verify that $\tau_{\alpha,\widehat{\Delta}}$ is a decreasing function of $\widehat{a}$ and $\widehat{b}$ and hence
\[
\frac{1/u - a}{(b-a)} \le \tau_{\alpha,\widehat{\Delta}} \le \frac{1/\ell - a}{(b-a)}\quad\Rightarrow\quad \frac{1 - 1/u}{b - a} \le \tau_{\alpha,\widehat{\Delta}} - \tau_{\alpha,\Delta} \le \frac{1/\ell - 1}{b - a}.
\]
Recall that $b = P(B_{\alpha,\Delta} - 1; \Delta)/\alpha$ and $a = P(B_{\alpha,\Delta}; \Delta)/\alpha$. It follows that
\[
\left(1/2 - \Delta\right) \le \frac{P(B_{\alpha,\Delta} - 1; \Delta)}{\alpha}\left(\frac{1}{2} - \Delta\right) ~\le~ b - a ~\le~ \frac{P(B_{\alpha,\Delta}; \Delta)}{\alpha} \le 1.
\]
This yields
\[
(1 - 1/u) \le \tau_{\alpha,\widehat{\Delta}} - \tau_{\alpha,\Delta} \le \frac{(1/\ell - 1)}{(1/2 - \Delta)}.
\]
Hence, on event $\mathcal{E}_0$,
\begin{equation}\label{eq:tau-difference-inequality}
|\tau_{\alpha,\widehat{\Delta}} - \tau_{\alpha,\Delta}| ~\le~ D_{\tau} ~:=~ \begin{cases}2B_{\alpha,0}(B_{\alpha,0} - 1)\eta^2(1 + 2\eta)^{B_{\alpha,0}}, &\mbox{if }\Delta = 0,\\
(1/2 - \Delta)^{-1}|(1 + 2\eta)^{B_{\alpha,\Delta}} - 1|, &\mbox{if }\Delta \neq 0.
\end{cases}
\end{equation}
Note that $D_{\tau}$ is non-random and only depends on $\eta$ in the event $\mathcal{E}_0$. Inequality~\eqref{eq:tau-difference-inequality} can be alternatively written as $|\tau_{\alpha,\widehat{\Delta}} - \tau_{\alpha,\Delta}|\mathbbm{1}\{\mathcal{E}_0\} \le D_{\tau}$.

From~\eqref{eq:miscoverage-subsampling} and the fact that $\widehat{\mathrm{CI}}_{\alpha}^{(j)} = \widetilde{\mathrm{CI}}_{\alpha}^{(j)}$ for $j = 0, 1$ on the event $\mathcal{E}_0$ (when $\eta \le C_{\alpha,\Delta}$), we get
\begin{equation}\label{eq:decomposition-subsampling-miscoverage-exact}
\begin{split}
\mathbb{E}_U[\mathbbm{1}\{\theta_0 \notin \widehat{\mathrm{CI}}_{\alpha}^{\mathrm{sub}}\}]\mathbbm{1}\{\mathcal{E}_0\} &= \tau_{\alpha,\widehat{\Delta}}\mathbbm{1}\{\theta_0 \notin \widetilde{\mathrm{CI}}_{\alpha}^{(0)}\}\mathbbm{1}\{\mathcal{E}_0\} + (1 - \tau_{\alpha,\widehat{\Delta}})\mathbbm{1}\{\theta_0 \notin \widetilde{\mathrm{CI}}_{\alpha}^{(1)}\}\mathbbm{1}\{\mathcal{E}_0\}\\
&= \tau_{\alpha,\Delta}\mathbbm{1}\{\theta_0 \notin \widetilde{\mathrm{CI}}_{\alpha}^{(0)}\}\mathbbm{1}\{\mathcal{E}_0\} + (1 - \tau_{\alpha,\Delta})\mathbbm{1}\{\theta_0 \notin \widetilde{\mathrm{CI}}_{\alpha}^{(1)}\}\mathbbm{1}\{\mathcal{E}_0\}\\
&\quad+ (\tau_{\alpha,\widehat{\Delta}} - \tau_{\alpha,\Delta})\mathbbm{1}\{\theta_0 \notin \widetilde{\mathrm{CI}}_{\alpha}^{(0)}\}\mathbbm{1}\{\mathcal{E}_0\} - (\tau_{\alpha,\widehat{\Delta}} - \tau_{\alpha,\Delta})\mathbbm{1}\{\theta_0 \notin \widetilde{\mathrm{CI}}_{\alpha}^{(1)}\}\mathbbm{1}\{\mathcal{E}_0\}\\
&= \tau_{\alpha,\Delta}\mathbbm{1}\{\theta_0 \notin \widetilde{\mathrm{CI}}_{\alpha}^{(0)}\} + (1 - \tau_{\alpha,\Delta})\mathbbm{1}\{\theta_0 \notin \widetilde{\mathrm{CI}}_{\alpha}^{(1)}\}\\
&\quad- \tau_{\alpha,\Delta}\mathbbm{1}\{\theta_0 \notin \widetilde{\mathrm{CI}}_{\alpha}^{(0)}\}\mathbbm{1}\{\mathcal{E}_0^c\} - (1 - \tau_{\alpha,\Delta})\mathbbm{1}\{\theta_0 \notin \widetilde{\mathrm{CI}}_{\alpha}^{(1)}\}\mathbbm{1}\{\mathcal{E}_0^c\}\\
&\quad+ (\tau_{\alpha,\widehat{\Delta}} - \tau_{\alpha,\Delta})\mathbbm{1}\{\theta_0 \notin \widetilde{\mathrm{CI}}_{\alpha}^{(0)}\}\mathbbm{1}\{\mathcal{E}_0\} - (\tau_{\alpha,\widehat{\Delta}} - \tau_{\alpha,\Delta})\mathbbm{1}\{\theta_0 \notin \widetilde{\mathrm{CI}}_{\alpha}^{(1)}\}\mathbbm{1}\{\mathcal{E}_0\}.
\end{split}
\end{equation}
Because 
\[
\mathbb{P}(\theta_0 \notin \widehat{\mathrm{CI}}_{\alpha,\Delta}) = \mathbb{E}[\tau_{\alpha,\Delta}\mathbbm{1}\{\theta_0 \notin \widetilde{\mathrm{CI}}_{\alpha}^{(0)}\} + (1 - \tau_{\alpha,\Delta})\mathbbm{1}\{\theta_0 \notin \widetilde{\mathrm{CI}}_{\alpha}^{(1)}\}],
\]
it follows from~\eqref{eq:decomposition-subsampling-miscoverage-exact} and~\eqref{eq:tau-difference-inequality} that
\begin{equation}\label{eq:final-inequality-subsampling-exact-interval}
\left|\mathbb{P}(\theta_0 \notin \widehat{\mathrm{CI}}_{\alpha}^{\mathrm{sub}}) - \mathbb{P}(\theta_0 \notin \widehat{\mathrm{CI}}_{\alpha,\Delta})\right| \le 2\mathbb{P}(\mathcal{E}_0^c) + D_{\tau}\mathbb{E}[\mathbbm{1}\{\theta_0\notin\widetilde{\mathrm{CI}}_{\alpha}^{(0)}\} + \mathbbm{1}\{\theta_0\notin\widetilde{\mathrm{CI}}_{\alpha}^{(1)}\}].
\end{equation} 
The second reminder term in~\eqref{eq:final-inequality-subsampling-exact-interval} is controlled using the definition of $D_{\tau}$ in~\eqref{eq:tau-difference-inequality} and the bounds~\eqref{eq:miscoverage-B-minus-1-subsampling},~\eqref{eq:miscoverage-B-subsampling} for $\mathbb{P}(\theta_0\notin\widetilde{\mathrm{CI}}_{\alpha}^{(0)})$ and $\mathbb{P}(\theta_0 \notin\widetilde{\mathrm{CI}}_{\alpha}^{(1)})$. Finally, we use the fact that $2\eta B_{\alpha,\Delta} \le 2C_{\alpha,\Delta}B_{\alpha,\Delta} \le 1/2$ (from $\eta \le C_{\alpha,\Delta}$ and Figure~\ref{fig:right-hand-side-of-consistency-subsampling-scaled}) to bound $D_{\tau}$ for $\Delta \neq 0$ as 
\[
(1 + 2\eta)^{B_{\alpha,\Delta}} - 1 \le e^{2\eta B_{\alpha,\Delta}} - 1 \le 2\eta B_{\alpha,\Delta}e^{2\eta B_{\alpha,\Delta}} \le 2\sqrt{e}\eta B_{\alpha,\Delta}.
\]
This completes the proof of~\eqref{eq:random-choice-B-unknown-Delta}. 
\section{Proof of Lemma~\ref{lem:subsampling-delta-control}}\label{appsec:subsampling-delta-control}
Set $J(x) = \mathbb{P}(W \le x)$ and $J_b(x) = \mathbb{P}(r_b(\widehat{\theta}_b - \theta_0) \le x)$.
Note that by the triangle inequality, $|\widehat{\Delta}_n - \Delta| \leq |L_n(0) - J(0)|$, so we obtain that,
\begin{align*}
|\widehat{\Delta}_n - \Delta| &\leq \left|\frac{1}{K_n}\sum_{j=1}^{K_n} \mathbbm{1}\{r_b(\widehat{\theta}_b^{(j)} - \widehat{\theta}_n) \le 0\} - J(0)\right|\\
&\le \left|\frac{1}{K_n}\sum_{j=1}^{K_n} \mathbbm{1}\{r_b(\widehat{\theta}_b^{(j)} - \widehat{\theta}_n) \le 0\} - \frac{1}{\binom{n}{b}}\sum_{s=1}^{\binom{n}{b}} \mathbbm{1}\{r_b(\widehat{\theta}_b^{(j)} - \widehat{\theta}_n) \le 0\}\right|
\\
&\quad
+ \left|\frac{1}{\binom{n}{b}}\sum_{j=1}^{\binom{n}{b}} \mathbbm{1}\{r_b(\widehat{\theta}_b^{(j)} - \widehat{\theta}_n) \le 0\} - J(0)\right|\\
&\le \left|\frac{1}{K_n}\sum_{j=1}^{K_n} \mathbbm{1}\{r_b(\widehat{\theta}_b^{(j)} - \widehat{\theta}_n) \le 0\} - \frac{1}{\binom{n}{b}}\sum_{j=1}^{\binom{n}{b}} \mathbbm{1}\{r_b(\widehat{\theta}_b^{(j)} - \widehat{\theta}_n) \le 0\}\right|\\
&\quad+\left|\frac{1}{\binom{n}{b}}\sum_{j=1}^{\binom{n}{b}} \mathbbm{1}\{r_b(\widehat{\theta}_b^{(j)} - \theta_0) \le r_b(\widehat{\theta} - \theta_0)\} - J(0)\right|.
\end{align*}
Fix $t > 0$ such that $r_bt/r_n \le r^*$. Define the event 
\[
\mathcal{E} := \{r_b(\widehat{\theta}_n - \theta_0) \le r_bt/r_n\}.
\]
On the event $\mathcal{E}$, we have
\begin{equation}
\begin{split}
&\left|\frac{1}{\binom{n}{b}}\sum_{j=1}^{\binom{n}{b}} \mathbbm{1}\{r_b(\widehat{\theta}_b^{(j)} - \theta_0) \le r_b(\widehat{\theta} - \theta_0)\} - J(0)\right|\\
&\quad\le \left|\frac{1}{\binom{n}{b}}\sum_{j=1}^{\binom{n}{b}} \mathbbm{1}\{r_b(\widehat{\theta}_b^{(j)} - \theta_0) \le r_bt/r_n\} - J(0)\right|\\
&\qquad+ \left|\frac{1}{\binom{n}{b}}\sum_{j=1}^{\binom{n}{b}} \mathbbm{1}\{r_b(\widehat{\theta}_b^{(j)} - \theta_0) \le -r_bt/r_n\} - J(0)\right|\\
&\quad\le \left|\frac{1}{\binom{n}{b}}\sum_{j=1}^{\binom{n}{b}} \mathbbm{1}\{r_b(\widehat{\theta}_b^{(j)} - \theta_0) \le r_bt/r_n\} - J_b(r_bt/r_n)\right| + \left|J_b(r_bt/r_n) - J(r_bt/r_n)\right| + \left|J(r_bt/r_n) - J(0)\right|\\
&\qquad+ \left|\frac{1}{\binom{n}{b}}\sum_{j=1}^{\binom{n}{b}} \mathbbm{1}\{r_b(\widehat{\theta}_b^{(j)} - \theta_0) \le -r_bt/r_n\} - J_b(-r_bt/r_n)\right| + \left|J_b(-r_bt/r_n) - J(-r_bt/r_n)\right| + \left|J(-r_bt/r_n) - J(0)\right|.
\end{split}
\end{equation}
Because $r_bt/r_n \le r^*$, assumption~\ref{eq:continuity-distribution} implies that
\[
\max\{|J(r_bt/r_n) - J(0)|, |J(-r_bt/r_n) - J(0)|\} \le \mathfrak{C}r_bt/r_n.
\]
From assumption~\ref{eq:limiting-distribution}, we conclude
\[
\max\{|J_b(r_bt/r_n) - J(r_bt/r_n)|, |J_b(-r_bt/r_n) - J(-r_bt/r_n)|\} \le \delta_b.
\]
Therefore, on the event $\mathcal{E}$,
\begin{equation}
\begin{split}
|\widehat{\Delta}_n - \Delta| &\le \left|\frac{1}{K_n}\sum_{j=1}^{K_n} \mathbbm{1}\{r_b(\widehat{\theta}_b^{(j)} - \widehat{\theta}_n) \le 0\} - \frac{1}{\binom{n}{b}}\sum_{j=1}^{\binom{n}{b}} \mathbbm{1}\{r_b(\widehat{\theta}_b^{(j)} - \widehat{\theta}_n) \le 0\}\right|\\
&\quad+ \left|\frac{1}{\binom{n}{b}}\sum_{j=1}^{\binom{n}{b}} \mathbbm{1}\{r_b(\widehat{\theta}_b^{(j)} - \theta_0) \le r_bt/r_n\} - J_b(r_bt/r_n)\right|\\
&\quad+ \left|\frac{1}{\binom{n}{b}}\sum_{j=1}^{\binom{n}{b}} \mathbbm{1}\{r_b(\widehat{\theta}_b^{(j)} - \theta_0) \le -r_bt/r_n\} - J_b(-r_bt/r_n)\right|\\
&\quad+ 2\delta_b + 2\mathfrak{C}r_bt/r_n.
\end{split}
\end{equation}
Observe that $\widehat{\theta}_b^{(j)}, 1\le j\le K_n$ are independent and identically distributed random variables conditional on the data drawn from the finite population $\widehat{\theta}_b^{(j)}, 1\le j\le \binom{n}{b}$. Corollary 1 of~\cite{massart1990tight} implies that
\[
\mathbb{P}\left(\left|\frac{1}{K_n}\sum_{j=1}^{K_n} \mathbbm{1}\{r_b(\widehat{\theta}_b^{(j)} - \widehat{\theta}_n) \le 0\} - \frac{1}{\binom{n}{b}}\sum_{j=1}^{\binom{n}{b}} \mathbbm{1}\{r_b(\widehat{\theta}_b^{(j)} - \widehat{\theta}_n) \le 0\}\right| \ge \sqrt{\frac{\log(2n)}{2K_n}}\right) \le \frac{1}{n}.
\]
Furthermore, note that $\binom{n}{b}\sum_{j=1}^{\binom{n}{b}}\mathbbm{1}\{r_b(\widehat{\theta}_b^{(j)} - \theta_0) \le -r_bt/r_n\}$ is a non-degenerate $U$-statistics of order $b$ and with a kernel bounded between $0$ and $1$. Hence, Hoeffding's inequality for $U$-statistics \citep[inequality (5.7)]{hoeffding1963probability} implies that
\[
\mathbb{P}\left(\left|\frac{1}{\binom{n}{b}}\sum_{j=1}^{\binom{n}{b}} \mathbbm{1}\{r_b(\widehat{\theta}_b^{(j)} - \theta_0) \le r_bt/r_n\} - J_b(r_bt/r_n)\right| \ge \sqrt{\frac{\log(2n/b)}{2[n/b]}}\right) \le \frac{b}{n}.
\] 
Hence, with probability at least $1 - \mathbb{P}(\mathcal{E}^c) - (b + 1)/n$,
\[
|\widehat{\Delta}_n - \Delta| \le \sqrt{\frac{\log(2n)}{2K_n}} + \sqrt{\frac{\log(2n/b)}{2[n/b]}} + 2\delta_b + 2\mathfrak{C}r_bt/r_n.
\]
Now, note that 
\begin{align*}
\mathbb{P}(\mathcal{E}^c) &= \mathbb{P}(r_b|\widehat{\theta}_n - \theta_0| > r_bt/r_n) \le 2\delta_n + \mathbb{P}(|W| > t).
\end{align*}
Therefore, with probability at least $1 - 2\delta_n - (b + 1)/n - \mathbb{P}(|W| > t)$,
\[
|\widehat{\Delta}_n - \Delta| \le \sqrt{\frac{\log(2n)}{2K_n}} + \sqrt{\frac{\log(2n/b)}{2[n/b]}} + 2\delta_b + 2\mathfrak{C}\frac{r_bt}{r_n}.
\]
\section{Proof of Theorem~\ref{thm:unimodal-CI-coverage}}\label{appsec:unimodal-CI-coverage}
Define
\[
\widehat{\theta}_{\max}^{(B)} ~:=~ \max_{1\le j\le B}\widehat{\theta}_j,\quad\mbox{and}\quad \widehat{\theta}_{\min}^{(B)} ~:=~ \min_{1\le j\le B}\widehat{\theta}_j.
\]
Similarly, define $\widehat{\theta}_{\max}^{(B-1)}$ and $\widehat{\theta}_{\min}^{(B-1)}$.
Set
\[
F_{n,j}(u) = \mathbb{P}(r_{n,\alpha}(\widehat{\theta}_j - \theta_0) \le u),\quad\mbox{and}\quad F_W(u) = \mathbb{P}(W \le u).
\]
The assumed hypothesis implies that $|F_{n,j}(u) - F_W(u)| \le \delta_{n,\alpha}$ for all $u$ and all $1\le j\le B$. Finally, define the miscoverage probability
\[
M_B := \mathbb{P}\left(\theta_0 \notin \left[\widehat{\theta}_{\max}^{(B)} - t(\widehat{\theta}_{\max}^{(B)} - \widehat{\theta}_{\min}^{(B)}), \widehat{\theta}_{\max}^{(B)} + t(\widehat{\theta}_{\max}^{(B)} - \widehat{\theta}_{\min}^{(B)})\right]\right).
\]
We will prove that for all $B\ge1$, $t\ge0$, and $\Delta\in[0, 1/2]$,
\begin{equation}\label{eq:required-inequality-miscoverage-unimodal}
\frac{M_B}{Q(B; t, \Delta)} \le \frac{1}{(1 - 10B(1 + t)\delta_{n,\alpha})_+}.
\end{equation}
The same bound holds true for $M_{B-1}/Q(B-1;t, \Delta)$. The definition of $\eta_{\alpha,t}$ implies that $$\eta_{\alpha,t}Q(B_{\alpha,t,\Delta}; t, \Delta) + (1 - \eta_{\alpha,t})Q(B_{\alpha,t,\Delta}-1; t, \Delta) = \alpha.$$ Combining this with the inequalities for $M_B$ and $M_{B-1},$ the result is proved.
Note that
\begin{align*}
M_B &= \mathbb{P}(\theta_0 < \widehat{\theta}_{\min}^{(B)} - t (\widehat{\theta}_{\max}^{(B)} - \widehat{\theta}_{\min}^{(B)})) + \mathbb{P}(\theta_0 > \widehat{\theta}_{\max}^{(B)} + t (\widehat{\theta}_{\max}^{(B)} - \widehat{\theta}_{\min}^{(B)}))\\
&= \mathbb{P}\left(r_{n,\alpha}(\widehat{\theta}_{\min}^{(B)} - \theta_0) > \frac{t}{1+t}  r_{n,\alpha}(\widehat{\theta}_{\max}^{(B)} - \theta_0)\right) + \mathbb{P}\left(r_{n,\alpha}(\widehat{\theta}_{\max}^{(B)} - \theta_0) <  \frac{t}{1+t} r_{n,\alpha}(\widehat{\theta}_{\min}^{(B)} - \theta_0)\right)\\
&= \mathbf{I}_B + \mathbf{II}_B.
\end{align*}
This implies that $M_B$ can be written in terms of the smallest and largest order statistic of $r_{n,\alpha}(\widehat{\theta}_j - \theta_0)$, $1\le j\le B$. Bounding $\mathbf{I}_B$ will also provide a bound for $\mathbf{II}_B$ by taking negative random variables $r_{n,\alpha}(\theta_0 - \widehat{\theta}_j)$. Under the assumption of continuous distribution for $r_{n,\alpha}(\widehat{\theta}_j - \theta_0)$, we get following the proof of~\citet[Theorem 1]{lanke1974interval} that
\begin{equation}\label{eq:exact-expression-I-B}
\mathbf{I}_B = \sum_{j=1}^{B} \int_0^{\infty} \prod_{i\neq j} (F_{n,i}(x) - F_{n,i}(tx/(1 + t)))dF_j(x) .
% for identical distributions, it is
% \int_{0}^{\infty} B(F_n(x) - F_n(tx/(1 + t)))^{B-1}dF_n(x).
\end{equation}
Recall that $F_{n,i}(x)$ and $F_W(x)$ are close and satisfy
\[
F_{n,i}(x) - F_{n,i}(tx/(1 + t)) \le F_W(x) - F_W(tx/(1 + t)) + 2\delta_{n,\alpha},
\]
and because the distribution of $W$ is unimodal at $0$, we get 
\[
F_{n,i}(x) - F_{n,i}(tx/(1 + t)) \le \frac{F_W(x) - F_W(0)}{1 + t} + 2\delta_{n,\alpha}.
\]
This follows from the fact that unimodality implies $F_W(\cdot)$ is convex below $0$ and concave above $0$ implying $F(\lambda x) \ge F(0) + \lambda(F(x) - F(0))$ for $\lambda\in[0, 1]$ and $x\ge0$. Finally, using the closeness of $F_{n,j}(\cdot)$ and $F_W(\cdot)$ once again, we conclude
\[
0 \le F_{n,i}(x) - F_{n,i}(tx/(1 + t)) \le \frac{F_{n,j}(x) - F_{n,j}(0)}{1 + t} + 4\delta_{n,\alpha}.
\]
Substituting this inequality in~\eqref{eq:exact-expression-I-B}, we obtain
\begin{align*}
\mathbf{I}_B &\le \sum_{j=1}^B \int_0^{\infty} \left(\frac{F_{n,j}(x) - F_{n,j}(0)}{1 + t} + 4\delta_{n,\alpha}\right)^{B-1}dF_{n,j}(x)\\
&= \frac{1}{(1 + t)^{B-1}}\sum_{j=1}^B\int_0^{\infty} \left(F_{n,j}(x) - F_{n,j}(0) + 4(1 + t)\delta_{n,\alpha}\right)^{B-1}dF_{n,j}(x)\\
&= \frac{1}{(1 + t)^{B-1}}\sum_{j=1}^B\int_{F_{n,j}(0)}^{1} \left(u - F_{n,j}(0) + 4(1 + t)\delta_{n,\alpha}\right)^{B-1}du\\ 
&\le \frac{1}{B(1 + t)^{B-1}}\sum_{j=1}^B \left[\left(1 - F_{n,j}(0) + 4(1 + t)\delta_{n,\alpha}\right)^{B}\right].
\end{align*}
Applying the same calculations with $r_{n,\alpha}(\theta_0 - \widehat{\theta}_j)$ which has the distribution function $G_n(t) = 1 - F_n(-t)$ would yield
\[
\mathbf{II}_B \le \frac{1}{B(1 + t)^{B-1}}\sum_{j=1}^B \left[\left(F_{n,j}(0) + 4(1 + t)\delta_{n,\alpha}\right)^{B}\right].
\]
Therefore,
\begin{equation}\label{eq:almost-final-inequality-M-B}
\begin{split}
M_B &\le \frac{1}{B(1 + t)^{B-1}}\sum_{j=1}^B \left[(1 - F_{n,j}(0) + 4(1 + t)\delta_{n,\alpha})^{B} + (F_{n,j}(0) + 4(1 + t)\delta_{n,\alpha})^{B}\right]\\
&\le \frac{1}{(1 + t)^{B-1}}\left[(1 - F_W(0) + 5(1 + t)\delta_{n,\alpha})^B + (F_W(0) + 5(1 + t)\delta_{n,\alpha})^B\right].
\end{split}
\end{equation}
The second inequality here follows again from the closeness of $F_{n,j}(0)$ and $F_W(0)$.
From the continuous distribution assumption, the asymptotic median bias is given by $\Delta = |1/2 - \mathbb{P}(W \le 0)|$. Hence, it follows that
\begin{equation}\label{eq:final-inequality-M-B}
M_B \le (1 + t)^{-B+1}\left[\left(\frac{1}{2} + 5(1 + t)\delta_{n,\alpha} - \Delta\right)^B + \left(\frac{1}{2} + 5(1 + t)\delta_{n,\alpha} + \Delta\right)^B\right].
\end{equation}
Now consider $M_B/Q(B; t, \Delta)$. 
\[
\frac{M_B}{Q(B; t, \Delta)} ~\le~ \frac{(1 - 2\Delta + 10(1 + t)\delta_{n,\alpha})^B + (1 + 2\Delta + 10(1 + t)\delta_{n,\alpha})^B}{(1 - 2\Delta)^B + (1 + 2\Delta)^B}.
\]
To bound the right hand side, consider the function $g(x) = (x + 1-2\Delta)^B + (x + 1 + 2\Delta)^B$ for $x \ge 0$. It is clear that
\[
0 ~\le~ g(x) - g(0) ~\le~ \int_0^x g'(t)dt ~\le~ Bx\left[(x + 1 - 2\Delta)^{B-1} + (x + 1 + 2\Delta)^{B-1}\right].
\]
Furthermore,
\[
\frac{(x + 1 - 2\Delta)^{B-1} + (x + 1 + 2\Delta)^{B-1}}{(x + 1 - 2\Delta)^{B} + (x + 1 + 2\Delta)^B} \le \frac{1}{x + 1}.
\]
Hence, we conclude that $g(x) \le g(0)/(1 - Bx/(x + 1))_+$ and
\[
\frac{M_B}{Q(B; t, \Delta)} ~\le~ \frac{1}{(1 - 10B(1 + t)\delta_{n,\alpha})_+}.
\]
This completes the proof of~\eqref{eq:required-inequality-miscoverage-unimodal} and implies~\eqref{eq:coverage-mode}.
% \end{proof}
\end{document}